\documentclass[11pt]{amsart}
\usepackage{amssymb}
\usepackage{amscd}
\usepackage{amsmath}
\usepackage[all]{xy}
\usepackage{braket}
\usepackage{mathrsfs}
\usepackage{setspace}
\usepackage{color}
\usepackage{bbm}
\usepackage{dsfont}
\usepackage[usenames,dvipsnames,svgnames,table]{xcolor}
\usepackage[utf8]{inputenc}
\usepackage[T1]{fontenc}
\usepackage{ tipa }
\usepackage{helvet}
\usepackage{stmaryrd}  
\usepackage{amsmath,amsxtra,amsthm,amssymb,xr,enumerate,fullpage,hyperref, comment}
\usepackage[all]{xy}
\usepackage[bbgreekl]{mathbbol}
\usepackage[top=26truemm,bottom=26truemm,left=25truemm,right=25truemm]{geometry}
\usepackage{enumitem}

\makeatletter
\newcommand{\mylabel}[2]{#2\def\@currentlabel{#2}\label{#1}}
\makeatother

\theoremstyle{plain}
\newtheorem{theorem}{Theorem}[section]
\newtheorem{proposition}[theorem]{Proposition}
\newtheorem{lemma}[theorem]{Lemma}
\newtheorem{corollary}[theorem]{Corollary}
\newtheorem{conjecture}[theorem]{Conjecture}
\newtheorem{question}[theorem]{Question}

\theoremstyle{remark}
\newtheorem{remark}[theorem]{Remark}

\theoremstyle{definition}
\newtheorem{definition}[theorem]{Definition}
\newtheorem{example}[theorem]{Example}

\newtheorem{convention}[theorem]{Convention}
\newtheorem*{acknowledgements}{Acknowledgements}

  1

\DeclareMathOperator{\Gal}{Gal}

\DeclareMathOperator{\Hom}{Hom}

\DeclareMathOperator{\res}{res}

\newcommand{\bC}{\mathbb{C}}

\newcommand{\bG}{\mathbb{G}}
\newcommand{\bfG}{\mathbf{G}}
\newcommand{\bI}{\mathbb{I}}
\newcommand{\bQ}{\mathbb{Q}}
\newcommand{\bL}{\mathbb{L}}
\newcommand{\bR}{\mathbb{R}}
\newcommand{\bT}{\mathbb{T}}
\newcommand{\bZ}{\mathbb{Z}}
\newcommand{\QQ}{\mathbb{Q}}

\newcommand{\cH}{\mathcal{H}}

\newcommand{\cL}{\mathcal{L}}

\newcommand{\cO}{\mathcal{O}}

\newcommand{\cR}{\mathcal{R}}

\newcommand{\fa}{\mathfrak{a}}

\newcommand{\fq}{\mathfrak{q}}
\newcommand{\fp}{\mathfrak{p}}

\newcommand{\fm}{\mathfrak{m}}

\newcommand{\TT}{\mathbb{T}}

\newcommand{\ZZ}{\mathbb{Z}}
\newcommand{\Z}{\mathbb{Z}}

\newcommand{\g}{\mathbf{g}}
\newcommand{\f}{\mathbf{f}}

\newcommand{\sA}{\mathscr{A}}
\newcommand{\sC}{\mathscr{C}}

\newcommand{\Ann}{\mathrm{Ann}}

\newcommand{\LL}{\Lambda}

\newcommand{\ord}{\mathrm{ord}}

\newcommand{\cyc}{\mathrm{cyc}}

\newcommand{\lra}{\longrightarrow}

\newcommand{\xra}{\xrightarrow}
\newcommand{\be}{\begin{equation}}
\newcommand{\ee}{\end{equation}}

\makeatletter
\@namedef{subjclassname@2020}{\textup{2020} Mathematics Subject Classification}
\makeatother

\begin{document}

\title[]{{$\cL$-}\lowercase{invariants, $p$-adic heights and factorization of $p$-adic}  $L$-\lowercase{functions}}

\author[K. B\"uy\"ukboduk]{K\^az\i m B\"uy\"ukboduk}
\address[B\"uy\"ukboduk]{UCD School of Mathematics and Statistics\\ University College Dublin\\ Ireland}
\author[R. Sakamoto]{Ryotaro Sakamoto}
\address[Sakamoto]{Center for Advanced Intelligence Project\\ RIKEN\\ Japan}

\email{kazim.buyukboduk@ucd.ie}
\email{ryotaro.sakamoto@riken.jp}

\keywords{$p$-adic heights, $p$-adic $L$-functions, exceptional zeros, Rankin--Selberg products, Hilbert Modular Forms}
\subjclass[2020]{11R23; 11F41}

\begin{abstract}
We continue with our study of the non-critical exceptional zeros of Katz' $p$-adic $L$-functions attached to a CM field $K$, following two threads. 
In the first thread, we redefine our (group-ring-valued) $\cL$-invariant associated to each $\ZZ_p$-extension $K_\Gamma$ of $K$ in terms of $p$-adic height pairings and interpolate them as $K_\Gamma$ varies to a \emph{universal} (multivariate)  group-ring-valued $\cL$-invariant. 
In the second thread, we use our results to study the exceptional zeros of the Rankin--Selberg $p$-adic $L$-functions at non-critical specializations of the self-products of nearly ordinary CM families, via the factorization statements we establish. The factorization theorems are extensions of the  results due to Greenberg and Palvannan.
\end{abstract}

\maketitle

\tableofcontents

\section{Introduction}
Fix forever a prime $p>2$, an isomorphism $\iota: \bC \xra{\sim} \bC_p$ as well embeddings $\iota_\infty:\overline{\QQ} \hookrightarrow \bC$ and $\iota_p:\overline{\QQ} \hookrightarrow \bC_p$ such that $\iota \circ\iota_\infty=\iota_p$. Let $K$ be a CM field and let $F$ denote its maximal totally real subfield. We denote by $c \in \Gal(K/F)$ the unique non-trivial $F$-automorphism of $K$. 
We assume throughout this article that the $p$-ordinary hypothesis that
\begin{enumerate}
    \item[\mylabel{item_ord}{{\bf (ord)}}] every prime of $F$ above $p$ splits in $K$
\end{enumerate} 
holds true.

One may then choose a $p$-ordinary CM type $\Sigma$ and with this data at hand, the work of Katz~\cite{Katz78} (see also \cite{HT93}, Theorem~II) equips us with the $p$-adic $L$-function $L_{p,\chi}^{\rm Katz}$ attached to a ray class character $\chi$ of prime-to-$p$ conductor and order. Let us set 
\[
e := \# \{v \in \Sigma^{c} \mid \chi(G_{K_{v}}) = \{1\}\}, 
\]
where $G_{K_{v}}$ denotes the absolute Galois group of the local field $K_{v}$. 
Take an arbitrary $\bZ_p$-extension of $K$ with Galois group $\Gamma$. By restricting the domain of $L_{p,\chi}^{\rm Katz}$ to $\Gamma$, one also has the one-variable ($p$-adic analytic) function $L_{p,\chi}^{\rm Katz}|_{\Gamma}$.

Since $1 - \chi(\mathrm{Frob}_v) = 0$ for any prime $v \in \Sigma^{c}$ with $\chi(G_{K_{v}}) = \{1\}$, one is lead to expect that the Katz $p$-adic $L$-function $L_{p,\chi}^{\rm Katz}|_{\Gamma}$ has an exceptional zero of order $e$ at the trivial character $\mathds{1}$  whenever $e>0$. This is not obvious a priori, since the defining interpolative property  for $L_{p,\chi}^{\rm Katz}$ (c.f. \eqref{Eqn_Katz_padic_interpolation}) has no bearing for its value at the trivial character.  
The 
arithmetic properties of $L_{p,\chi}^{\rm Katz}|_{\Gamma}$ at $\mathds{1}$ are expected to be governed by Perrin-Riou's $p$-adic Beilinson conjectures, and the exceptional zeros in this context are relative to her conjectural leading term formulae.

In our previous work~\cite[Theorem 1.1]{BS19}, we studied the exceptional zeros of 
$L_{p,\chi}^{\rm Katz}|_{\Gamma}$
at 
$\mathds{1}$ and gave a formula for its $e^{\rm th}$ derivative at $\mathds{1}$ (which, in the notation of op. cit., is the image of $L_{p,\chi}^{\rm Katz}|_{\Gamma}$ inside $\sA_\Gamma^e/\sA_{\Gamma}^{e+1}$)
in terms of the (Galois) group-ring-valued $\cL$-invariants  we have introduced therein. A consequence of the leading term formula \cite[Theorem 1.1]{BS19} is the calculation of the order of vanishing of $L_{p,\chi}^{\rm Katz}$ at $\mathds{1}$ (see Corollary 1.4 in op. cit.). Even though our results were unconditional in the case when $K$ is an imaginary quadratic field, the proof of \cite[Theorem 1.1]{BS19} rests upon a number of assumptions and the validity of a list of plausible but difficult conjectures in the general case. 
This list includes the Rubin--Stark conjectures for abelian extensions of $K$ and an Iwasawa theoretic variant of the Mazur--Rubin--Sano conjecture.

One of our objectives is to eliminate these hypotheses on \cite[Corollary 1.4]{BS19}. We summarize our main results in the present article in the paragraphs \S\ref{subsubsec_1_1_1}--\S\ref{subsubsec_intro_5}.

\subsection{Synopsis of results}
\label{subsec_summary}
In what follows, $\cO$ denotes a sufficiently large finite flat extension of $\ZZ_p$ (e.g., it contains the values of all ray class characters that make an appearance in \S\ref{subsec_summary}). We set $\Gamma_\infty:=\Gal(K(p^\infty)/K)$ where $K(p^\infty)$ is the compositum of all $\bZ_p$-extensions of $K$. We let  $\sA\subset \cO[[\Gamma_\infty]]$ denote the full augmentation ideal. Given $P\in  \cO[[\Gamma_\infty]]$,  we call the largest natural number $m$ so that $P\in \mathscr{A}^m$ the \emph{total order of vanishing} of $P$ at $\mathds{1}$ and denote it by ${\rm ord}_{\sA}P$. For a natural number $m \leq {\rm ord}_{\sA}P$, the image of $P$ inside the quotient $\sA^{m}/\sA^{m+1}$ is called the $m^{\rm th}$ \emph{total derivative} of $P$ at $\mathds{1}$.


\subsubsection{Support of the algebraic $p$-adic $L$-function at the full augmentation ideal}
\label{subsubsec_1_1_1}
We  give an alternative proof of \cite[Corollary 1.4]{BS19} under less stringent hypotheses, but still assuming the validity of the $\Sigma$-Leopoldt conjecture and the Iwasawa Main Conjecture for the CM field $K$. We remark that the $\Sigma$-Leopoldt conjecture is (strictly, in general) stronger than the Leopoldt conjecture. The proof of either of these conjectures in general appears to require a strong input from transcendental number theory; for example, the $\Sigma$-Leopoldt conjecture follows as a consequence of the $p$-adic Schanuel conjecture (c.f.  \cite{HT94}, Lemma 1.2.1). Our alternative proof proceeds through a study of the algebraic Katz $p$-adic $L$-function $L_{p,\chi}^{\rm alg}$, where we exploit homological properties of underlying Selmer complexes. This is carried out in \S \ref{sec_Selmer_Groups}--\S\ref{sec_Linvariants_heights_exceptionalzeros} and our main  results concerning the total order of vanishing of $L_{p,\chi}^{\rm alg}$ at $\mathds{1}$ are the following:

\begin{theorem}[Proposition \ref{proposition_vanishing_order_e}]\label{theorem:mult-exp-zeros_lower}
We 
have 
$
L_{p,\chi}^{\rm alg}  \in \sA^e\,.
$
\end{theorem}

\begin{theorem}[Corollary~\ref{cor_algebraic_multivariate_exceptional_zero}]\label{theorem:mult-exp-zeros}
If the $\Sigma$-Leopoldt Conjecture~\ref{conj_sigma_leopoldt} for $\overline{K}^{\ker(\chi)}/K$ holds true, then 
\[
 L_{p,\chi}^{\rm alg}  \in \sA^e \setminus \sA^{e+1}.
\]
\end{theorem}
It is worth mentioning that the lower bound on the total order of vanishing of $L_{p,\chi}^{\rm alg}$ (Theorem~\ref{theorem:mult-exp-zeros_lower}) is unconditional. In order to obtain the upper bound, one needs to prove that the $e^{\rm th}$ derivative of the algebraic Katz $p$-adic $L$-function $L^{\rm alg}_{p,\chi}$ along a certain $\ZZ_p$-extension of $K$ is non-zero. We do so by choosing a particular $\ZZ_p$-extension unramified outside $\Sigma$, along which we can compute our group-ring-valued $\cL$-invariant directly; c.f. Lemma~\ref{lemma:non-vanishing} and Theorem \ref{thm_nonvanishingL_semisimple_nondegpadicheight}. 

\begin{remark}\label{remark:totally_real_excep}
Suppose $\eta \colon G_F \to \overline{\bQ}^\times$ is a totally odd character of finite order other than the inverse $\omega_F^{-1}$ of the Teichm\"uller character $\omega_F$. We let $L^{\textrm{alg,DR}}_{\eta\omega_F}$ denote the algebraic Deligne--Ribet $p$-adic $L$-function in single (cyclotomic) variable attached to the totally even character $\eta\omega_F$.  Federer and Gross showed in \cite[Proposition 3.10]{FedererGross81} that 
the order of vanishing of $L^{\textrm{alg,DR}}_{\eta\omega_F}$ at 
$\mathds{1}$  is at least 
\[
\# \{v \in S_p(F) \mid \eta(G_{K_v}) = \{1\} \}. 
\]

\end{remark}

In view of the Iwasawa Main Conjectures, Theorem \ref{theorem:mult-exp-zeros_lower} can be thought of as an evidence towards the expectation that $L_{p,\chi}^{\rm Katz}|_{\Gamma}$ has an exceptional zero at $\mathds{1}$ of order at least $e$, for \emph{every} $\ZZ_p$-extension $K_\Gamma/K$ with Galois group $\Gamma$. In more precise terms, we have the following: 

\begin{corollary}\label{cor:katz-order}
Suppose that the Iwasawa Main Conjecture for $K$ holds true. 
\item[i)] $L_{p,\chi}^{\rm Katz}  \in \sA^e$. In particular, for any $\bZ_p$-extention of $K$ with Galois group $\Gamma$, the order of vanishing of $L_{p,\chi}^{\rm Katz}|_{\Gamma}$ at $\mathds{1}$ is at least $e$. 

\item[ii)] If, in addition, the $\Sigma$-Leopoldt Conjecture~\ref{conj_sigma_leopoldt} for $\overline{K}^{\ker(\chi)}/K$ holds true, then 
$L_{p,\chi}^{\rm Katz}  \in \sA^e \setminus \sA^{e+1}$. 
\end{corollary}

\begin{remark}
For a given $\ZZ_p$-extension $K_\Gamma/K$ with Galois group $\Gamma$, Theorem \ref{theorem:mult-exp-zeros} and Corollary \ref{cor:katz-order} do not give any information towards an upper bound on the order of vanishing of the $p$-adic $L$-function $L_{p,\chi}^{\rm Katz}|_{\Gamma}$  at $\mathds{1}$.
\end{remark}

\begin{remark}
We retain the notation of Remark~\ref{remark:totally_real_excep}. 
The discussion in Remark~\ref{remark:totally_real_excep}
combined with the Iwasawa Main Conjectures for the totally real field $F$ (proved by Wiles in \cite{Wiles90}) tell us that the order of vanishing of the Deligne--Ribet $p$-adic $L$-function $L^{\textrm{DR}}_{\eta\omega_F}$ (c.f. \cite{DeligneRibet} for its construction; see also \cite{Barsky, CassouNogues1, CassouNogues2} for an alternative construction of this $p$-adic $L$-function) at $\mathds{1}$  is at least 
\[
\# \{v \in S_p(F) \mid \eta(G_{K_v}) = \{1\} \}. 
\]
We further remark that Spiess in \cite{Spiess14} and Charollois--Dasgupta in \cite{CharolloisDasgupta14} gave a proof of this result without relying on the truth of the Iwasawa Main Conjectures for $F$. 
\end{remark}

\begin{remark}
\label{remark_intro_Hsieh_1}
Let us assume in this remark that $K=FM$ for an imaginary quadratic field $M$ where $p$ splits. Suppose also that the $p$-adic CM type $\Sigma$ is obtained by extending $\iota_p:M\hookrightarrow \bC_p$, and that $\overline{K}^{\ker(\chi)}$ is abelian over $M$. 

In this scenario, the $\Sigma$-Leopoldt Conjecture~\ref{conj_sigma_leopoldt} for ${\overline{K}^{\ker(\chi)}}/K$ follows from Brumer's $p$-adic analogue of Baker's theorem.  It therefore follows from Theorem~ \ref{theorem:mult-exp-zeros} that
\[
L_{p,\chi}^{\rm alg}  \in \sA^e \setminus \sA^{e+1}
\]
\emph{unconditionally}. Furthermore, if we in addition assume that $p\nmid 3h_K^- [F : \bQ]$, then Corollary \ref{cor:katz-order} combined with \cite[Theorem C.(iii)]{HsiehJAMS_IMC}  shows that 
\begin{align*}
     L_{p,\chi}^{\rm Katz}  \in \sA^e \setminus \sA^{e+1}. 
\end{align*}
\end{remark}

\subsubsection{Group-ring-valued $\cL$-invariants and $p$-adic heights} 
\label{subsubsec_group_ring_valued_L_invariants}  
For any $\ZZ_p$-extension $K_\Gamma$ of $K$ with Galois group $\Gamma$, we denote by $\sA_\Gamma \subset \cO[[\Gamma]]$ the augmentation ideal. We recall that in our previous work \cite{BS19}, we have defined (Galois) group-ring-valued $\cL$-invariants $\cL_{\Sigma,\Gamma}^{\rm Gal}  \in \sA_\Gamma^e/\sA_\Gamma^{e+1}$  associated to each $\ZZ_p$-extension $K_\Gamma/K$. We showed in op. cit. that our group-ring-valued $\cL$-invariant $\cL_{\Sigma,\Gamma}^{\rm Gal}$ naturally appears in the formulae for the $e^{\rm th}$ derivative of $L_{p,\chi}^{\rm Katz}|_{\Gamma}$  evaluated at $\mathds{1}$ (see  \cite{BS19}, Theorem 1.1) in the presence of exceptional zeros, hence entitling them to their name. 
 We also note that, when $K$ is an imaginary quadratic field, $K_\Gamma/K$ is the cyclotomic $\ZZ_p$-extension, and $e=1$, Greenberg has defined what we will refer to as the classical $\cL$-invariant. In this particular scenario, the value of our group-ring-valued $\cL$-invariant 
\begin{align*}
 \cL_{\Sigma,\Gamma}^{\rm Gal}\quad\in\quad &\quad \mathscr{A}_\Gamma/\mathscr{A}_\Gamma^2\quad\,\,\,\,\stackrel{\sim}{\longleftrightarrow}\,\, \Gamma\\
& \left(\gamma-1+ \mathscr{A}_\Gamma^2\right) \longleftrightarrow \gamma
\end{align*}
under the cyclotomic character coincides with Greenberg's classical $\cL$-invariant, 
 c.f.  \cite[Remark 1.5]{BS19}. 

As one of our results in this article, we explain in Proposition~\ref{proposition:equiv_non-vanish_non-degen} that the ``directional''   group-ring-valued $\cL$-invariant $\cL_{\Sigma,\Gamma}^{\rm Gal}$ 
can be defined in terms of a $\Gamma$-valued $p$-adic height pairing.  This allows us to relate the problem to determine the exact directional order of vanishing of $L_{p,\chi}^{\rm Katz}$ to the non-degeneracy of this $p$-adic height pairing. Similar results in the context of cyclotomic $\cL$-invariants are due to Benois~\cite[Theorem 1]{BenoisIwasawa2012}; see also \cite{kbbCMH} for results with similar flavour.

Moreover, relying on Nekov\'a\v{r}'s formalism, we interpolate $\cL_{\Sigma,\Gamma}^{\rm Gal}$ (as $\Gamma$ varies) to what we call the universal (multivariate) group-ring-valued $\cL$-invariant  $\cL_\Sigma^{\rm Gal} \in \sA^e/\sA^{e+1}$; this is the content of Definition~\ref{defn_multivariate_L_invariant}. We hope that this construction is of independent interest. We remark that in a closely related setting, Hida in~\cite{Hida2004} has studied similar $\cL$-invariants. It would be interesting to compare our construction with his.

\subsubsection{Rankin--Selberg $p$-adic $L$-functions}
\label{subsubsec_intro_3} One may recast our results concerning the exceptional zeros of the algebraic Katz $p$-adic $L$-function  $L_{p,\chi}^{\rm alg}$ in terms of the non-critical specialization $\widetilde{\mathscr{D}}^{\rm alg}=\widetilde{\mathscr{D}}^{\rm alg}(\Theta_{\Sigma,\psi}\otimes \Theta_{\Sigma,\psi})$ of the algebraic Rankin--Selberg $p$-adic $L$-function  attached to the self-Rankin--Selberg product $\Theta_{\Sigma,\psi}\otimes \Theta_{\Sigma,\psi}$ of the nearly ordinary CM family of Hilbert modular forms $\Theta_{\Sigma,\psi}$ with branch character $\psi$ (see \S\ref{sec_Hida_RS} for precise definitions); see the final five paragraphs of \S\ref{subsubsec_intro_5} for the motivation to pursue this line of thought. Before illustrating our results in this direction, let us comment on the significance of the adjective ``non-critical'': The Selmer groups that give rise to $\widetilde{\mathscr{D}}^{\rm alg}$ do not admit any specializations that are given via local conditions at $p$ which verify the Panchishkin condition. However, $\widetilde{\mathscr{D}}^{\rm alg}$ can be recovered, relying on the descriptions of the relevant Selmer complexes in the derived category (which, among other things, allow us to bound the projective dimensions of the appropriate Selmer groups), as a specialization of an algebraic $p$-adic $L$-function over a higher-dimensional weight space. The details pertaining to this discussion have been relegated to Appendix~\ref{sec_appendix}.

The relation between the exceptional zeros of $L_{p,\chi}^{\rm alg}$ and that of $\widetilde{\mathscr{D}}^{\rm alg}$ is established through the following factorization of $\widetilde{\mathscr{D}}^{\rm alg}$ into a product of algebraic Katz $p$-adic $L$-functions, which is the algebraic counterpart of one of the main results of \cite{HT93}. 
\begin{theorem}\label{intro_thm:iwasama main conj rankin-selberg}
The algebraic Rankin--Selberg $p$-adic $L$-function factors as
\[
 \widetilde{\mathscr{D}}^{\rm alg}
= {\rm ver}_{\rm +}(L_{p, \mathds{1}}^{\rm alg}\vert_{\Gamma_\infty^+})
\cdot 
 L_{p, \psi^{\rm ad}}^{\rm alg}
\]
where the equality takes place in ${\rm Frac}({\cO}[[\Gamma_\infty]])^\times \big{/}{\cO}[[\Gamma_\infty]]^\times$ and 
\begin{itemize}
\item ${\Gamma_\infty^+}$ denotes  the Galois group of the maximal $\bZ_p$-power extension of $F$ (note that ${\Gamma_\infty^+} = \Gamma_{\rm cyc}$ if the Leopoldt conjecture for $F$ is valid), 
\item $L_{p, \mathds{1}}^{\rm alg}\vert_{\Gamma_\infty^+}$ is the image of $L_{p, \mathds{1}}^{\rm alg}$ under the canonical map 
$${\rm Frac}({\cO}[[\Gamma_\infty]])^\times \big{/}{\cO}[[\Gamma_\infty]]^\times \longrightarrow {\rm Frac}({\cO}[[\Gamma_\infty^+]])^\times \big{/}{\cO}[[\Gamma_\infty^+]]^\times \cup \{0\},$$ 
\item $ {\rm ver}_{+} \colon {\rm Frac}({\cO}[[ \Gamma_\infty^+]]) \hookrightarrow {\rm Frac}({\cO}[[\Gamma_\infty]])$ is the transfer map given explicitly as in \eqref{eqn_cyc_transfer_map}. 
\end{itemize}
\end{theorem}
This statement corresponds to Theorem~\ref{thm:iwasama main conj rankin-selberg}(i) in the main text. 
\subsubsection{Factorization of algebraic Rankin--Selberg $p$-adic $L$-function}
\label{subsubsec_intro_4} Along the way, we also prove the following extension of Greenberg's factorization theorem in~\cite{Greenberg1981_factorization} (whose analytic counterpart is due to Gross) for CM fields:
\begin{theorem}[Theorem~\ref{thm_algebraic_gross_factorization} in the main text]
\label{intro_thm_algebraic_gross_factorization}
Let $\epsilon_{K/F}$ denote the quadratic character associated to $K/F$ and $\omega_F$ denote the Teichm\"uller character of $G_F$. Let $L_{\epsilon_{K/F}\omega_F}^{\rm alg, DR}$ denote the algebraic Deligne--Ribet $p$-adic $L$-function associated to the totally even character $\epsilon_{K/F}\omega_F$ and $\zeta_{\rm alg}$ the algebraic Dedekind $p$-adic zeta function, given as in Definition~\ref{defn_alg_DR_p-adicL}. We have the following equality taking place in ${\rm Frac}({\ZZ_p}[[\Gamma_{\rm cyc}]])^\times/{\ZZ_p}[[\Gamma_{\rm cyc}]]^\times$: 
$$  L_{p, \mathds{1}}^{\rm alg}\vert_{\Gamma_{\rm cyc}} = L_{\epsilon_{K/F}\omega_F}^{\rm alg, DR}\cdot
  \zeta_{\rm alg}\,.$$
\end{theorem}
Paralleling the discussion in Appendix~\ref{sec_appendix}, we remark that $L_{p, \mathds{1}}^{\rm alg}\vert_{\Gamma_{\rm cyc}}$ is a non-critical specialization of the algebraic Katz $p$-adic $L$-function $L_{p, \mathds{1}}^{\rm alg}$. The adjective ``non-critical'' refers to the fact that the local conditions that determine the Selmer group giving rise to $L_{p, \mathds{1}}^{\rm alg}\vert_{\Gamma_{\rm cyc}}$ fail the Panchishkin condition for any specialization. This reflects the fact that no cyclotomic character falls within the interpolation range of the Katz $p$-adic $L$-function $L_{p, \mathds{1}}^{\rm Katz}$.

Theorem~\ref{intro_thm_algebraic_gross_factorization} combined with Theorem~\ref{intro_thm:iwasama main conj rankin-selberg} and Theorem~\ref{theorem_adjoint} yields the following extension of the work of Palvannan~\cite{palvannan_factorization} (whose analytic counterpart was conjectured by Citro~\cite{Citro2008} and proved by Dasgupta in~\cite{Dasgupta2016Factorization}, Theorem 2) in the context of the symmetric-products of nearly ordinary families of Hilbert modular forms with CM: 

\begin{theorem}
\label{intro_thm_algebraic_dasgupta_factorization}
Let us put $\Gamma_{\infty}^{\circ} := \Gal(K_{\rm cyc}K_{\rm ac}/K)$, where $K_{\rm cyc}$ is the cyclotomic $\bZ_p$-extension of $K$ and $K_{\rm ac}$ is the maximal $\bZ_p$-power anticyclotomic extension of $K$. 
Let $\widetilde{\mathscr D}^{\rm alg}_{{\rm ad}^0\Theta_{\Sigma,\psi}} \in \cO[[\Gamma_\infty^\circ]]$ denote the algebraic adjoint $p$-adic $L$-function, given as in \S\ref{subsec_selmer_traceless_adj}. 
We then have the factorization
\[
\widetilde{\mathscr{D}}^{\rm alg} = {\rm ver}_{\rm cyc}(\zeta_{\rm alg})\cdot \widetilde{\mathscr D}^{\rm alg}_{{\rm ad}^0\Theta_{\Sigma,\psi}}, 
\]
where the equality takes place in ${\rm Frac}({\cO}[[\Gamma_\infty^\circ]])^\times \big{/}{\cO}[[\Gamma_\infty^\circ]]^\times$ and 
\[
 {\rm ver}_{\rm cyc} \colon \cO[[\Gamma_{\rm cyc}]]  \longrightarrow \cO[[\Gamma_\infty^\circ]]
\]
denotes the canonical homomorphism induced from the transfer map; c.f. Equation \eqref{defn_cyclo_verchiubung}.
\end{theorem}

Note that we have $\Gamma_\infty^\circ=\Gamma_\infty$ under the validity of the Leopoldt conjecture for $F$.  We also note that the $p$-adic analytic versions of these two factorization results  (recorded in \eqref{eqn_Gross_CM} and \eqref{eqn_Dasgupta_CM} in the main text)  seem currently out of reach when $F\neq \QQ$. 

\begin{remark} 
We have stated our factorization results for the non-critical specializations of Rankin--Selberg $p$-adic $L$-functions only in the context of CM families (respectively, of the cyclotomic restriction of the Katz $p$-adic $L$-function attached to the trivial character). The formalism in the present paper can be extended to treat the general case. Here the general case refers, in the context of  Rankin--Selberg $p$-adic $L$-functions, to the scenario with non-CM families; whereas in the context of Katz $p$-adic $L$-functions, to the scenario concerning the Katz $p$-adic $L$-functions attached to ray class characters $\chi$ for which we have $\chi=\chi\circ \,c$.  However, the factorization results for the non-critical specializations of Rankin--Selberg $p$-adic $L$-functions in the non-CM cases would require a suitable $p$-distinguished condition (paralleling Palvannan's work). Since our sights are set on applications towards the exceptional zeros of Katz $p$-adic $L$-functions, we have chosen not to include a discussion of the general case, so as not to digress from our main objective.
\end{remark}

\subsubsection{Exceptional zeros of Rankin--Selberg $p$-adic $L$-functions}
\label{subsubsec_intro_5} The factorization results Theorem~\ref{intro_thm:iwasama main conj rankin-selberg} and Theorem~\ref{intro_thm_algebraic_dasgupta_factorization} have the following consequence on the exceptional zeroes of the non-critical specializations of algebraic Rankin--Selberg $p$-adic $L$-functions: 

\begin{corollary}[Corollary~\ref{cor_exceptional_zeros_RankinSelberg_Hida} in the main text]
\label{intro_cor_exceptional_zeros_RankinSelberg_Hida}
Let $S_p(F)$ denote the set of $p$-adic places of $F$. If the $\Sigma$-Leopoldt Conjecture~\ref{conj_sigma_leopoldt} for $\overline{K}^{\ker(\psi)}$ holds true and Gross' $p$-adic regulator for $K$ given as in \cite[(3.8)]{FedererGross81} does not vanish, then 
\[
{ \widetilde{\mathscr D}^{\rm alg}_{{\rm ad}^0\Theta_{\Sigma,\psi}} } \in  (\mathscr{A}^\circ)^{e+ \#S_p(F)} \setminus (\mathscr{A}^\circ)^{e + \#S_p(F) + 1}. 
\]
Here $\mathscr{A}^\circ := \ker(\cO[[\Gamma_\infty^\circ]] \to \cO)$ denotes the argumentation ideal. If in addition Leopoldt's conjecture for $F$ holds true, then we have 
\[
\widetilde{\mathscr{D}}^{\rm alg}\in \mathscr{A}^{e+ \#S_p(F)-1} \setminus \mathscr{A}^{e+ \#S_p(F)}. 
\]
\end{corollary}
Note that we have $\Gamma_\infty^\circ=\Gamma_\infty$ and $\mathscr{A}^\circ = \mathscr{A}$ under the validity of the Leopoldt  conjecture for $F$. 
\begin{remark}
\label{rem_Hsieh_2_Intro}
Granted the truth of Iwasawa Main Conjectures for the CM field $K$, Corollary~\ref{intro_cor_exceptional_zeros_RankinSelberg_Hida} can be stated in terms of the non-critical specialization of Hida's ``analytic'' Rankin--Selberg $p$-adic $L$-functions $\widetilde{\mathscr D}^{\rm Hida}$ and  $\widetilde{\mathscr D}_{{\rm ad}^0\Theta_{\Sigma,\psi}}$, which are given in Definition~\ref{defn_regularized_Hida_RS}. The adjective ``non-critical'' in this context refers to the fact that the $p$-adic $L$-function $\widetilde{\mathscr D}^{\rm Hida}$ comes associated to a $p$-adic family of motives without any critical values, and it is defined as the specialization of a  Rankin--Selberg $p$-adic $L$-function over a higher-dimensional weight space.

To apply the main results of \cite{HsiehJAMS_IMC} towards the Iwasawa Main Conjectures, let us suppose $K=FM$ for an imaginary quadratic field $M$ where $p$ splits and $p\nmid 3h_K^-$. Suppose also that $\Sigma$ is obtained by extending $\iota_p:M\hookrightarrow \bC_p$,  that $\overline{K}^{\ker(\psi)}$ is abelian over $M$ and finally that $p\nmid [\overline{K}^{\ker(\psi)}:M]$. As we have noted in Remark~\ref{remark_intro_Hsieh_1}, both the $\Sigma$-Leopoldt and the Leopoldt conjecture holds true for $\overline{K}^{\ker(\psi)}$.  Corollary~\ref{intro_cor_exceptional_zeros_RankinSelberg_Hida} combined with \cite[Theorem C.(iii)]{HsiehJAMS_IMC}  shows that
$$\widetilde{\mathscr D}_{{\rm ad}^0\Theta_{\Sigma,\psi}} \in  \mathscr{A}^{e+ \#S_p(F)} \setminus \mathscr{A}^{e + \#S_p(F) + 1}\,\hbox{\,\,\,\,\,\,\,\,and\,\,\,\,\,\,\,\,\,}{ \widetilde{\mathscr{D}}^{\rm Hida}}\in \mathscr{A}^{e+ \#S_p(F)-1} \setminus \mathscr{A}^{e+ \#S_p(F)}$$
assuming the non-vanishing of the Gross' $p$-adic regulator. 
\end{remark}

We would like to conclude this introduction by indicating how the two seemingly independent threads (first is the portion summarized in \S\ref{subsubsec_1_1_1}--\S\ref{subsubsec_group_ring_valued_L_invariants}, whereas the second in \S\ref{subsubsec_intro_3}--\S\ref{subsubsec_intro_5}) are expected to tie together.

Since the work of Greenberg--Stevens~\cite{GreenbergStevens1993} on the Mazur--Tate--Teitelbaum conjecture, deformation theoretic arguments have been successfully employed in the study of exceptional zeros, c.f. \cite{BetinaDimitrovKatz} where the authors study the anticyclotomic exceptional zeros of the Katz' $p$-adic $L$-function (in the special case $F=\QQ$) in terms of the geometric properties of the cuspidal eigencurve about irregular weight-one points. See also \cite{kbbMTT,benoisbuyukboduk,Venerucci} for works where $p$-adic analytic families were put to use to study exceptional zeros.

The impetus to explore the exceptional zero phenomena from the perspective of Rankin--Selberg $p$-adic $L$-functions (which is the content of the portions of our work summarized in \S\ref{subsubsec_intro_3}--\S\ref{subsubsec_intro_5}) grew out of our attempt to understand the connections between the treatment of Rivero--Rotger \cite{RiveroRotgerJEMS} of the exceptional zero problem for Rankin--Selberg $p$-adic $L$-functions at weight-one points with the aid of Beilinson--Flach elements, the work of Betina--Dimitrov~\cite{BetinaDimitrovKatz} on the local geometry of the eigencurve and its relation to the  \emph{anticyclotomic} exceptional zeros of Katz' $p$-adic $L$-function, and our previous work~\cite{BS19}. We note that the restriction in \cite{BetinaDimitrovKatz} (as compared to the more general results in \cite{BS19} that concern any $\ZZ_p$-extension of $K$) is due to the nature of the approach in~\cite{BetinaDimitrovKatz}: The local geometry of the eigencurve is controlled by the adjoint $L$-functions, which admit (at points corresponding to eigenforms with CM) the $L$-functions of anticyclotomic Hecke characters as factors. 

We will amalgamate in a sequel~\cite{BS3} the ideas in \cite{RiveroRotgerJEMS,BetinaDimitrovKatz, BS19} and the results of the present article, to utilize Beilinson--Flach elements to further study the cyclotomic group-ring-valued $\cL$-invariant $\cL_{\Sigma,\Gamma_{\cyc}}^{\rm Gal}$, to give a non-vanishing criterion for it in terms of Beilinson--Flach elements. Before closing this section we expand on this final point. Until the end of \S\ref{subsubsec_intro_5}, we take $F=\QQ$.


The exceptional zeros that we study here concern outside the range of interpolation (equivalently, non-critical values of the $L$-function of the underlying motive). As a reflection of this fact, our group-ring-valued $\mathcal{L}$-invariants do not appear to be entirely local entities (e.g. unlike the ones introduced by Mazur--Tate--Teitelbaum) but rather carry a global flavor. 

In light of the factorization results for non-critical Rankin--Selberg $p$-adic $L$-functions (c.f. \cite[Theorem 8.1]{HT93} and Theorem~\ref{intro_thm:iwasama main conj rankin-selberg} above), we propose to use Beilinson--Flach elements as the required global input in our work \cite{BS3} in progress. To be more precise let, only in this paragraph, $\theta(\chi)$ (resp. $\theta(\chi^{-1})$) denote the $p$-stabilized theta-series of $\chi$ (resp. of $\chi^{-1}$), which is a weight-one CM form. We note that the factorization of non-critical Rankin--Selberg $p$-adic $L$-functions allows us to recast the problem to determine the order of vanishing of $L_{p,\chi}^{\rm Katz}|_{\Gamma_\cyc}$ at $\mathds{1}$  to the same for the non-critical Rankin--Selberg $p$-adic  $L$-function attached to the Rankin--Selberg convolution $\theta(\chi)\times \theta(\chi^{-1})$. Our goal in \cite{BS3} is to use the deformations of Beilinson--Flach elements to weight-one (more precisely, to the point in the ${\rm GL}_2\times{\rm GL}_2$-eigenvariety that corresponds to $\theta(\chi)\times \theta(\chi^{-1})$) to obtain a non-vanishing criterion for the cyclotomic group-ring-valued $\mathcal L$-invariant in terms of these. This draws from the work of Rivero--Rotger~\cite{RiveroRotgerJEMS}, where they study a similar problem in a scenario when the relevant local Galois representation at $p$ is regular. In the setting we place ourselves, the Galois representation at hand is no longer $p$-regular and therefore, the results of \cite{RiveroRotgerJEMS} are not immediately available.
\subsection{Notation and set up}
\label{subsec_notation_set_up_1_2}
For a number field $E$, we write $S_{\infty}(E)$ for the archimedean places of $E$ and $S_p(E)$ for the primes of $E$ above $p$. When $E'/E$ is an extension of number fields, we denote by $D_{E'/E}$  the relative discriminant and $S_{\rm ram}(E'/E)$ the set of places of $E$ which ramifies in the extension $E'/E$.

We fix a CM field $K$ and denote by $F$ its maximal totally real subfield. We let $c \in \Gal(K/F)$ be the unique non-trivial automorphism of $K$. For any Hecke character $\lambda$ of $K$, we set $\lambda^c:=\lambda\circ c$ and $\lambda^{\rm ad}:=\lambda^c/\lambda$. We assume throughout this article that $K$ verifies the hypothesis \ref{item_ord}. We fix a set $\Sigma\subset S_p(K)$ such that 
\[
\Sigma \cup \Sigma^c = S_p(K) \ \text{ and } \ \Sigma \cap \Sigma^c = \emptyset. 
\]
Such a choice is called a $p$-adic CM type. 

Let $F_\cyc/F$ denote the cyclotomic $\ZZ_p$-extension and put $\Gamma_\cyc:=\Gal(F_\cyc/F)$. We let $K(p^\infty)/K$ denote the compositum of all $\ZZ_p$-extensions of $K$ and set $\Gamma_\infty:=\Gal(K(p^\infty)/K)$. We denote by $K_{\rm ac}/K$ the anticyclotomic tower over $K$ and define $\Gamma_{\rm ac}:=\Gal(K_{\rm ac}/K)$. We let $K_{\rm cyc}$ denote the cyclotomic $\ZZ_p$-extension of $K$. Throughout this paper, we always identify the Galois group of the cyclotomic $\bZ_{p}$-extension $K_{\rm cyc}/K$ with $\Gamma_{\rm cyc}$ via the canonical isomorphism $\Gal(K_{\rm cyc}/K) \stackrel{\sim}{\longrightarrow} \Gamma_{\rm cyc}$. We denote by $\Gamma_\infty^+$ the Galois group of the maximal $\ZZ_p$-power extension $F_\infty$ of $F$; note that $F_\infty=F_\cyc$ and $\Gamma_\infty^+=\Gamma_\cyc$ if Leopoldt's conjecture for $F$ holds true. We identify the Galois group of the extension $KF_\infty/K$ with $\Gamma_\infty^+$ via the canonical isomorphism $\Gal(KF_{\infty}/K) \stackrel{\sim}{\longrightarrow} \Gamma_{\infty}^+$. Finally, we put $\Gamma_\infty^\circ:=\Gal(K_{\rm ac}K_{\rm cyc}/K)$. Note that $\Gamma_\infty^\circ\cong \ZZ_p^{[F:\QQ]+1}$ and $\Gamma_\infty\cong \ZZ_p^{[F:\QQ]+1+\delta}$ where $\delta$ is the Leopoldt defect. If Leopoldt's conjecture for $F$ holds true, then $K_{\rm ac}K_\cyc=K(p^\infty)$ and $\Gamma_\infty^\circ=\Gamma_\infty$.

Via the fixed embedding $\iota$, we treat all ray class characters also as $p$-adic characters. In \S\S\ref{sec_Selmer_Groups}--\ref{sec_Linvariants_heights_exceptionalzeros} where our discussion is related to \cite{BS19}, we denote the ray class character we work with by $\chi$ (so as to parallel to notation in op. cit.). In \S\S\ref{sec_Hida_RS}--\ref{sec:Factorisation of Selmer groups for symmetric Rankin--Selberg products} where we extensively dwell on the constructions of Hida and Hida--Tilouine, we shall denote the ray class character we have fixed by $\psi$ (so as to parallel to notation of Hida--Tilouine). We denote the conductor of both $\chi$ and $\psi$ by $\frak{c}$ and assume that $(\frak{c},p)=1$. With a slight abuse, we shall denote the field cut out either by $\chi$ or $\psi$ by $L$ and assume that $[L:K]$ is coprime to $p$. 

We let $\bZ_p^{\rm ur}:= W(\overline{\mathbb{F}}_p)$ denote the ring of Witt vectors of the field $\overline{\mathbb{F}}_p$. {We choose a discrete valuation $\bZ_{p}$-algebra $\cO$ which is  finite flat over $\bZ_{p}$ and contains all values of $\chi$ and $\psi$. We also assume $\cO\subset \bZ_p^{\rm ur}$. We note that such a ring  exists since  the order of $\chi$ and $\psi$ is prime to $p$}.

For a pro-finite group $G$, we denote by $\widehat G$ the group of $\overline{\QQ}_p$-valued continuous characters of $G$. Given a topological $G$-module $M$, we let $C^{\bullet}(G,M)$ denote the complex of continuous cochains. For any torsion-free quotient $\Gamma$ of $\Gamma_\infty$, we shall put $\LL(\Gamma):=\ZZ_p[[\Gamma]]$ and for any {linearly topologized} $\ZZ_p$-algebra $A$, we set 
$\LL_{A}(\Gamma):=A \widehat{\otimes} \LL(\Gamma)$. We let $\LL_{A}(\Gamma)^\sharp$ denote free $\LL_{A}(\Gamma)$-module of rank one on which $G_K$ acts via the tautological character $G_K\twoheadrightarrow \Gamma \hookrightarrow \LL_{A}(\Gamma)^\times$ and $\LL_{A}(\Gamma)^\iota$ denote its $\LL_{A}(\Gamma)$-linear contragredient, so that $\LL_{A}(\Gamma)^\iota$ is the free $\LL_{A}(\Gamma)$-module of rank one on which $G_K$ acts via the character $G_K\twoheadrightarrow \Gamma \xrightarrow{\gamma\mapsto \gamma^{-1}} \Gamma\hookrightarrow \LL_{A}(\Gamma)^\times$. We note that the $G_K$-representations $\LL_{A}(\Gamma_{\rm cyc})^\sharp$ and $\LL_{A}(\Gamma_{\rm cyc})^\iota$ both extend to a representation of $G_F$.

\subsubsection{} We collect here the notation for all the $p$-adic $L$-functions (algebraic or analytic) that we work with in this article. Whenever applicable, we indicate where in our article they are introduced.
\begin{itemize}
    \item $L_{p,\chi}^{\rm Katz}$: Katz $p$-adic $L$-function in $[F:\QQ]+1+\delta$-variables. Its defining properties are presented in \S\ref{subsec_Katz_padic_L_analytic_non_critical_exceptional_zeros}.
    \begin{itemize}
        \item Its algebraic counterpart $L_{p,\chi}^{\rm alg}$ is introduced in Definition~\ref{defn_algebraic_Katz_padic_L}.
        \item Given a $\ZZ_p$-extension $K_\Gamma/K$ with Galois group $\Gamma$, we write $L_{p,\chi}^{\rm Katz}\vert_\Gamma$ for its restriction to $\widehat{\Gamma}$, which is a one-variable $p$-adic $L$-function.
        \item We note that in our factorization formulae, the $p$-adic $L$-functions $L_{p, \mathds{1}}^{\rm Katz}$ and $L_{p, \mathds{1}}^{\rm alg}$ corresponding to the particular case when $\chi=\mathds{1}$ makes an appearance.
    \end{itemize}
    \item $\widetilde{\mathscr{D}}^{\rm Hida}=\widetilde{\mathscr{D}}^{\rm Hida}(\Theta_{\Sigma,\psi}\otimes \Theta_{\Sigma,\psi})$:  regularized non-critical specialization of the Rankin--Selberg $p$-adic $L$-function in $[F:\QQ]+1$-variables, introduced in Definition~\ref{defn_regularized_Hida_RS}, is attached to the self-Rankin--Selberg product $\Theta_{\Sigma,\psi}\otimes \Theta_{\Sigma,\psi}$ of the nearly ordinary CM family of Hilbert modular forms $\Theta_{\Sigma,\psi}$ with branch character $\psi$.
    \begin{itemize}
        \item Its algebraic counterpart $\widetilde{\mathscr{D}}^{\rm alg}=\widetilde{\mathscr{D}}^{\rm alg}(\Theta_{\Sigma,\psi}\otimes \Theta_{\Sigma,\psi})$  is introduced in Definition~\ref{defn_algebraic_padic_Rankin_Selberg} and Theorem~\ref{thm:iwasama main conj rankin-selberg}(ii).
\end{itemize}
\item $\widetilde{\mathscr D}_{{\rm ad}^0\Theta_{\Sigma,\psi}}$: adjoint $p$-adic $L$-function in $[F:\QQ]+1$-variables, introduced in Definition~\ref{defn_regularized_Hida_RS}.
\begin{itemize}
        \item Its algebraic counterpart $\widetilde{\mathscr D}^{\rm alg}_{{\rm ad}^0\Theta_{\Sigma,\psi}} $ is introduced in \S\ref{defn_algebraic_adjoint_padic_L_function}.
\end{itemize}
\item $L_p^{\rm DR}(s,\epsilon_{K/F}\omega)$ and $\zeta_p(s)$: Deligne--Ribet $p$-adic $L$-function associated to the totally even character $\epsilon_{K/F}\omega_F$ and the Dedekind $p$-adic zeta function respectively, both on the single (cyclotomic) variable $s$.   
\begin{itemize}
    \item Their algebraic counterparts (that our main factorization results concern) $L_{\epsilon_{K/F}\omega_F}^{\rm alg, DR}$ and $\zeta_{\rm alg}$ are introduced in Definition~\ref{defn_alg_DR_p-adicL}. 
\end{itemize}\end{itemize}

\acknowledgements{We thank the anonymous referee for many useful suggestions and helpful comments, which guided us towards many technical and stylistic improvements to the earlier versions of our article.}

\section{Selmer complexes and main conjectures}
\label{sec_Selmer_Groups}
In this section, we review Nekov\'a\v{r}'s definition of Selmer complexes and study their basic properties in the particular case that is relevant to the study of Katz' $p$-adic $L$-functions attached to CM fields. Our main goal in this section is to reformulate (see Conjecture~\ref{conj:iwasawa main conj}) the Iwasawa Main Conjectures for CM fields in terms of Nekov\'a\v{r}'s extended Selmer groups.


\subsection{Definition of Selmer complexes}
In this subsection, we shall work in great generality and let $(R,\fm)$ be a complete Noetherian local ring with residue characteristic $p>2$. We let $K$ be a any number field and fix a finite set $S$ of primes of $K$ satisfying $S_p(K) \subset S$, where $S_p(K)$ is the set of primes of $\cO_K$ that lie above $p$.  We write $K_S$ for the maximal algebraic extension of $K$ which is unramified outside $S \cup S_{\infty}(K)$. 
We put $G_{K,S} := \Gal(K_S/K)$.  
Let $T$ be a free $R$-module of finite rank which is equipped with a continuous $G_{K,S}$-action. 

We next recall basic constructions of Nekov\'a\v{r} \cite{Nek} in this general set up and collect useful facts for our study.
\begin{definition}
\item[i)] For $G=G_{K,S}$ or $G_{K_v}$ with $v\in S$, we let $C^{\bullet}(G,T)$ denote the complex of continuous cochains with values in $T$. 

\item[ii)] When $G = G_{K_{v}}$ with $v \nmid p$, we define the complex $C^{\bullet}_{\rm ur}(G, T)$ of unramifed cochains on setting 
\[
C^{\bullet}_{\rm ur}(G, T):=[T^{I_{v}} \xrightarrow{{\rm Frob}_{v}-1} T^{I_{v}}]
\]
which is concentrated in degrees $0$ and $1$. 
Here, $I_{v}$ denotes the inertia subgroup of $G_{K_{v}}$ and ${\rm Frob}_{v} \in G_{K_{v}}/I_{v}$ denotes the geometric Frobenius. As explained in \cite[\S7.2]{Nek}, there is a natural morphism 
\[
C^{\bullet}_{\rm ur}(G, T) \longrightarrow C^{\bullet}(G, T)\,. 
\]

\item[iii)] For each prime $v\in S$, a local condition $\Delta_v$ on $C^{\bullet}(G_{K,S},T)$ is the datum given by a morphism of complexes of $R$-modules 
\[
i_v^+ \colon U_v^+ \longrightarrow C^{\bullet}(G_{K_{v}},T) \,.
\]
Given a local condition $\Delta_v$ as above, we define 
\[
U_{v}^{-}(T) := {\rm Cone}\left(U_{v}^{+} \xrightarrow{-i_{v}^{+}} C^{\bullet}(G_{K_{v}},T)\right) 
\]
and set $i_{S}^{+} := \{i_{v}^{+}\}_{v \in S}$ and $U_{S}^{\pm} := \bigoplus_{v \in S }U^{\pm}_{v}$. 
We also put 
\[
{\rm res}_{S} \colon C^{\bullet}(G_{K,S},T) \xrightarrow{\oplus_{v\in S}{\rm res}_v} \bigoplus_{v \in S }C^{\bullet}(G_{K_{v}},T)
\]
for the sum of the restriction maps. 

\item[iv)] Given a local condition $\Delta_v$ for each $v\in S$, let us put $\Delta := (\Delta_v)_{v \in S}$. The Selmer complex associated to the pair $(T,\Delta)$ is the complex
\[
\widetilde{C}_{\rm f}^\bullet(G_{K,S},T,\Delta) := {\rm Cone}\left(C^{\bullet}(G_{K,S},T) \oplus U^{+}_{S} \xrightarrow{{\rm res}_{S}-i_{S}^{+}} \bigoplus_{v \in S}C^{\bullet}(G_{K_{v}},T)\right)[-1], 
\]
where $[n]$ is translation of a complex by $n \in \ZZ$ (in the sense of \cite[\S1.1.1]{Nek}). 
\item[iv)] We denote by $\widetilde{{\bf R}\Gamma}_{\rm f}(G_{K,S},T,\Delta) \in D({}_R{\rm Mod})$ the object in the derived category corresponding to $\widetilde{C}_{\rm f}^\bullet(G_{K,S},T,\Delta)$ and we denote its cohomology by $\widetilde{H}^{i}_{\rm f}(G_{K,S},T,\Delta)$. 
\end{definition}

\begin{lemma}
There is a natural exact triangle 
\begin{align}\label{fundamental exact tri}
\widetilde{{\bf R}\Gamma}_{\rm f}(G_{K,S},T,\Delta) \longrightarrow {\bf R}\Gamma(G_{K,S},T) \longrightarrow U^{-}_{S}(T) 
\end{align}
in $D({}_R{\rm Mod})$.
\end{lemma}
\begin{proof}
This is immediate from the definition of $\widetilde{{\bf R}\Gamma}_{\rm f}(G_{K,S},T,\Delta)$. 
\end{proof}

\begin{lemma}
\label{lem:perf[0,3]}
Suppose for every $v \in S$ that $U_{v}^{+} \in D_{\rm parf}^{[0,2]}({}_R{\rm Mod})$ is represented by a perfect complex of $R$-modules with degrees concentrated in $0,1,2$. Then, 
\[
\widetilde{{\bf R}\Gamma}_{\rm f}(G_{K,S},T,\Delta) \in D_{\rm parf}^{[0,3]}({}_R{\rm Mod}). 
\]
\end{lemma}
\begin{proof}
The argument in the proof of \cite[Proposition 9.7.2(ii)]{Nek} applies verbatim. 
\end{proof}


\subsection{Selmer complexes attached to Artin motives}
\label{subsection:example-cm}
In this subsection, we concentrate in the particular case of Artin motives and explicitly determine the Selmer complex associated to the universal deformations of Artin motives over a CM field $K$. 

We recall that $\cO \subset \bZ_p^{\rm ur}$ is a fixed discrete valuation $\bZ_{p}$-algebra which is  finite flat over $\bZ_{p}$ and contains all values of $\chi$ and $\psi$. We fix a subextension $\frak{K}\subset K(p^\infty)$ of $K$ and set $R:= \cO[[\Gal(\frak{K}/K)]]$. We let $\cO[[\Gal(\frak{K}/K)]]^{\iota}$ denote the free $R$-module of rank one on which $G_{K}$ acts via the character
$$G_{K} \twoheadrightarrow \Gal(\frak{K}/K) \xrightarrow{g\mapsto g^{-1}} \Gal(\frak{K}/K) \hookrightarrow R^\times$$
and set $\bT_{\frak K}:= \cO(1) \otimes \chi^{-1} \otimes_{\cO} \cO[[\Gal(\frak{K}/K)]]^{\iota}$, on which $G_K$ acts diagonally. 
We take a finite set $S$ of primes of $K$ containing $S_{\rm ram}(\chi) \cup S_{p}(K)$, where $S_{\rm ram}(\chi)$ consists of the set of primes of $K$ where $\chi$ is ramified. 

\begin{definition}\label{def:extended selmer}
For any ideal $I$ of $R$, the local conditions $\Delta_{\Sigma}$ on $C^{\bullet}(G_{K,S}, \bT_{\frak K}/I \bT_{\frak K})$ are given with the choices
\[
U_{v}^{+} := 
\begin{cases}
C^{\bullet}(G_{K_{v}}, \bT_{\frak K}/I \bT_{\frak K}) & \text{if} \ v \in \Sigma^c, 
\\
0  & \text{if} \ v \in \Sigma,   
\\
C_{\rm ur}^{\bullet}(G_{K_{v}}, \bT_{\frak K}/I \bT_{\frak K}) & \text{if} \ v \in S \setminus S_{p}(K)
\end{cases}
\]
\end{definition}

\begin{remark}
\label{remark_local_cohom_bad_primes_acyclic}
For each prime $v \in S_{\rm ram}(\chi) \setminus S_p(K)$, the complex $C^\bullet(G_{K_v}, \bT_{\frak K})$ is acyclic. Indeed, let us write $\overline{T} := \bT_{\frak K}/\fm \bT_{\frak K}$ for the residual representation of $\bT_{\frak K}$. 
Since we have a canonical isomorphism 
\[
{\bf R}\Gamma(G_{K_v}, \bT_{\frak K}) \otimes^{\bL}_R R/\fm \stackrel{\sim}{\longrightarrow} {\bf R}\Gamma(G_{K_v}, \overline{T}), 
\] 
it suffices to show that the complex $C^\bullet(G_{K_v}, \overline{T})$ is acyclic. 
Note that we have $\chi_{\vert_{I_v}} \neq \mathds{1}$, since $v \in S_{\rm ram}(\chi)$. 
Since $I_v$ acts trivially on $\mu_p$, we have $H^0(G_{K_v}, \overline{T}) = 0$. Moreover, it follows from  the local Tate duality that $H^2(G_{K_v}, \overline{T})=0$. Since $v \nmid p$, the local Euler characteristic formula reads
\[
\# H^1(G_{K_v}, \overline{T}) = \# H^0(G_{K_v}, \overline{T}) \cdot \# H^2(G_{K_v}, \overline{T}), 
\]
and concludes the vanishing of $H^1(G_{K_v}, \overline{T})$. In particular, when $S = S_{\mathrm{ram}}(\chi) \cup S_p(K)$, we obtain an exact triangle 
\begin{align}\label{fundamental_exact_tri2}
\widetilde{{\bf R}\Gamma}_{\rm f}(G_{K,S}, \bT_{\frak K}/I\bT_{\frak K}, \Delta_{\Sigma}) \longrightarrow {\bf R}\Gamma(G_{K,S}, \bT_{\frak K}/I\bT_{\frak K}) 
\longrightarrow \bigoplus_{v \in \Sigma}{\bf R}\Gamma(G_{K_{v}}, \bT_{\frak K}/I\bT_{\frak K})  
\end{align}
thanks to \eqref{fundamental exact tri}, for any ideal $I$ of $R$. 
\end{remark}

\begin{lemma}\label{lemma:indep}
The object $\widetilde{{\bf R}\Gamma}_{\rm f}(G_{K,S}, \bT_{\frak K}/I\bT_{\frak K}, \Delta_{\Sigma})$ is independent of the choice of the finite set $S$. In more precise terms, for any finite set $S'$ of primes of $K$ containing $S$, the inflation map 
\[
C^\bullet(G_{K,S}, \bT_{\frak K}/I\bT_{\frak K}) \longrightarrow C^\bullet(G_{K,S'}, \bT_{\frak K}/I\bT_{\frak K})
\] 
induces an isomorphism 
\[
\widetilde{{\bf R}\Gamma}_{\rm f}(G_{K,S}, \bT_{\frak K}/I\bT_{\frak K}, \Delta_{\Sigma}) 
\stackrel{\sim}{\longrightarrow} 
\widetilde{{\bf R}\Gamma}_{\rm f}(G_{K,S'}, \bT_{\frak K}/I\bT_{\frak K}, \Delta_{\Sigma}). 
\]
\end{lemma}
\begin{proof}
This is a consequence of \cite[Proposition 7.8.8(ii)]{Nek}. 
\end{proof}
By definition, one clearly has for each $v \in S$
\[
U_v^+ \in D_{\rm parf}^{[0,2]}({}_{R}{\rm Mod})\,.
\]
In particular, the complex $\widetilde{C}_{\rm f}^\bullet(G_{K,S}, \bT_{\frak K}, \Delta_{\Sigma})$ is perfect. Furthermore, Pottharst proved in \cite[Theorem~1.12]{Jonathan13} (see also \cite{Nek}, Proposition 9.7.3(i)) that 
we have the following natural base-change isomorphisms in the derived category:
\begin{align}\label{eq:base-change}
\widetilde{{\bf R}\Gamma}_{\rm f}(G_{K,S}, \bT_{\frak K}/I\bT_{\frak K}, \Delta_{\Sigma}) \otimes^{\bL}_{R} R/J \stackrel{\sim}{\longrightarrow} \widetilde{{\bf R}\Gamma}_{\rm f}(G_{K,S}, \bT_{\frak K}/J\bT_{\frak K}, \Delta_{\Sigma})
\end{align}
for any ideal $I \subseteq J$ of $R$, 
\begin{align}\label{eq:scalar-extension}
\widetilde{{\bf R}\Gamma}_{\rm f}(G_{K,S}, \bT_{\frak K}, \Delta_{\Sigma}) \otimes^{\bL}_{R} \bZ_p^{\rm ur}[[\Gal(\frak{K}/K)]] \stackrel{\sim}{\longrightarrow} \widetilde{{\bf R}\Gamma}_{\rm f}(G_{K,S}, \bT_{\frak K} \otimes_R \bZ_p^{\rm ur}[[\Gal(\frak{K}/K)]], \Delta_{\Sigma}), 
\end{align}
where $\bZ_p^{\rm ur}[[\Gal(\frak{K}/K)]]$ is endowed with the trivial $G_K$-action.

\begin{proposition}\label{prop:parf[1,2]}
We have
\[
\widetilde{{\bf R}\Gamma}_{\rm f}(G_{K,S}, \bT_{\frak K}, \Delta_{\Sigma}) \in D_{\rm parf}^{[1,2]}({}_R{\rm Mod}). 
\]
In particular, combined with the isomorphism \eqref{eq:base-change}, we have a canonical isomorphism 
\begin{align}\label{isom_base_change_H2}
\widetilde{H}^2_{\rm f}(G_{K,S}, \bT_{\frak K}, \Delta_{\Sigma}) \otimes_{R} R/I \stackrel{\sim}{\longrightarrow} \widetilde{H}^2_{\rm f}(G_{K,S}, \bT_{\frak K}/I\bT_{\frak K}, \Delta_{\Sigma}) 
\end{align}
for any ideal $I$ of $R$.
\end{proposition}
\begin{proof}
Let  $I$ be any ideal of $R$. We shall verify that
\[
\widetilde{H}_{\rm f}^0(G_{K,S}, \bT_{\frak K}/I\bT_{\frak K}, \Delta_{\Sigma}) = 0=
\widetilde{H}_{\rm f}^3(G_{K,S}, \bT_{\frak K}/I\bT_{\frak K},\Delta_{\Sigma}) 
\]
and use this to prove that $\widetilde{{\bf R}\Gamma}_{\rm f}(G_{K,S}, \bT_{\frak K}, \Delta_{\Sigma}) \in D_{\rm parf}^{[1,2]}({}_R{\rm Mod})$.

We first explain the vanishing of $\widetilde{H}_{\rm f}^3(G_{K,S}, \bT_{\frak K}/I \bT_{\frak K},\Delta_{\Sigma})$. By the definition of the Selmer complex $\widetilde{C}_{\rm f}^\bullet(G_{K,S}, \bT_{\frak K}/I \bT_{\frak K},\Delta_{\Sigma})$, we have an exact sequence
\[
H^2(G_{K,S}, \bT_{\frak K}/I \bT_{\frak K}) \longrightarrow \bigoplus_{v \in S \setminus  \Sigma^c} H^2(G_{K_v}, \bT_{\frak K}/I \bT_{\frak K})\longrightarrow 
\widetilde{H}_{\rm f}^3(G_{K,S}, \bT_{\frak K}/I \bT_{\frak K},\Delta_{\Sigma}) \longrightarrow 0\,.  
\]
By the Poitou-Tate global duality, we also have an exact seqence 
\[
H^2(G_{K,S}, \bT_{\frak K}/I \bT_{\frak K}) \longrightarrow \bigoplus_{v \in S} H^2(G_{K_v}, \bT_{\frak K}/I \bT_{\frak K})\longrightarrow 
H^0(G_{K,S}, (\bT_{\frak K}/I \bT_{\frak K})^\vee(1))^\vee \longrightarrow 0\,.  
\]
Here, $(-)^\vee := \Hom_{\rm cont}(-, \bQ_p/\bZ_p)$ denotes the Pontryagin duality functor. For any prime $v \in S_p(K)$, the homomorphism 
\[
H^2(G_{K_v}, \bT_{\frak K}/I \bT_{\frak K}) \longrightarrow H^0(G_{K,S}, (\bT_{\frak K}/I \bT_{\frak K})^\vee(1))^\vee
\]
is surjective since the Pontryagin dual of this homomorphism is the injection 
\[
H^0(G_{K,S}, (\bT_{\frak K}/I \bT_{\frak K})^\vee(1)) \longrightarrow  H^0(G_{K_v}, (\bT_{\frak K}/I \bT_{\frak K})^\vee(1)) \cong H^2(G_{K_v}, \bT_{\frak K}/I \bT_{\frak K})^\vee. 
\]
Since $\Sigma^c$ is non-empty set, we infer that the homomorphism 
\[
H^2(G_{K,S}, \bT_{\frak K}/I \bT_{\frak K}) \longrightarrow \bigoplus_{v \in S \setminus  \Sigma^c} H^2(G_{K_v}, \bT_{\frak K}/I \bT_{\frak K})
\]
is surjective, which concludes the proof that $\widetilde{H}_{\rm f}^3(G_{K,S}, \bT_{\frak K}/I \bT_{\frak K},\Delta_{\Sigma})=0$. 

We next explain the vanishing of $\widetilde{H}_{\rm f}^0(G_{K,S}, \bT_{\frak K}/I \bT_{\frak K},\Delta_{\Sigma})$. By the definition of the Selmer complex $\widetilde{C}_{\rm f}^\bullet(G_{K,S}, \bT_{\frak K}/I \bT_{\frak K},\Delta_{\Sigma})$, we have an isomorphism  
\[
\widetilde{H}_{\rm f}^0(G_{K,S}, \bT_{\frak K}/I \bT_{\frak K}, \Delta_{\Sigma}) 
\stackrel{\sim}{\longrightarrow} \ker\left( H^0(G_{K,S}, \bT_{\frak K}/I \bT_{\frak K}) \longrightarrow \bigoplus_{v \in S \setminus \Sigma^c} H^0(G_{K_v}, \bT_{\frak K}/I \bT_{\frak K})\right)\,.
\]
This shows that $\widetilde{H}_{\rm f}^0(G_{K,S}, \bT_{\frak K}/I \bT_{\frak K},\Delta_{\Sigma}) = 0$.

It follows from the vanishing of $\widetilde{H}_{\rm f}^3(G_{K,S}, \bT_{\frak K}, \Delta_{\Sigma})$ together with Lemma~\ref{lem:perf[0,3]} that  
\[
\widetilde{{\bf R}\Gamma}_{\rm f}(G_{K,S},\bT_{\frak K},\Delta) \in D_{\rm parf}^{[0,2]}(R), 
\]
so that $\widetilde{{\bf R}\Gamma}_{\rm f}(G_{K,S},\bT_{\frak K}, \Delta_{\Sigma})$ can be represented by a complex 
\[
P^{\bullet} = [ \ \cdots \longrightarrow 0 \longrightarrow P^{0} \longrightarrow P^{1} \longrightarrow P^{2} \longrightarrow 0 \longrightarrow \cdots \ ]
\] 
of finitely generated free $R$-modules. Since  $\widetilde{H}^{0}_{\rm f}(G_{K,S},\bT_{\frak K},\Delta_{\Sigma})=0$, the homomorphism $P^{0} \longrightarrow P^{1}$ is injective. 

To conclude with the proof that $\widetilde{{\bf R}\Gamma}_{\rm f}(G_{K,S}, \bT_{\frak K}, \Delta_{\Sigma}) \in D_{\rm parf}^{[1,2]}({}_R{\rm Mod})$, it suffices to show that the $R$-module $X := {\rm coker}\left(P^{0} \longrightarrow P^{1}\right)$ is projective. Since $R$ is a Noetherian local ring, it suffices to verify that ${\rm Tor}_{1}^{R}(X,R/\fm) = 0$ by the local criterion for flatness. 
Since $P^{1}$ is flat, one has 
\begin{align*}
{\rm Tor}_{1}^{R}(X,R/\fm) 
&= \ker \left(P^{0}\otimes_{R} R/\fm \longrightarrow P^{1}\otimes_{R} R/\fm\right) 
\\
&=  H^{0}(\widetilde{{\bf R}\Gamma}_{\rm f}(G_{K,S},\bT_{\frak K},\Delta_{\Sigma}) \otimes^{\bL}_{R}R/\fm)
\\
&\cong \widetilde{H}_{\rm f}^0(G_{K,S},\bT_{\frak K}/\fm \bT_{\frak K},\Delta_{\Sigma}) =0, 
\end{align*}
where the isomorphism on the third line follows from \eqref{eq:base-change}. This shows that $X$ is indeed projective and we have proved $\widetilde{{\bf R}\Gamma}_{\rm f}(G_{K,S}, \bT_{\frak K}, \Delta_{\Sigma}) \in D_{\rm parf}^{[1,2]}({}_R{\rm Mod})$.

The isomorphism \eqref{isom_base_change_H2} is an immediate consequence of this assertion.
\end{proof}

\begin{definition}
Suppose $C \in D_{\rm parf}({}_R{\rm Mod})$ is represented by a complex  
\[
F^{\bullet} = [ \ \cdots \longrightarrow F^0 \longrightarrow F^1 \longrightarrow F^2 \longrightarrow \cdots \ ]
\]
of finitely generated free $R$-modules such that $F^{N} = F^{-N} = 0$ for all sufficiently large integers $N$.  We define the Euler--Poincar\'e characteristic $\chi(C)$ as the alternating sum
\[
\chi(C):= \sum_{n \in \bZ} (-1)^n  {\rm rank}_R(F^n). 
\]
\end{definition}

\begin{lemma}
\label{lemma:Euler characteristic = 0}
$\chi(\widetilde{{\bf R}\Gamma}_{\rm f}(G_{K,S},\bT_{\frak K},\Delta_{\Sigma})) = 0. $
\end{lemma}
\begin{proof}
Since $K$ is a CM field, it follows from \cite[Theorem~7.8.6]{Nek} that
\begin{align*}
\chi(\widetilde{{\bf R}\Gamma}_{\rm f}(G_{K,S},\bT_{\frak K},\Delta_{\Sigma}))&= 
\sum_{v \in S_\infty(k)}1 - \sum_{v \in \Sigma^c}[K_v \colon \bQ_p] \\
& = [F \colon \bQ] -  \sum_{v \in \Sigma^c}[K_v \colon \bQ_p]
\end{align*}
It follows from our running ordinarity hypothesis \ref{item_ord} that
\[
\sum_{v \in \Sigma^c}[K_v \colon \bQ_p] = [F \colon \bQ], 
\]
which shows that $\chi(\widetilde{{\bf R}\Gamma}_{\rm f}(G_{K,S},\bT_{\frak K},\Delta_{\Sigma}))=0$, as required. 
\end{proof}

\begin{corollary}\label{corollary:H_f^1 vanishes iff H_f^2 is torsion and fitt=char}
The following assertions are valid. 
    \item[i)] The $R$-module $\widetilde{H}^{2}_{\rm f}(G_{K,S}, \bT_{\frak K},\Delta_{\Sigma})$ is torsion  if and only if  $\widetilde{H}^{1}_{\rm f}(G_{K,S}, \bT_{\frak K},\Delta_{\Sigma})=0$.
    \item[ii)] If $R$ is regular, we have 
    \[
    {\rm Fitt}_R(\widetilde{H}^2_{\rm f}(G_{K,S}, \bT_{\frak K}, \Delta_{\Sigma}))={\rm char}_R(\widetilde{H}^2_{\rm f}(G_{K,S}, \bT_{\frak K}, \Delta_{\Sigma})). 
    \]
    \end{corollary}

\begin{proof}
\item[i)] Lemma~\ref{lemma:Euler characteristic = 0} combined with Proposition~\ref{prop:parf[1,2]} shows that there exists an exact sequence of $R$-modules 
\begin{equation}
\label{eqn_H2_free_resolution}
 0 \longrightarrow \widetilde{H}^1_{\rm f}(G_{K,S}, \bT_{\frak K}, \Delta_{\Sigma}) \longrightarrow P \stackrel{f}{\longrightarrow} 
P \longrightarrow \widetilde{H}^2_{\rm f}(G_{K,S}, \bT_{\frak K}, \Delta_{\Sigma}) \longrightarrow 0,  
\end{equation}
where $P$ is a free $R$-module of finite rank. 
Let ${\rm Frac}(R)$ denote the total ring of fractions of $R$. The vanishing of $\widetilde{H}^1_{\rm f}(G_{K,S}, \bT_{\frak K}, \Delta_{\Sigma})$ is equivalent to the requirement that the homomorphism 
$$f\otimes {\rm id}: P \otimes_R {\rm Frac}(R) \longrightarrow P \otimes_R {\rm Frac}(R)$$ 
be injective. By the definition of the ring $R$, we see that ${\rm Frac}(R)$ is a product of fields. Thence,  the homomorphism $f\otimes {\rm id}$ is injective if and only if it is surjective, which is equivalent to the requirement that $\widetilde{H}^2_{\rm f}(G_{K,S}, \bT_{\frak K}, \Delta_{\Sigma})$ is a torsion $R$-module. 

\item[ii)] If the $R$-module $\widetilde{H}^2_{\rm f}(G_{K,S}, \bT_{\frak K}, \Delta_{\Sigma})$ is not torsion, then we have 
\[
{\rm Fitt}_R(\widetilde{H}^2_{\rm f}(G_{K,S}, \bT_{\frak K}, \Delta_{\Sigma}))= 0 = {\rm char}_R(\widetilde{H}^2_{\rm f}(G_{K,S}, \bT_{\frak K}, \Delta_{\Sigma})). 
\]
We may therefore assume without loss that the $R$-module $\widetilde{H}^2_{\rm f}(G_{K,S}, \bT_{\frak K}, \Delta_{\Sigma})$ is torsion. It then follows from (i) that $\widetilde{H}^{1}_{\rm f}(G_{K,S}, \bT_{\frak K},\Delta_{\Sigma})=0$. The proof of~(ii) follows from the projective resolution \eqref{eqn_H2_free_resolution} of the $R$-module $\widetilde{H}^2_{\rm f}(G_{K,S}, \bT_{\frak K}, \Delta_{\Sigma})$. 
\end{proof}

\begin{corollary}\label{cor:H^1_f=0 and H^2_f is torsion}
If $\frak{K}$ contains the cyclotomic $\bZ_p$-extension of $K$, then the following assertions hold. 
    \item[i)] $\widetilde{H}^{1}_{\rm f}(G_{K,S}, \bT_{\frak K},\Delta_{\Sigma}) = 0$. 
    \item[ii)] The $R$-module $\widetilde{H}^{2}_{\rm f}(G_{K,S},\bT_{\frak K},\Delta_{\Sigma})$ is torsion.
\end{corollary}
\begin{proof}
\item[i)] It follows from the definition of the Selmer complex $\widetilde{{\bf R}\Gamma}_{\rm f}(G_{K,S},\bT_{\frak K},\Delta_{\Sigma})$ that
\begin{align*}
\widetilde{H}^{1}_{\rm f}(G_{K,S}, \bT_{\frak K} ,\Delta_{\Sigma}) = \ker\left(H^1(G_{K,S},\bT_{\frak K}) \longrightarrow \bigoplus_{v \in \Sigma} H^1(G_{K_v}, \bT_{\frak K}) 
\right). 
\end{align*}
Since $K$ contains the cyclotomic $\bZ_p$-extension of $K$, the required vanishing follows from the weak $\Sigma$-Leopoldt Conjecture proved by Hida and Tilouine in this set up (see \cite{HT94}, Theorem~1.2.2 and Lemma~1.2.5). 

\item[ii)] This assertion is an immediate consequence of (i) and Corollary~\ref{corollary:H_f^1 vanishes iff H_f^2 is torsion and fitt=char}(i). 
\end{proof}

\subsubsection{Iwasawa Main Conjectures for CM fields}
Our goal in this short subsection is to reformulate the Iwasawa main conjectures for CM fields in terms of the Fitting ideals of the Selmer complexes we have introduced in Section~\ref{subsection:example-cm}.

We let $L := \overline{K}^{\ker(\chi)}$ denote the field cut by the ray class character $\chi$ and let $M_{L, \Sigma}$ denote the maximal abelian pro-$p$-extension of $LK(p^{\infty})$ which is unramified outside the primes which lie above $\Sigma$. We set $X_{L, \Sigma} := \Gal(M_{L, \Sigma}/K(p^{\infty}))$ and define the Iwasawa module $X_{\Sigma}^{\chi}$ as be the maximal $\chi$-isotypic direct summand of $X_{L, \Sigma}$. We set (with the notation of \S\ref{subsection:example-cm}) $\bT_\infty:=\TT_{K(p^{\infty})}$.

\begin{proposition}\label{prop:quasi-isom}
There is a natural pseudo-isomorphism 
\[
X_\Sigma^{\chi} \longrightarrow \widetilde{H}^{2}_{\rm f}(G_{K,S}, \bT_\infty, \Delta_{\Sigma}). 
\]
\end{proposition}
\begin{proof}
This is \cite[Lemma C.6]{BS19} combined with the exact sequence (29) in op. cit.
\end{proof}

The Iwasawa main conjecture for the CM field $K$ asserts that the characteristic ideal of $X_\Sigma^{\chi}$ is generated by the $p$-adic $L$-function $(L_{p,\chi}^{\rm Katz})^{\iota} $. Using the isomorphism~\eqref{eq:scalar-extension} together with Corollary~\ref{corollary:H_f^1 vanishes iff H_f^2 is torsion and fitt=char}(ii) and  Proposition~\ref{prop:quasi-isom}, Iwasawa Main Conjecture for a CM field admits the following reformulation.

\begin{conjecture}[Iwasawa Main Conjecture for $K$]
\label{conj:iwasawa main conj}
$$
{\rm Fitt}_{\LL_{\cO}(\Gamma_\infty)}\left(\widetilde{H}^{2}_{\rm f}(G_{K,S}, \bT_\infty, \Delta_{\Sigma})\right)\bZ_p^{\rm ur}[[\Gamma_\infty]] = (L_{p,\chi}^{\rm Katz})^{\iota}  \cdot \bZ_p^{\rm ur}[[\Gamma_\infty]]. 
$$
\end{conjecture}

\begin{definition}
\label{defn_algebraic_Katz_padic_L}
We define the algebraic Katz $p$-adic $L$-function 
$$L_{p,\chi}^{\rm alg}\in {\rm Frac}(\LL_{\cO}(\Gamma_\infty))^\times\big{/}\LL_{\cO}(\Gamma_\infty)^\times$$
as the image of a generator of the $\LL_{\cO}(\Gamma_\infty)$-module
${\rm Fitt}_{\LL_{\cO}(\Gamma_\infty)}\left(\widetilde{H}^{2}_{\rm f}(G_{K,S}, \bT_\infty, \Delta_{\Sigma})\right)\,.$
\end{definition}

\begin{proposition}\label{proposition_vanishing_order_e}
Let us denote by $\sA := \ker(\LL_{\cO}(\Gamma_\infty) \to \cO)$ the full augmentation ideal and put $e := \# \{v \in \Sigma^c \mid \chi(G_{K_v}) = \{1\}\}$. We then have
\[
L_{p,\chi}^{\rm alg} \in \sA^e\,. 
\]
\end{proposition}
\begin{proof}
This is \cite[Lemma C.8]{BS19}; we present its proof here for the sake of the completeness. By definition, we have an exact triangle 
\[
{\bf R}\Gamma_{\rm c}(G_{K,S}, \bT_{\infty})  \longrightarrow \widetilde{{\bf R}\Gamma}_{\rm f}(G_{K,S}, \bT_{\infty}, \Delta_{\Sigma})
\longrightarrow \bigoplus_{v \in \Sigma^c}{\bf R}\Gamma(G_{K_{v}}, \bT_{\infty}). 
\]
Since $H^2_{\rm c}(G_{K,S}, \bT_{\infty}) \cong H^0(G_{K,S}, \bT_{\infty}^\vee(1))^\vee = 0$ by global duality, we deduce a surjection 
\[
\widetilde{H}^{2}_{\rm f}(G_{K,S}, \bT_\infty, \Delta_{\Sigma}) \longrightarrow 
\bigoplus_{v \in \Sigma^c} H^2(G_{K_v}, \bT_\infty). 
\]
For each prime $v \in \Sigma^c$ with $\chi(G_{K_v}) = \{1\}$, local Tate duality shows that  
\[
H^2(G_{K_v}, \bT_\infty) \cong H^0(G_{K_v}, \bT_\infty^\vee(1))^\vee \cong \cO[[\Gamma_\infty/\Gamma_v]], 
\]
where $\Gamma_v \subset \Gamma_\infty$ denotes the decomposition subgroup at $v$. Notice also that $H^2(G_{K_v}, \bT_\infty) = 0$ whenever $\chi(G_{K_v}) \neq \{1\}$. 
Hence we conclude that 
\[
L_{p,\chi}^{\rm alg} \in \prod_{v \in \Sigma^c}\mathrm{Fitt}_{\LL_{\cO}(\Gamma_\infty)}^0\left(H^2(G_{K_v}, \bT_\infty) \right) \subset \sA^e\,,
\]
as required.
\end{proof}


\section{$\mathcal{L}$-invariants, $p$-adic heights and exceptional zeros of Katz' $p$-adic $L$-functions}\label{sec_Linvariants_heights_exceptionalzeros}

In this section, we retain the notation of \S\ref{subsection:example-cm} as well as the $p$-ordinarity hypothesis~\ref{item_ord}. Our primary objective is to prove Theorem~\ref{thm_exceptional_zero_conj_OK}, which improves \cite[Corollary 1.4]{BS19}. Along the way, 
\begin{itemize}
    \item[i)] we introduce a universal group-ring-valued $\cL$-invariant  $\cL_\Sigma^{\rm Gal}$ in \eqref{eqn_universal_L_invariant}, which interpolates in a natural way the group-ring-valued $\cL$-invariants  $\cL_{\Sigma,\Gamma}^{\rm Gal}$ associated to the $\ZZ_p$-extensions $K_\Gamma$ of $K$; 
    \item[ii)] we relate the non-vanishing of group-ring-valued $\cL$-invariants to the non-degeneracy of Nekov\'a\v{r}'s $p$-adic height pairings and the semi-simplicity\footnote{The genesis of the semi-simplicity problems of the sort we consider here is  \cite{CoatesLichtenbaum1973}\,.} of the relevant extended Selmer groups.
\end{itemize}
We note that (ii) extends the main results of  \cite[Theorem 1]{BenoisIwasawa2012} where Benois treats the cyclotomic $\cL$-invariants. We also remark that the definition of the universal group-ring-valued $\cL$-invariant  $\cL_\Sigma^{\rm Gal}$ involves an alternative definition of $\cL_{\Sigma,\Gamma}^{\rm Gal}$ in terms of $p$-adic height pairings.

Let $\chi \colon G_{K} \longrightarrow \overline{\bQ}^{\times}$ be a finite prime-to-$p$ order non-trivial  character. 

\begin{definition}
\item[i)] We put $r := [F \colon \bQ]$ and define
\[
T := \cO(1) \otimes \chi^{-1} \,, \textrm{ \,\,\,\,\,\, } \, T^*(1) := \Hom_\cO(T, \cO)(1) = \cO(\chi). 
\]
We define $L := \overline{K}^{\ker(\chi)}$, the field cut out by the character $\chi$.
\item[ii)] We set $E(\Sigma^c, \chi) := \{v \in \Sigma^{c} \mid \chi(G_{K_{v}}) = \{1\}\}$ and put
$e := \# E(\Sigma^c, \chi)$. 
\end{definition}


\subsection{A semisimplicity criterion for extended Selmer groups}
Our goal in this subsection is to prove Corollary~\ref{corollary:equiv_semi-simple_fitt}, where we relate the order of exceptional zeros of an algebraic $p$-adic $L$-function to the semisimplicity properties of the relevant extended Selmer group.


\begin{definition}
Only in this paragraph,  we let $M$ be a finite extension of $\QQ_\ell$ and let $\cO_M$ denote its ring of integers. We put $H^1_{{\rm f}}(G_{M}, \ZZ_p(1)):=\widehat{\cO_{M}^{\times}}$, the $p$-adic completion of the units of $M$. 
\end{definition}
Note that we have a natural identification $H^1(G_{M}, \ZZ_p(1))=\widehat{M^{\times}}$ by Kummer theory, via which we shall think of $H^1_{{\rm f}}(G_{M}, \ZZ_p(1))$ as a submodule of  $H^1(G_{M}, \ZZ_p(1))$.
\begin{definition}
\label{defn_Kummer_theory}
Let $M$ be a number field.  Suppose that $\eta \colon G_{M} \longrightarrow \overline{\bQ}^\times$ is any non-trivial character of finite prime-to-$p$ order. 
Let us define $M_\eta:= \overline{M}^{\ker(\eta)}$ and set $\cO_\eta := \bZ_p[ {\rm im}(\eta) ]$. 
\item[i)] For any $\ZZ[\Gal(M/L)]$-module (respectively, $\ZZ_p[\Gal(M/L)]$-module) $X$, 
let us denote by $X^\eta$ the $\eta$-isotypic component of $\widehat{X}\otimes_{\Z_p}\cO_\eta$ (respectively, of $X\otimes_{\ZZ_p}\cO_\eta$). 
\item[ii)] Suppose $\cO$ is a finite flat extension of $\cO_\eta$.
We define 
\[
H^1_{\rm f}(G_{M_v}, \cO(1) \otimes \eta^{-1}):=(\cO_{M_\eta}\otimes_{\cO_{M}} \cO_{M_v})^{\times, \eta} \otimes_{\cO_{\eta}} \cO
\]
for any prime $v$ of $M$. 
\end{definition}

\begin{remark}
We retain the notation of Definition~\ref{defn_Kummer_theory}. Kummer theory together with the inflation-restriction sequence and our assumption that $[M_\eta : M]$ be coprime to $p$ yields a natural isomorphism
\[
H^1(G_{M_v}, \cO(1) \otimes \eta^{-1}) \stackrel{\sim}{\longrightarrow} (M_\eta \otimes_{M} M_v)^{\times, \eta} \otimes_{\cO_{\eta}} \cO \,.
\]
Through this identification, we will treat $H^1_{\rm f}(G_{M_v}, \cO(1) \otimes \eta^{-1})$ as a submodule of $H^1(G_{M_v}, \cO(1) \otimes \eta^{-1})$ and define the singular quotient
\begin{equation}
    \label{eqn_ingular_quotient}
    H^1_{/{\rm f}}(G_{M_v}, \cO(1) \otimes  \eta^{-1}):=H^1(G_{M_v}, \cO(1) \otimes  \eta^{-1})/H^1_{\rm f}(G_{M_v}, \cO(1) \otimes  \eta^{-1})\,.
\end{equation}
We then have the following commutative diagram, where the vertical identifications are induced from Kummer theory:
\begin{align}
\begin{split}\label{diagram:kummer}
    \xymatrix{H^1(G_{M}, \cO_{\eta}(1) \otimes  \eta^{-1}) \ar@{=}[d]\ar[r]^(.4){\res_p} & 
\prod_{v\in S_p(M)}
H^1(G_{M_v},  \cO_{\eta}(1) \otimes  \eta^{-1}) \ar@{=}[d]\ar[r]^(.55){\rm ord} & 
H^1_{/{\rm f}}(G_{M_v},  \cO_{\eta}(1) \otimes  \eta^{-1})\ar@{=}[d]
\\
M_\eta^{\times,\eta} \ar[r] & 
\prod_{v\in S_p(M)} ({M}_\eta\otimes_{{M}} {M_v})^{\times,\eta}\ar[r]^(.55){\prod {\rm ord}_v} &
\left(\prod_{v\in S_p(M)} \cO_{\eta}\right)^{\eta}\,.
}
\end{split}
\end{align}
\end{remark}

\begin{conjecture}[$\Sigma$-Leopoldt conjecture for $M/K$]
\label{conj_sigma_leopoldt}
For any finite extension $M$ of $K$, the canonical map 
\begin{equation}
\label{eqn_sigma_leopoldt_map}
    \cO_M^\times \otimes_\bZ \bZ_p \longrightarrow  \prod_{w \in \Sigma_M} \widehat{\cO_{M_w}^\times}=:\prod_{w \in \Sigma_M} H^1_{\rm f}(G_{M_w}, \bZ_p(1))
\end{equation}
is injective. Here, $\Sigma_M$ is the set of primes of $M$ which lie above those in $\Sigma$.
\end{conjecture}

We note that the $\Sigma$-Leopoldt conjecture for $M/K$ is (strictly, in general) stronger than the Leopoldt conjecture for $M/K$. As explained in \cite[Lemma 1.2.1]{HT94}, the $\Sigma$-Leopoldt conjecture is a consequence of the $p$-adic Schanuel conjecture.

\begin{remark}
When $K$ is an imaginary quadratic field and $M/K$ is abelian, the $\Sigma$-Leopoldt conjecture holds true. 
\end{remark}

\begin{remark} Suppose $M$ is a finite abelian extension of $K$. 
\label{remark_consequences_of_sigma_leopoldt}
\item[i)] Observe that 
\begin{align*}
    \mathrm{rank}_\bZ(\cO_M^\times) + 1 &= \frac{[M \colon \bQ]}{2} \\
    &= \sum_{w \in \Sigma_M} [M_w:\QQ_p]=\sum_{w \in \Sigma_M} \mathrm{rank}_{\bZ_p}\left(H^1_{\rm f}(G_{M_w}, \bZ_p(1))\right), 
\end{align*}
where the first equality follows from Dirichlet's unit theorem whereas the second as a consequene of our running assumption \ref{item_ord}.
\item[ii)] Suppose that 
${\eta} \colon \Gal(M/K) \longrightarrow \overline{\bQ}^\times$ 
is any non-trivial character of finite prime-to-$p$ order. 
Assuming the validity of $\Sigma$-Leopoldt conjecture for $M/K$, the $\QQ[\Gal(M/K)]$-module description of $\cO_M^\times\otimes_{\ZZ}\QQ$ shows that the cokernel of the homomorphism  
\begin{equation}
    \label{eqn_sigma_leopoldt_map_psi}
    \cO_M^{\times,{\eta} } \longrightarrow \prod_{v \in \Sigma} (\cO_{M}\otimes_{\cO_K} \cO_{K_v})^{\times,{\eta} }=:\prod_{v \in \Sigma} H^1_{\rm f}(G_{K_v}, \cO_{{\eta} }(1) \otimes {\eta} ^{-1})
\end{equation}
(which is induced from \eqref{eqn_sigma_leopoldt_map} on restricting to ${\eta}$-isotypic components) has finite cardinality. 
\end{remark}


\begin{definition}
Let us set $\Phi := {\rm Frac}(\cO)$ and $V := T \otimes_{\cO} \Phi$. We put $H^1_{\rm f}(G_{K_v}, V):=(\cO_{L}\otimes_{\cO_k} \cO_{K_v})^{\times,\chi}\otimes_{\cO_{\chi}}\Phi$, which we may naturally identify as a subspace of $H^1(G_{K_v}, V)\stackrel{\sim}{\longrightarrow}({L}\otimes_{k} {K_v})^{\times,\chi} \otimes_{\cO_{\chi}}\Phi$ by Kummer theory.  We finally set $H^1_{/{\rm f}}(G_{K_v}, V) := H^1(G_{K_v}, V)/H^1_{{\rm f}}(G_{K_v}, V)$ and observe that we have a natural isomorphism $H^1_{/{\rm f}}(G_{K_v}, V)\xrightarrow{\sim} (\prod_{v\in S_p(K)} \Phi)^{\chi}$ induced from the map $\ord$ in \eqref{diagram:kummer}. 
\end{definition}

Let us set $S := S_{p}(K) \cup S_{\rm ram}(L/K) = S_{p}(K) \cup S_{\rm ram}(\chi)$.

\begin{definition}
We define the local conditions $\Delta_{\Sigma}^*$ on $C^{\bullet}(G_{K,S}, T^*(1))$ by setting  
\[
U_{v}^{+} := 
\begin{cases}
0 & \text{if} \ v \in \Sigma^c, 
\\
C^{\bullet}(G_{K_{v}}, T^*(1))  & \text{if} \ v \in \Sigma, 
\\
0 & \text{if} \ v \in S_{\rm ram}(\chi) \setminus \Sigma. 
\end{cases}
\]
\end{definition}

We note that $T^{I_v}=0=(T^*(1))^{I_v}$ for any prime $v \in S_{\rm ram}(\chi)$ and as a result, the unramified local conditions at such primes are indeed the zero local conditions.

The following lemma tells us that the extended Selmer groups detect the exceptional zeros in this particular setting.
\begin{lemma}\label{lemma:rank of selmer}
Suppose that the $\Sigma$-Leopoldt conjecture for $L/K$ holds true. 
    \item[i)] The canonical map 
    \[
    \widetilde{H}_{\rm f}^1(G_{K,S}, T, \Delta_{\Sigma}) \otimes_{\bZ_p} \bQ_p \longrightarrow \bigoplus_{v \in E(\Sigma^c, \chi)}H^1_{/{\rm f}}(G_{K_v}, T) \otimes_{\bZ_p} \bQ_p 
    \]
    is an isomorphism. In particular, the $\cO$-module $\widetilde{H}^{1}_{\rm f}(G_{K,S}, T, \Delta_{\Sigma})$ is free of rank $e$. 
    \item[ii)] The canonical map induced from the fundamental exact triangle \eqref{fundamental exact tri} 
    \[
    \bigoplus_{v \in E(\Sigma^c, \chi)}H^0(G_{K_v}, T^*(1)) \longrightarrow \widetilde{H}_{\rm f}^1(G_{K,S}, T^*(1), \Delta_{\Sigma}^*) 
    \]
    is an isomorphism, and the $\cO$-module $\widetilde{H}^{1}_{\rm f}(G_{K,S}, T^*(1), \Delta_{\Sigma}^*)$ is free of rank $e$. 
\end{lemma}
\begin{proof}
\item[i)] By the finiteness of ideal class groups, we have the following commutative diagram with exact rows: 
\begin{align}
\begin{split}
\label{eqn_ord_selmer_diagram}
    \xymatrix@C=.2in{
0 \ar[r] & e_\chi\left(\cO_L^{\times} \otimes_{\bZ} \Phi \right) \ar[r] \ar[d]^-{\cong} & e_\chi\left(\cO_L[1/p]^{\times} \otimes_{\bZ} \Phi \right) \ar[r]^-{\rm ord} \ar[d] &  \bigoplus_{v \in S_p(K)}H^1_{/{\rm f}}(G_{K_v}, V) \ar[r] \ar@{->>}[d] & 0
\\
0 \ar[r] & \bigoplus_{v \in \Sigma} H^1_{\rm f}(G_{K_v}, V) \ar[r]  & \bigoplus_{v \in \Sigma} H^1(G_{K_v}, V) \ar[r]  & \bigoplus_{v \in \Sigma} H^1_{/{\rm f}}(G_{K_v}, V) \ar[r] & 0. 
}
\end{split}
\end{align}
Here, we have set 
\[
e_{\chi} := \frac{1}{[L \colon k]}\sum_{g \in \Gal(L/k)}\chi(g^{-1})g. 
\]
We note that the left vertical arrow is an isomorphism thanks to the validity of the $\Sigma$-Leopoldt conjecture for $L/K$ (which we have assumed). We therefore have 
\begin{align*}
    \widetilde{H}_{\rm f}^1(G_{K,S}, T, \Delta_{\Sigma}) \otimes_{\bZ_p} \bQ_p 
    &\stackrel{\sim}{\longrightarrow} \ker\left(e_\chi\left(\cO_L[1/p]^{\times} \otimes_{\bZ} \Phi \right) \longrightarrow \bigoplus_{v \in \Sigma} H^1(G_{K_v}, V) \right)
    \\
    &\xrightarrow[{\rm ord}]{\sim} \bigoplus_{v \in \Sigma^c}H^1_{/{\rm f}}(G_{K_v}, T) \otimes_{\bZ_p} \bQ_p 
    \\
    &= \bigoplus_{v \in E(\Sigma^c, \chi)}H^1_{/{\rm f}}(G_{K_v}, T) \otimes_{\bZ_p} \bQ_p,  
\end{align*}
where the first isomorphism follows from \eqref{fundamental_exact_tri2} and Kummer theory, the second from the diagram \eqref{eqn_ord_selmer_diagram} and the final equality is valid by the definition of the set $E(\Sigma^c, \chi)$. This completes the proof of (i).
\item[ii)] By the exact triangle~\eqref{fundamental exact tri}, we have the following exact sequence: 
\begin{align}
\begin{split}\label{align_Fund_exact_Tstar} 
        0 \longrightarrow \bigoplus_{v \in E(\Sigma^c, \chi)}H^0(G_{K_v}, T^*(1)) \longrightarrow &\widetilde{H}_{\rm f}^1(G_{K,S}, T^*(1), \Delta_{\Sigma}^*)
    \\
   &\longrightarrow H^1(G_{K,S}, T^*(1)) 
    \stackrel{\alpha}{\longrightarrow} \bigoplus_{v \in S \setminus \Sigma} H^1(G_{K_v}, T^*(1)). 
    \end{split}
\end{align}
Moreover, it follows from global class field theory and the finiteness of the ideal class group that we have an isomorphism 
\begin{align*}
    \ker(\alpha) \cong 
    \Hom_\cO\left({\rm coker}\left(e_\chi\left(\cO_L^{\times} \otimes_{\bZ} \cO \right) \longrightarrow \bigoplus_{v \in \Sigma}  H^1_{\rm f}(G_{K_v}, T)   \right), \cO \right). 
\end{align*}
As we have explained in Remark~\ref{remark_consequences_of_sigma_leopoldt}(ii), $\Sigma$-Leopoldt conjecture for $L/K$ shows that
\[
{\rm coker}\left(e_\chi\left(\cO_L^{\times} \otimes_{\bZ} \cO \right) \longrightarrow \bigoplus_{v \in \Sigma}  H^1_{\rm f}(G_{K_v}, T)\right)
\]
has finite cardinality. This in turn shows that $\ker(\alpha)=0$. Combining the vanishing of $\ker(\alpha)$ with \eqref{align_Fund_exact_Tstar}, we conclude the proof of (ii).
\end{proof}

\begin{definition}
\label{def_kGamma_AGamma}
For any $\bZ_p$-extension $K_\Gamma$ of $K$, let us set $\Gamma := \Gal(K_\Gamma/K)$. We also put $\bT_\Gamma := \bT_{K_\Gamma}$ and define the augmentation ideal
\[
\sA_\Gamma := \ker\left(\Lambda_{\cO}(\Gamma) \longrightarrow \cO\right) \subset \Lambda_{\cO}(\Gamma)\,.
\]
\end{definition}

\begin{lemma}
\label{lemma:torsion_ciriteria}
If the $\Lambda_{\cO}(\Gamma)$-moudule $\widetilde{H}_{\rm f}^2(G_{K,S}, \bT_{\Gamma},\Delta_{\Sigma})$ is torsion, then no prime in the set $E(\Sigma^c, \chi)$ splits completely in $K_\Gamma$. 
Furthermore, if the $\Sigma$-Leopoldt conjecture for any finite subextension of $K_\Gamma/K$ holds true, then the converse is also true. 
\end{lemma}
\begin{proof}
The first assertion is proved in \cite[Corollary C.11]{BS19}.
Let us prove the second one. 
By Corollary~\ref{corollary:H_f^1 vanishes iff H_f^2 is torsion and fitt=char}~(i), it suffices to show that $\widetilde{H}_{\rm f}^1(G_{K,S}, \bT_{\Gamma},\Delta_{\Sigma}) = 0$. 
For any integer $n \geq 0$, we write $K_n \subset K_\Gamma/K$ for the unique subextension with $[K_n \colon K] = p^n$. 
Since the $\Sigma$-Leopoldt conjecture for $K_n/K$ holds true, our argument in the proof of Lemma~\ref{lemma:rank of selmer}(i) (replacing the number field $L$ with $LK_n$) applies to obtain an injection 
\begin{equation}
    \label{eqn_sigmaleoplodt_alognGamma}
    \widetilde{H}_{\rm f}^1(G_{K,S}, T_n, \Delta_{\Sigma}) \longrightarrow  \bigoplus_{v \in E(\Sigma^c, \chi)}H^1_{/{\rm f}}(G_{K_v}, T_n). 
\end{equation}
Here $T_n := T \otimes_{\bZ_p} \bZ_p[\Gal(K_n/K)]^\iota$. 
Note that, for each prime $v \in E(\Sigma^c, \chi)$, the $\cO$-rank of $H^1_{/{\rm f}}(G_{K_v}, T_n)$ is equal to the number of primes of $K_n$ above $v$. 
Since we assume that no prime in the set $E(\Sigma^c, \chi)$ splits completely in $K_\Gamma$, the number of primes of $K_\Gamma$ above $E(\Sigma^c, \chi)$ is finite. This fact together with the injection~\eqref{eqn_sigmaleoplodt_alognGamma} shows that the $\cO$-module
\[
\widetilde{H}_{\rm f}^1(G_{K,S}, \bT_\Gamma, \Delta_{\Sigma}) = \varprojlim_{n > 0} \widetilde{H}_{\rm f}^1(G_{K,S}, T_n, \Delta_{\Sigma})
\]
is finitely generated. 
As the $\Lambda_{\cO}(\Gamma)$-module $\widetilde{H}_{\rm f}^1(G_{K,S}, \bT_{\Gamma},\Delta_{\Sigma})$ is torsion-free by Proposition \ref{prop:parf[1,2]}, we conclude that $\widetilde{H}_{\rm f}^1(G_{K,S}, \bT_{\Gamma},\Delta_{\Sigma})=0$. As we have noted at the start of our proof, this vanishing is equivalent to the requirement that the $\Lambda_{\cO}(\Gamma)$-module $\widetilde{H}_{\rm f}^2(G_{K,S}, \bT_{\Gamma},\Delta_{\Sigma})$ be torsion. 
\end{proof}

\begin{definition}
Let $A$ be a commutative ring and $M$ a finitely generated $A$-module. 
For a prime ideal $I$ of $A$, we say that $M$ is semi-simple at $I$ if the localization $M_I := M \otimes_A A_I$ of $M$ at $I$ is a semi-simple $A_I$-module. 
\end{definition}

\begin{corollary}\label{corollary:gen by e elements}
Suppose that the $\Sigma$-Leopoldt conjecture for $L/K$ holds true. 
Then the $\Lambda_{\cO}(\Gamma)_{\sA_\Gamma}$-module $\widetilde{H}_{\rm f}^2(G_{K,S}, \bT_{\Gamma}, \Delta_{\Sigma})_{\sA_\Gamma}$ is generated by exactly $e$ elements. 
\end{corollary}
\begin{proof}
Note that the residue filed of $\Lambda_{\cO}(\Gamma)_{\sA_\Gamma}$ is $\Phi$. By the perfect control theorem for Selmer complexes (c.f. \eqref{isom_base_change_H2} above), we have an isomorphism 
\[
\widetilde{H}_{\rm f}^2(G_{K,S}, \bT_{\Gamma},\Delta_{\Sigma})_{\sA_\Gamma} \otimes_{\Lambda_{\cO}(\Gamma)_{\sA_\Gamma}} \Phi  \stackrel{\sim}{\longrightarrow} \widetilde{H}_{\rm f}^2(G_{K,S}, T,\Delta_{\Sigma})\otimes \Phi. 
\]
By Nakayama's Lemma (applied with the local ring $\Lambda_{\cO}(\Gamma)_{\sA_\Gamma}$), it suffices to show that 
$$
\dim_\Phi\left(\widetilde{H}_{\rm f}^2(G_{K,S}, T,\Delta_{\Sigma})\otimes\Phi\right) = e.
$$ 
Recall that we have $\chi(\widetilde{{\bf R}\Gamma}_{\rm f}(G_{K,S}, T,\Delta_{\Sigma})) = 0$ by Lemma~\ref{lemma:Euler characteristic = 0}.  This fact combined with Lemma~\ref{lemma:rank of selmer}(i) shows that 
\[
\dim_\Phi\left(\widetilde{H}_{\rm f}^2(G_{K,S}, T,\Delta_{\Sigma})\otimes\Phi\right) = \dim_\Phi\left(\widetilde{H}_{\rm f}^1(G_{K,S}, T,\Delta_{\Sigma})\otimes\Phi\right) = e\,,
\]
as required.
\end{proof}

\begin{definition}
\label{defn_algebraic_Katz_padic_L_Gamma}
We define the restriction 
$$L_{p,\chi}^{\rm alg}{\vert_\Gamma}\in {\rm Frac}(\LL_{\cO}(\Gamma))^\times\big{/}\LL_{\cO}(\Gamma)^\times \cup \{0\}$$
of the algebraic Katz $p$-adic $L$-function to $\Gamma$ (via the surjection $\Gamma_\infty\twoheadrightarrow \Gamma$) as a generator of 
the cyclic $\LL_{\cO}(\Gamma)$-module
${\rm Fitt}_{\LL_{\cO}(\Gamma)}\left(\widetilde{H}^{2}_{\rm f}(G_{K,S}, \bT_\Gamma, \Delta_{\Sigma})\right)\,.$
\end{definition}

\begin{remark}
\label{remark_algebraic_Katz_Gamma_justified}
To justify the notation $L_{p,\chi}^{\rm alg}{\vert_\Gamma}$ in Definition~\ref{defn_algebraic_Katz_padic_L_Gamma}, we note that
it follows from Proposition~\ref{prop:parf[1,2]} that $L_{p,\chi}^{\rm alg}{\vert_\Gamma}$ is indeed the image of the algebraic Katz $p$-adic $L$-function 
$$L_{p,\chi}^{\rm alg}\in {\rm Frac}(\LL_{\cO}(\Gamma_\infty))^\times\big{/}\LL_{\cO}(\Gamma_\infty)^\times$$ 
(given as in Definition~\ref{defn_algebraic_Katz_padic_L}) under the map
$$ {\rm Frac}(\LL_{\cO}(\Gamma_\infty))^\times\big{/}\LL_{\cO}(\Gamma_\infty)^\times \lra  {\rm Frac}(\LL_{\cO}(\Gamma))^\times\big{/}\LL_{\cO}(\Gamma)^\times \cup \{0\}$$
induced from the natural surjection $\Gamma_\infty\twoheadrightarrow \Gamma$. 
\end{remark}

\begin{corollary}\label{corollary:equiv_semi-simple_fitt}
Suppose that the $\Sigma$-Leopoldt conjecture for $L/K$ holds true. 
The $\Lambda_{\cO}(\Gamma)$-moudule $\widetilde{H}_{\rm f}^2(G_{K,S}, \bT_{\Gamma},\Delta_{\Sigma})$ is semi-simple at $\sA_\Gamma$ if and only if  $L_{p,\chi}^{\rm alg}{\vert_\Gamma} \in \sA_\Gamma^{e} \setminus \sA_\Gamma^{e+1}$. 
\end{corollary}

\begin{proof}
Since $\Lambda_{\cO}(\Gamma)_{\sA_\Gamma}$ is a discrete valuation ring, this is an immediate consequence of Corollary~\ref{corollary:H_f^1 vanishes iff H_f^2 is torsion and fitt=char}(ii) and Corollary~\ref{corollary:gen by e elements}.  
\end{proof}


\subsection{Nekov\'a\v{r}'s $p$-adic height pairing}

In this subsection, we review Nekov\'a\v{r}'s description in \cite[\S11]{Nek} of $p$-adic height pairings and establish a relation (in Lemma~\ref{lemma:equiv_non-degen_fitt} below) between the order of exceptional zeros of an algebraic $p$-adic $L$-function and the non-degeneracy of Nekov\'a\v{r}'s $p$-adic height pairings. 

As in Definition~\ref{def_kGamma_AGamma}, we fix a $\bZ_p$-extension $K_\Gamma$ of $K$. We then have an exact sequence
\begin{equation}
\label{eqn_bockstein_exact_seq}
    0 \longrightarrow \bT_\Gamma \otimes_{\Lambda_{\cO}(\Gamma)} \sA_\Gamma/\sA_\Gamma^2 \longrightarrow \bT_\Gamma/\sA_\Gamma^2  \longrightarrow T \longrightarrow 0. 
\end{equation}
Observe that we have 
\[
\bT_\Gamma \otimes_{\Lambda_{\cO}(\Gamma)} \sA_\Gamma/\sA_\Gamma^2 = T \otimes_{\cO} \sA_\Gamma/\sA_\Gamma^2
\]
as Galois modules and since the action of $G_{K}$ on $\sA_\Gamma/\sA_\Gamma^2$ induced by the tautological character $G_{K} \twoheadrightarrow \Gamma \hookrightarrow \Lambda_{\cO}(\Gamma)^\times$ is trivial, 
\[
C^\bullet(G, T \otimes_{\cO} \sA_\Gamma/\sA_\Gamma^2) = C^\bullet(G, T)  \otimes_{\cO} \sA_\Gamma/\sA_\Gamma^2
\]
for any closed subgroup $G$ of $G_{K,S}$. 
This combined with the exact sequence \eqref{eqn_bockstein_exact_seq} gives rise to a Bockstein morphism 
\[
\beta \colon \widetilde{{\bf R}\Gamma}_{\rm f}(G_{K,S}, T, \Delta_\Sigma) \longrightarrow \widetilde{{\bf R}\Gamma}_{\rm f}(G_{K,S}, T, \Delta_\Sigma)[1] \otimes_\cO \sA_\Gamma/\sA_\Gamma^2. 
\]
Furthermore, it follows from Nekov\'a\v{r}'s global duality (see \cite{Nek}, \S6.3), the cup product pairing induces a morphism 
\[
\cup \colon \widetilde{{\bf R}\Gamma}_{\rm f}(G_{K,S}, T, \Delta_\Sigma) \otimes \widetilde{{\bf R}\Gamma}_{\rm f}(G_{K,S}, T^*(1), \Delta_\Sigma^*) \longrightarrow \cO[-3]. 
\]
\begin{definition}
\label{defn_nekovars_padic_height}
Nekov\'a\v{r}'s $p$-adic height pairing $\langle\,\,, \,\,\rangle_{\Gamma}$ is given as the compositum of the arrows
\begin{align*}
\widetilde{H}^1_{\rm f}(G_{K,S}, T, \Delta_\Sigma) \otimes  \widetilde{H}^1_{\rm f}(G_{K,S}, T^*(1), \Delta_\Sigma^*) 
\xrightarrow{\beta \otimes {\rm id}} 
(&\widetilde{H}^2_{\rm f}(G_{K,S}, T, \Delta_\Sigma)  \otimes_\cO \sA_\Gamma/\sA_\Gamma^2)  \otimes  \widetilde{H}^1_{\rm f}(G_{K,S}, T^*(1), \Delta_\Sigma^*)  
\\ 
&\stackrel{\cup}{\longrightarrow} \sA_\Gamma/\sA_\Gamma^2\,\,. 
\end{align*}
\end{definition}

For each prime $v$ of $K$, we similarly have a Bockstein morphism 
\[
\beta_v \colon {\bf R}\Gamma(G_{K_v}, T) \longrightarrow {\bf R}\Gamma(G_{K_v}, T)[1] \otimes_\cO \sA_\Gamma/\sA_\Gamma^2. 
\]
The local Tate pairing at $v$ induces the pairing $\langle\,\,,\,\,\rangle_{\Gamma, v}$ given as the compositum
\begin{align*}
    H^1(G_{K_v}, T) \otimes H^0(G_{K_v}, T^*(1)) 
    \xrightarrow{\beta_v \times {\rm id}} (H^2(G_{K_v}, T) \otimes_\cO \sA_\Gamma/\sA_\Gamma^2) \otimes H^0(G_{K_v}, T^*(1)) 
     \xrightarrow{{\rm inv}_v} 
    \sA_\Gamma/\sA_\Gamma^2\,. 
\end{align*}

\begin{lemma}
\label{lemma:commutative}
For  each prime  $v \in \Sigma^c$, the following diagram commutes: 
$$\xymatrix@C=5pt{
     \widetilde{H}^1_{\rm f}(G_{K,S}, T, \Delta_\Sigma) \ar[d]^{{\rm res}_v} & \otimes &  \widetilde{H}^1_{\rm f}(G_{K,S}, T^*(1), \Delta_\Sigma^*) \ar[rrrr]^(.65){\langle\,\,, \,\,\rangle_{\Gamma}} &&&& \sA_\Gamma/\sA_\Gamma^2 \ar@{=}[d]
     \\
    H^1(G_{K_v}, T) &\otimes& H^0(G_{K_v}, T^*(1)) \ar[u] \ar[rrrr]^(.6){\langle\,\,, \,\,\rangle_{\Gamma, v}} &&&& \sA_\Gamma/\sA_\Gamma^2\,. 
}$$
Here, the left vertical map is the compositum of the arrows
\[
\widetilde{H}^1_{\rm f}(G_{K,S}, T, \Delta_\Sigma) \longrightarrow H^1(G_{k, S}, T)  
 \xrightarrow{\res_v} 
H^1(G_{K_v}, T). 
\]
\end{lemma}

\begin{proof}
By the definition of Bockstein maps $\beta$ and $\beta_v$, the diagram 
\[
\xymatrix{
\widetilde{H}^1_{\rm f}(G_{K,S}, T, \Delta_\Sigma) \ar[d]  \ar[r]^-{\beta} & \widetilde{H}^2_{\rm f}(G_{K,S}, T, \Delta_\Sigma) \otimes_\cO \sA_\Gamma/\sA_\Gamma^2 \ar[d] 
\\
H^1(G_{K_v}, T) \ar[r]^-{\beta_v} & H^2(G_{K_v}, T) \otimes_\cO \sA_\Gamma/\sA_\Gamma^2
}
\]
commutes. Moreover, applying \cite[Proposition~6.3.3]{Nek} with $X_{1} = T$, $\Delta(X_{1}) = \Delta_{\Sigma}$, $X_{2} = T^{*}(1)$, and $\Delta(X_{2}) = \Delta_{\Sigma}^{*}$, we obtain the following  commutative diagram:
\[
\xymatrix{
\widetilde{{\bf R}\Gamma}_{\rm f}(G_{K,S}, T, \Delta_\Sigma) \ar[d]  \ar[r] & {\bf R}{\rm Hom}_\cO(\widetilde{{\bf R}\Gamma}_{\rm f}(G_{K,S}, T^*(1), \Delta_\Sigma^*), \cO[-3]) \ar[d] 
\\
\bigoplus_{v \in \Sigma^{c}}{\bf R}\Gamma(G_{K_v}, T) \ar[r] &\bigoplus_{v \in \Sigma^{c}} {\bf R}{\rm Hom}_\cO({\bf R}\Gamma(G_{K_v}, T^*(1))[-1], \cO[-3]). 
}
\]
Considering the second cohomology groups of the above diagram, we get a commutative diagram 
\[
\xymatrix{
\widetilde{H}^2_{\rm f}(G_{K,S}, T, \Delta_\Sigma) \ar[d]  \ar[r] & \Hom_\cO(\widetilde{H}^1_{\rm f}(G_{K,S}, T^*(1), \Delta_\Sigma^*), \cO) \ar[d] 
\\
H^2(G_{K_v}, T) \ar[r] & \Hom_\cO({H}^0(G_{K_v}, T^*(1)), \cO). 
}
\]
Since both horizontal arrows are induced from the cup product pairings, the proof of our lemma follows combining these two diagrams.
\end{proof}

\begin{lemma}
\label{lemma:equiv_non-degen_fitt}
Suppose that the $\Sigma$-Leopoldt Conjecture~\ref{conj_sigma_leopoldt} for $L/K$ holds true. 
Then Nekov\'a\v{r}'s $p$-adic height pairing $\langle\,\,, \,\,\rangle_{\Gamma}$ is non-degenerate if and only if  $L_{p,\chi}^{\rm alg}{\vert_\Gamma} \in \sA_\Gamma^e \setminus \sA_\Gamma^{e+1}$. 
\end{lemma}
\begin{proof}
Let us set $R := \Lambda_{\cO}(\Gamma)_{\sA_{\Gamma}}$. 
Proposition~11.7.6(vii) of \cite{Nek} (applied with $\fq = (0)$ and $\overline{\fq} = \sA_\Gamma$) tells us that the $p$-adic height pairing $\langle\,\,, \,\,\rangle_{\Gamma}$ is non-degenerate if and only if 
$${\rm length}_{R}(\widetilde{H}_{\rm f}^2(G_{K,S}, \bT_{\Gamma},\Delta_{\Sigma}) \otimes_{\Lambda_{\cO}(\Gamma)} R) = \dim_{\Phi}(\widetilde{H}^{1}_{\rm f}(G_{K,S}, T, \Delta_{\Sigma})\otimes_{\cO}\Phi)\,.$$ 
By Lemma~\ref{lemma:rank of selmer}(i), this equivalent to the requirement that
$${\rm length}_{R}(\widetilde{H}_{\rm f}^2(G_{K,S}, \bT_{\Gamma},\Delta_{\Sigma}) \otimes_{\Lambda_{\cO}(\Gamma)} R) = e.$$
By Corollary~\ref{corollary:gen by e elements}, this is equivalent to the condition that
\[
L_{p,\chi}^{\rm alg}{\vert_\Gamma} \in  \sA_\Gamma^e \setminus \sA_\Gamma^{e+1}
\]
as required. 
\end{proof}


\subsection{Group-ring-valued $\cL$-invariants}

In this subsection, we shall recall from \cite{BS19} the definition of the group-ring-valued $\cL$-invariant  $\cL_{\Sigma, \Gamma}^{\rm Gal}$ along each $\bZ_p$-extension $K_\Gamma/K$ given as in Definition~\ref{def_kGamma_AGamma}. We remind the reader that the leading term formula in \cite[Theorem 1.1]{BS19} for the restriction $L_{p,\Sigma}^\chi\big{\vert}_{\Gamma}$ of the Katz $p$-adic $L$-function to $\Gamma$  in the present set up involves, among other things, this group-ring-valued $\cL$-invariant. Our main goal in this section is provide an alternative description of $\cL_{\Sigma, \Gamma}^{\rm Gal}$ in terms of the $p$-adic height pairing $\langle\,\,,\,\,\rangle_\Gamma$.  In Section~\ref{subsec_universal_L_invariant}, we shall rely on this new description of  $\cL_{\Sigma, \Gamma}^{\rm Gal}$  to interpolate them (as $\Gamma$ varies) to the universal group-ring-valued $\cL$-invariant ${ \cL_{\Sigma}^{\rm Gal}}$.

\begin{definition}
Suppose $v \in S_p(K)$. 
\item[i)] Fix a prime $w$ of the extension $L = \overline{K}^{\ker(\chi)}$ that lies above $v$. We define the map ${\rm ord}_v$ as the composite 
\[
H^1(G_{K,S}, T) \longrightarrow H^1(G_{K_v}, T) \longrightarrow H^1(G_{L_w}, T) = \widehat{L_w^{\times}} \otimes_{\bZ_p} \cO \xrightarrow{{\rm ord}_w} \cO. 
\]
We recall that $\widehat{L_w^{\times}}$ stands for the $p$-adic completion of $L_w^{\times}$. 
\item[ii)] Using the reciprocity map, we define ${\rm rec}_{\Gamma, w}$ as the compositum of the arrows
\[
 \widehat{L_w^{\times}} \otimes_{\bZ_p} \cO \longrightarrow 
\Gal(L_w^{\rm ab}/L_w) \widehat{\otimes}_{\bZ_p} \cO \longrightarrow  \Gamma \otimes_{\bZ_p} \cO \stackrel{\sim}{\longrightarrow} \sA_\Gamma/\sA_{\Gamma}^2, 
\]
where the isomorphism $\Gamma \otimes_{\bZ_p} \cO \to \sA_\Gamma/\sA_{\Gamma}^2$ is given by $\gamma \mapsto \gamma -1$. We define the homomorphism ${\rm rec}_{\Gamma, v}$ as the composite 
\[
H^1(G_{K,S}, T) \longrightarrow H^1(G_{K_v}, T) \longrightarrow H^1(G_{L_w}, T)  \xrightarrow{{\rm rec}_{\Gamma, w}}  \sA_\Gamma/\sA_\Gamma^2\,. 
\]
\end{definition}

\begin{remark}\label{remark_order_rec}
\item[i)] {The maps ${\rm ord}_v$ and ${\rm rec}_{\Gamma, v}$ do depend the choice of the prime $w$ of $L$ above the prime $v$.}
\item[ii)] If $K_\Gamma/K$ is unramified at $v \in S_p(K)$, then 
\[
{\rm rec}_{\Gamma,v}(x) = {\rm ord}_{v}(x) \cdot ({\rm Frob}_v - 1) \in \sA_\Gamma/\sA_\Gamma^2\,.  
\]
\end{remark}

The following definition of the group-ring-valued $\cL$-invariant  dwells on \cite[Lemma~5.4]{BS19}. 

\begin{definition}
Suppose that the $\Sigma$-Leopoldt Conjecture~\ref{conj_sigma_leopoldt} for $L/K$ holds true and $e \geq 1$. By Lemma~\ref{lemma:rank of selmer}(i), 
\[
0 \neq \bigwedge_{v \in E(\Sigma^c, \chi)}{\rm ord}_v \in \det\left( \Hom_\cO(\widetilde{H}_{\rm f}^1(G_{K,S}, T,\Delta_{\Sigma}), \cO)\right). 
\]
\item[i)] We define the group-ring-valued $\cL$-invariant  $\cL_{\Sigma, \Gamma}^{\rm Gal} \in \bQ_p \otimes_{\bZ_p} \sA_\Gamma^e/\sA_\Gamma^{e+1}$ along $\Gamma$ as the unique element that verifies 
\[
\bigwedge_{v \in E(\Sigma^c, \chi)}{\rm rec}_{\Gamma, v} = \cL_{\Sigma, \Gamma}^{\rm Gal} \otimes \bigwedge_{v \in E(\Sigma^c, \chi)}{\rm ord}_v\,.
\]
Here, the equality takes place in 
\[
\det\left( \Hom_\cO(\widetilde{H}_{\rm f}^1(G_{K,S}, T,\Delta_{\Sigma}), \sA_\Gamma^e/\sA_\Gamma^{e+1})\right) \otimes_{\bZ_p} \bQ_p. 
\]
\item[ii)] When $e = 0$, we put $\cL_{\Sigma,\Gamma}^{\rm Gal}  = 1$. 
\end{definition}

\begin{lemma}\label{lemma:non-vanishing}
There is a $\bZ_p$-extension $K_\Gamma/K$ with Galois group $\Gamma$ such that  $\cL_{\Sigma, \Gamma}^{\rm Gal} \neq 0$. 
\end{lemma}
\begin{proof}
If $K_\Gamma/K$ is unramified outside $\Sigma$, then we have 
\[
\cL_{\Sigma, \Gamma}^{\rm Gal} = \prod_{v \in E(\Sigma^c, \chi)} ({\rm Frob}_v - 1) \in \sA_\Gamma^e. 
\]
It therefore suffices to prove the existence of a $\bZ_p$-extension $K_\Gamma/K$ which is unramified outside $\Sigma$ and such that no prime in $E(\Sigma^c,\chi)$ splits completely in $K_\Gamma$. 

By the ad\`elic description of class field theory, any $\bZ_p$-extension $K_\Gamma/K$ which is  unramified outside $\Sigma$ corresponds to a subspace $X \subset \prod_{v \in \Sigma}{\widehat{\cO_{K_v}^{\times}}} \otimes_{\bZ_p} \bQ_p$ of dimension $[K \colon \bQ] -1$ containing the image of $\cO_K^\times \otimes_{\bZ} \bQ_p$. Let $h_K$ denote the class number of $K$. 
For a prime $v^c \in \Sigma^c$, let $\pi_{v^c}$ be a generator of the principal ideal $\fp_{v^c}^{h_K}$, where $\fp_{v^c}$ is the prime ideal of $\cO_K$ corresponding $v^c$. Then by class field theory, $v^c$ splits completely in the extension $K_\Gamma/K$ if and only if $\tilde{\pi}_{v^c} \in X$. We therefore need to prove the existence of a subspace $X \subset \prod_{v \in \Sigma}{\widehat{\cO_{K_v}^{\times}}} \otimes_{\bZ_p} \bQ_p$ of dimension $[K \colon \bQ] -1$ such that
\begin{itemize}
    \item $X$ contains the image of $\cO_K^\times \otimes_{\bZ} \bQ_p$;
    \item $\pi_{v^c} \not\in X$ for any $v^c \in \Sigma^c$.
\end{itemize}
Since $\pi_{v^c}$ is not contained in the image of $\cO_K^\times \otimes_{\bZ} \bQ_p$, there always exists a subspace $X$ with the required properties. 
\end{proof}

\begin{definition}
\label{defn_evfv_preheight}
\item[i)] For each prime $v \in E(\Sigma^c,\chi)$, via the fixed prime $w$ of $L$ above $v$, we have an isomorphism 
\[
\cO\stackrel{\sim}{\longrightarrow} H^0(G_{K_v}, T^*(1)). 
\]
Let $q_v\in H^0(G_{K_v}, T^*(1))$ denote the image of $1\in \cO$ under this isomorphism.
\item[ii)] Assuming the truth of the $\Sigma$-Leopoldt conjecture for $L/K$, we let
$$\{e_v\,:\,{v \in E(\Sigma^c, \chi)}\}\subset \widetilde{H}^1_{\rm f}(G_{K,S}, T, \Delta_\Sigma) \otimes_{\bZ_p} \bQ_p$$  
be a basis which verifies 
\[
{\rm ord}_{v'}(e_v) = 
\begin{cases}
1  & \, \textrm{ if } \, v' = v, 
\\
0  & \, \textrm{ if } \, v' \neq v. 
\end{cases}
\]
A basis $\{e_v\,:\,{v \in E(\Sigma^c, \chi)}\}$ with this property exists thanks to Lemma~\ref{lemma:rank of selmer}(i).
\end{definition}

\begin{proposition}\label{proposition:equiv_non-vanish_non-degen}
Suppose that the $\Sigma$-Leopoldt Conjecture~\ref{conj_sigma_leopoldt} for $L/K$ holds true. 
Then we have 
\[
\cL_{\Sigma, \Gamma}^{\rm Gal} = (-1)^e \det\left(\langle e_v, q_{v'} \rangle_{\Gamma}\,\,:\,\,v,v' \in E(\Sigma^c, \chi)\right)\,.
\]
In particular, the group-ring-valued $\cL$-invariant  $\cL_{\Sigma, \Gamma}^{\rm Gal}$ is non-vanishing if and only if  
the Nekov\'a\v{r}'s $p$-adic height pairing $\langle\,\,,\,\,\rangle_{\Gamma}$ is non-degenerate. 
\end{proposition}
\begin{proof}
By the definition of $\cL_{\Sigma, \Gamma}^{\rm Gal}$, we have 
\[
\cL_{\Sigma, \Gamma}^{\rm Gal} =  \det
\left({\rm rec}_{\Gamma, v}(e_{v'})\,\,:\,\,{v, v' \in E(\Sigma^c, \chi)}\right). 
\]
It therefore suffices to check that ${\rm rec}_{\Gamma, v}(e_{v'}) = -\langle e_{v'}, q_v \rangle_{\Gamma}.$ 
By Lemma~\ref{lemma:commutative}, we have 
\[
\langle e_{v'}, q_v \rangle_{\Gamma} = \langle {\rm res}_v(e_{v'}), q_v \rangle_{\Gamma, v}. 
\]
Let $z_v \colon G_{K_v} \to \Gamma$ denote the tautological $1$-cocyle. By \cite[Lemma 11.2.3]{Nek}, we have 
\[
\beta_v({\rm res}_v(e_{v'})) = -z_v \cup {\rm res}_v(e_{v'}). 
\]
On the other hand, the formula \cite[(11.3.5.3)]{Nek} shows that 
\[
{\rm rec}_{\Gamma, v}(e_{v'}) = {\rm inv}_v(z_v \cup {\rm res}_v(e_{v'})). 
\]
Here ${\rm inv}_v \colon H^2(G_{K_v}, T) \stackrel{\sim}{\longrightarrow} H^2(G_{L_w}, T) \stackrel{\sim}{\longrightarrow} \cO$ denotes the invariant map. 
Thence, 
\[
\langle {\rm res}_v(e_{v'}), q_v \rangle_{\Gamma, v} = 
{\rm inv}_v(-z_v \cup {\rm res}_v(e_{v'}) \cup q_v) = 
{\rm inv}_v(-z_v \cup {\rm res}_v(e_{v'})) = - {\rm rec}_{\Gamma, v}(e_{v'}),  
\]
where the second equality follows from the definition of $q_v$ as the image of $1 \in \cO$. This concludes the proof of our proposition. 
\end{proof}

\begin{theorem}
\label{thm_nonvanishingL_semisimple_nondegpadicheight}
Assuming the truth of the $\Sigma$-Leopoldt Conjecture~\ref{conj_sigma_leopoldt} for $L/K$, the following are equivalent. 
    \item[a)] The Nekov\'a\v{r}'s $p$-adic height pairing $\langle\,\,, \,\,\rangle_{\Gamma}$ is non-degenerate. 
    \item[b)] The algebraic Katz $p$-adic $L$-function $L_{p,\chi}^{\rm alg}{\vert_\Gamma}$ belongs to  $\sA_\Gamma^e \setminus \sA_\Gamma^{e+1}$.  
    \item[c)] The $\Lambda_{\cO}(\Gamma)$-moudule $\widetilde{H}_{\rm f}^2(G_{K,S}, \bT_{\Gamma},\Delta_{\Sigma})$ is semi-simple at $\sA_\Gamma$. 
    \item[d)] The group-ring-valued $\cL$-invariant  $\cL_{\Sigma, \Gamma}^{\rm Gal}$ does not vanish. 
\end{theorem}
\begin{proof}
This follows on combining Corollary~\ref{corollary:equiv_semi-simple_fitt}, Lemma~\ref{lemma:equiv_non-degen_fitt}, and Proposition~\ref{proposition:equiv_non-vanish_non-degen}. 
\end{proof}

Recall that $\sA:=\ker\left(\Lambda_{\cO}(\Gamma_\infty)\to \cO\right)$ is the augmentation ideal of $\Lambda_{\cO}(\Gamma_\infty)$. 

\begin{corollary}
\label{cor_algebraic_multivariate_exceptional_zero}
Suppose that the $\Sigma$-Leopoldt conjecture for $L/K$ holds true. Then we have 
\[
L_{p,\chi}^{\rm alg} \in \sA^e \setminus \sA^{e+1}. 
\]
\end{corollary}
\begin{proof}
The containment $L_{p,\chi}^{\rm alg} \in \sA^e$ follows from Proposition \ref{proposition_vanishing_order_e}. To prove that $L_{p,\chi}^{\rm alg} \not\in \sA^{e+1}$, it suffices to find a $\ZZ_p$-extension $K_{\Gamma}/K$ with Galois group $\Gamma$ such that ${L_{p,\chi}^{\rm alg} }{\vert_\Gamma}\not\in \sA_{\Gamma}^{e+1}$. Such $\Gamma$ exists thanks to Theorem~\ref{thm_nonvanishingL_semisimple_nondegpadicheight} and Lemma~\ref{lemma:non-vanishing}.
\end{proof}

\begin{remark} 
For a given $\bZ_p$-extension $K_\Gamma /K$ with Galois group $\Gamma$, we crucially use  the fact that $\Lambda_{\cO}(\Gamma)_{\sA_\Gamma}$ is a discrete valuation ring in our proof that 
$$ L_{p,\chi}^{\rm alg}{\vert_\Gamma} \in \sA_\Gamma^e \setminus \sA_\Gamma^{e+1} {\,\,}\implies{\,\,} \hbox{ the }\Lambda_{\cO}(\Gamma)\hbox{-module } \widetilde{H}_{\rm f}^2(G_{K,S}, \bT_{\Gamma},\Delta_{\Sigma}) \hbox{ is semi-simple at } \sA_\Gamma\,.$$
Our argument therefore does not generalize to conclude that the module 
$\widetilde{H}^{2}_{\rm f}(G_{K,S}, \bT_\infty, \Delta_{\Sigma})$ is semi-simple at $\sA$. 
\end{remark}


\subsection{Interpolation of $p$-adic heights and the universal group-ring-valued $\cL$-invariant }
\label{subsec_universal_L_invariant}
As an application of our interpretation of the group-ring-valued $\cL$-invariants  in terms of Nekov\'a\v{r}'s $p$-adic height pairings, we may interpolate them to define what we call the universal group-ring-valued $\cL$-invariant.

Nekov\'a\v{r}'s formalism of $p$-adic height pairings give rise to a multivariate $p$-adic height pairing 
\begin{align*}
    \langle\,\,, \,\,\rangle_{\Gamma_\infty} \colon \widetilde{H}^1_{\rm f}(G_{K,S}, T, \Delta_\Sigma) \times  \widetilde{H}^1_{\rm f}(G_{K,S}, T^*(1), \Delta_\Sigma^*) \longrightarrow \sA/\sA^2. 
\end{align*}
when applied with $\Gamma_\infty$ in place of $\Gamma$ in Definition~\ref{defn_nekovars_padic_height}. Let  $K_\Gamma/K$ be a $\bZ_p$-extension with Galois group $\Gamma$. The commutative diagram with exact rows
\[
\xymatrix{
0 \ar[r] & \bT_\infty \otimes_{\Lambda_{\cO}(\Gamma_{\infty})} \sA/\sA^2 \ar[r] \ar[d] &  \bT_\infty/\sA^2 \bT_\infty \ar[d] \ar[r] & T \ar[d] \ar[r] & 0
\\
0 \ar[r] & \bT_\Gamma \otimes_{\Lambda_{\cO}(\Gamma)} \sA_\Gamma/\sA^2_\Gamma \ar[r] &  \bT_\Gamma/\sA^2_\Gamma \bT_\Gamma \ar[r] & T \ar[r] & 0 
}
\]
shows that the diagram 
\[
\xymatrix{
\widetilde{{\bf R}\Gamma}_{\rm f}(G_{K,S}, T, \Delta_\Sigma) \ar[r]^-{\beta} \ar[d] & \ar[d] \widetilde{{\bf R}\Gamma}_{\rm f}(G_{K,S}, T, \Delta_\Sigma)[1] \otimes_\cO \sA/\sA^2
\\
\widetilde{{\bf R}\Gamma}_{\rm f}(G_{K,S}, T, \Delta_\Sigma) \ar[r]^-{\beta} & \widetilde{{\bf R}\Gamma}_{\rm f}(G_{K,S}, T, \Delta_\Sigma)[1] \otimes_\cO \sA_\Gamma/\sA_\Gamma^2
}
\]
commutes. 
By the construction of the $p$-adic height pairings (due to Nekov\'a\v{r}), we conclude that:
\begin{lemma}\label{lem:conpatible}
 For any $\bZ_p$-extension $K_\Gamma/K$ with Galois group $\Gamma$, 
the pairing 
\[
 \widetilde{H}^1_{\rm f}(G_{K,S}, T, \Delta_\Sigma) \times  \widetilde{H}^1_{\rm f}(G_{K,S}, T^*(1), \Delta_\Sigma^*) \stackrel{\langle\,\,, \,\,\rangle_{\Gamma_\infty}}{\longrightarrow} \sA/\sA^2 \longrightarrow \sA_\Gamma/\sA^2_\Gamma
\]
coincides with the $p$-adic height pairing  $\langle\,\,, \,\,\rangle_{\Gamma}$. 
\end{lemma}
\begin{definition}
\label{defn_multivariate_L_invariant}
Suppose that the $\Sigma$-Leopoldt conjecture holds true. 
Let $\{e_v\}$ and $\{q_v\}$ be as in Definition~\ref{defn_evfv_preheight}.  We set
\begin{equation}
\label{eqn_universal_L_invariant}
   { \cL_{\Sigma}^{\rm Gal}} := (-1)^e \det\left(\langle e_v, q_{v'} \rangle_{\Gamma_\infty}\,\,:\,\,{v,v' \in E(\Sigma^c, \chi)}\right) \in \sA^e/\sA^{e+1} \otimes_{\bZ_p} \bQ_p\,.
\end{equation}
We call ${ \cL_{\Sigma}^{\rm Gal}}$ the universal group-ring-valued $\cL$-invariant.
\end{definition}
The first part of the following proposition justifies the adjective ``universal'' attached to the group-ring-valued $\cL$-invariant  ${ \cL_{\Sigma}^{\rm Gal}}$.

\begin{proposition}
Suppose that the $\Sigma$-Leopoldt Conjecture~\ref{conj_sigma_leopoldt} for $L/K$ holds true. 
    \item[i)] For any $\bZ_p$-extension $K_\Gamma/K$ with Galois group $\Gamma$, the image of ${ \cL_{\Sigma}^{\rm Gal}}$ under the canonical map 
    \[
    \sA^e/\sA^{e+1} \otimes_{\bZ_p}   \bQ_p \longrightarrow \sA_{\Gamma}^e/\sA_{\Gamma}^{e+1} \otimes_{\bZ_p} \bQ_p
    \]
    equals $\cL_{\Sigma, \Gamma}^{\rm Gal}$. 
    \item[ii)] We have ${ \cL_{\Sigma}^{\rm Gal}} \neq 0$ and the pairing $\langle\,\,, \,\,\rangle_{\Gamma_\infty}$ is non-degenerate.
\end{proposition}
\begin{proof}
The first assertion follows from Lemma~\ref{lem:conpatible} and the second part is immediate from the first combined with Lemma \ref{lemma:non-vanishing} and Lemma \ref{lem:conpatible}.
\end{proof}


\subsection{Katz' $p$-adic $L$-function and non-critical exceptional zeros}
\label{subsec_Katz_padic_L_analytic_non_critical_exceptional_zeros}

We retain our notation from \S\ref{subsection:example-cm} and continue assuming the hypothesis~\ref{item_ord}. 
We recall that 
\begin{itemize}
\item $K(p^{\infty})$ denotes the compositum of all $\bZ_{p}$-extensions of $K$,
\item we have fixed a character $\chi \colon G_{K} \to \overline{\bQ}^{\times}$ which has finite prime-to-$p$ order and set $\cO:= \bZ_p[{\rm im}(\chi)]$ (note that $\cO \subset \bZ_p^{\rm ur}$ since the order of $\chi$ is coprime to $p$),
\item we have put $\Gamma_\infty := \Gal(K(p^{\infty})/K) \ \text{ and } \ \bT_\infty := \cO(1) \otimes \chi^{-1} \otimes_{\cO} \cO[[\Gamma_\infty]].$ 
\end{itemize}

Let us denote by $\frak{c}$ the conductor of the ray class character that one associates to $\chi$ by class field theory. We assume that $(\frak{c},p)=1$.
\begin{definition}
Let $I_K$ denote the set of embeddings $K\hookrightarrow \overline{\QQ}$. Via our fixed embedding $\iota_p$, we can identify $I_K$ with $S_p(K)$. Likewise, via our fixed embedding $\iota_\infty$, we can identify $I_K$ with the set of embeddings $K \hookrightarrow \mathbb{C}$. We define the CM type $\Sigma_\infty$ as the set of embeddings of $K$ into $\mathbb{C}$ which correspond to the $p$-adic CM type $\Sigma$. In more precise terms, $\Sigma_\infty=\{\iota_\infty\circ\sigma: \sigma \in I_K \hbox{ with } \iota_p\circ \sigma \in \Sigma_p\}$ and we have $\Sigma_\infty\sqcup \Sigma_\infty^c=S_\infty(K)$.
\end{definition}

\begin{definition}
Let $K(\frak{c}p^\infty)=\cup_r K(\frak{c}p^r)$ be the union of all ray class extensions of $K$ modulo $\frak{c}p^r$, for $r\in \ZZ^+$. We set $\bG_\frak{c}:=\Gal(K(\frak{c}p^\infty)/K)$ and write $\Delta_{\frak c}$ for the torsion-subgroup of $\bG_{\frak{c}}$. Then the free part $W_K:=\bG_{\frak{c}}/\Delta_{\frak c}$ of $\bG_\frak{c}$ is independent of $\frak{c}$ and by class field theory, it is isomorphic to $\ZZ_p^{[F:\QQ]+1+\delta}$, where $\delta$ is the Leopoldt defect. We note that the canonical map $\bG_\frak{c}\to \Gamma_\infty$ gives rise to a natural isomorphism $W_K\xrightarrow{\sim} \Gamma_\infty$; we identify these two groups via this isomorphism. We have introduced and work with $W_K$ solely for consistency with the notation of Hida and Tilouine in \cite{HT93} that we greatly dwell on throughout \S\ref{subsec_Katz_padic_L_analytic_non_critical_exceptional_zeros}.
\end{definition}

Under the running ordinarity hypothesis \ref{item_ord}, there exists a $p$-adic $L$-function 
\[
L_{p,\chi}^{\rm Katz} \in \bZ_{p}^{\rm ur}[[\Gamma_{\infty}]]
\] 
(see \cite{Katz78} and \cite{HT93}, Theorem~II) which is  characterised by the following interpolative property: For any $A_0$-type Hecke character $\epsilon$ whose associated $p$-adic Galois character $\widehat{\epsilon}$ factors through $\Gamma_\infty$ and whose infinity type is of the form $m\Sigma_\infty+d(1-c) \in\ZZ[I_K]$, where  $d=\sum_{\sigma\in \Sigma_\infty}d_\sigma \sigma\in \ZZ[\Sigma_\infty]$ and $m\in \ZZ$ are such that either $m>0$ and $d\geq 0$ or $m\leq 1$ and $d_\sigma\geq 1-m$, we have (with the conventions set following the statement of  \cite[Theorem~II]{HT93})
\begin{align}
\label{Eqn_Katz_padic_interpolation}
\frac{\widehat \epsilon (L_{p,\Sigma}^\chi)}{\Omega_p^{m\Sigma_\infty+2d}} = 
{[\cO_K^\times:\cO_{F}^\times]}\, W_p(\psi) &\times\,\frac{(-1)^{m\Sigma_\infty}\pi^d\prod_{\sigma \in \Sigma_\infty}\Gamma (m+d_\sigma)}{{\rm Im}(\delta)^d \sqrt{|D_{F}|_{\mathbb{R}}}}\\ 
\notag &\times\, \prod_{\frak{l}\mid \frak{c}}(1-\psi(\frak{l})) \prod_{\frak{p}\in \Sigma} \left(1-\frac{\psi^{-1}(\fp)}{\mathbb{N}\fp}\right)\left(1-\psi(\fp^c)\right))\times 
\frac{L(\psi,0)}{\Omega_\infty^{m\Sigma_\infty+2d}}, 
\end{align}
where 
\begin{itemize}
\item $W_p(\psi)$ is the local $\varepsilon$-factor given as in \cite[(0.10)]{HT93},
\item $\delta$ is chosen to verify \cite{HT93}, (0.9a) and (0.9b),
\item $D_{F}$ is the absolute discriminant of $F$,
\item $\psi$ is the Hecke character $\chi^{-1}\epsilon$,
\item $\Omega_p$ and $\Omega_\infty$ are the $p$-adic and archimedean CM periods (which we shall not define here)\,.
\end{itemize}

\subsubsection{Exceptional zeros}
Recall that we have set
\[
e := \# \{v \in \Sigma^{c} \mid \chi(G_{K_{v}}) = \{1\}\}.
\]
In the interpolation formula \eqref{Eqn_Katz_padic_interpolation}, precisely $e$ of the factors $\{1-\psi_v(\varpi_v): {v\in \Sigma^c}\}$ vanish when $\epsilon=\mathds{1}$. Despite the fact that the trivial character $\mathds{1}$ is not in the range of interpolation for Katz' $p$-adic $L$-function, Perrin-Riou's philosophy leads one to the following question (which we think of as a multivariate variant of the Exceptional Zero Conjecture): 

\begin{question}\label{conj:ex zero}
Let $\sA_{\bZ_p^{\rm ur}}$ denote the augmentation ideal of $\Lambda_{\bZ_p^{\rm ur}}(\Gamma_{\infty})$. 
Then is it true that 
\[
  L_{p,\chi}^{\rm Katz}  \in \sA^{e}_{\bZ_p^{\rm ur}} \setminus \sA^{e+1}_{\bZ_p^{\rm ur}}\,\,?
\]
\end{question} 

We have the following results towards an answer to Question~\ref{conj:ex zero}.

\begin{theorem}
\label{thm_exceptional_zero_conj_OK}
If the $\Sigma$-Leopoldt conjecture for $L/K$ and the Iwasawa Main Conjecture~\ref{conj:iwasawa main conj} for $K$ hold true, then $L_{p,\chi}^{\rm Katz}  \in \sA^{e}_{\bZ_p^{\rm ur}} \setminus \sA^{e+1}_{\bZ_p^{\rm ur}}$. 
\end{theorem}
\begin{proof}
This is an immediate consequence of Corollary~\ref{cor_algebraic_multivariate_exceptional_zero}.
\end{proof}

\begin{remark}
Theorem~\ref{thm_exceptional_zero_conj_OK} was deduced in our previous work \cite{BS19} from the leading to formula (Theorem 1.1 in op. cit.) for the Katz' $p$-adic $L$-function (c.f. \cite{BS19}, Corollary~1.4). Theorem 1.1 in \cite{BS19} requires a number of more stringent hypotheses and the argument we present here to prove Theorem~\ref{thm_exceptional_zero_conj_OK} enables us to by-pass some of these. 
\end{remark}



\section{Rankin--Selberg $p$-adic $L$-functions for nearly ordinary families of Hilbert Modular Forms}
\label{sec_Hida_RS}
We give a brief overview of Hida's theory of (nearly) ordinary families of Hilbert Modular forms and recall Hida's $p$-adic $L$-function attached to the Rankin--Selberg product of two such families. We refer to \cite{Hida88Annals, Hida89ANT, Hida91AIF} for details.

The main objective of this section is to explain the main motivation behind our results on the exceptional zeros of \emph{algebraic} Rankin--Selberg $p$-adic $L$-functions (which we establish in Section~\ref{sec:Factorisation of Selmer groups for symmetric Rankin--Selberg products}, relying on our work in Sections~\ref{sec_Selmer_Groups} and \ref{sec_Linvariants_heights_exceptionalzeros}). We will revisit the analytic objects we consider in the present section in the sequel~\cite{BS3}, where we expand on our outline here.


\subsection{Generalities}\label{subsec_4_1_generalities}  Let $I_F$ denote the set of all embeddings $F$ into $\overline{\QQ}$, which may equally regard as the set of infinite places $S_\infty(F)$ of $F$ via the fixed embedding $\iota_\infty$. We call elements of $\ZZ[I_F]$ weight vectors. For weight vectors $v={ \sum_{\sigma \in I_F} }\, v_\sigma\sigma$ and $ v^\prime={ \sum_{\sigma \in I_F} } \,v^\prime_\sigma\sigma$ we write $v \geq v^\prime$ to mean that $v_\sigma\geq v^\prime_\sigma$ for all $\sigma \in I_F$. Let us put $t_F=\sum_{\sigma\in I_F} \sigma$.

We let $M_F/F$ denote the maximal abelian extension of $F$ unramified outside $p$ and we set $\bfG^\prime:=\Gal(M_F/F)\big{/}{\rm (torsion)}$. We note that $\bfG'=\Gamma_\infty^+$ in the notation of \S\ref{subsec_notation_set_up_1_2}; it has been introduced here solely for consistency with Hida's notation in \cite{Hida88Annals, Hida89ANT, Hida91AIF} that we greatly dwell on throughout \S\ref{sec_Hida_RS}. We put $\bfG:=U_{1,F}\times\bfG^\prime$, where $U_{1,F}=(\cO_F\otimes\ZZ_p)^\times\big{/}{\rm (torsion)}$. Note that $\bfG$ is isomorphic to $\ZZ_p^{[F:\QQ]+1+\delta}$, where $\delta$ is the Leopoldt defect. 

Let $\frak{n}_0$ be an integral ideal of $\cO_F$ and let ${\bf h}^{\rm n.o.}(\frak{n}_0)$ denote Hida's nearly ordinary Hecke algebra of tame level $\frak{n}_0$, defined as in \cite{Hida88Annals}. Hida's fundamental theorem \cite[Theorem 2.4]{Hida88Annals} tells us that the algebra ${\bf h}^{\rm n.o.}(\frak{n}_0)$ is a finitely generated projective module over $\LL^{\rm n.o.}:=\LL_{\cO}(\bfG)$, and as such, it is the product of its localizations at its finitely many maximal ideals.
\begin{definition}
\label{defn_arithmetic_specialization}
An arithmetic specialization $\kappa$ of ${\bf h}^{\rm n.o.}(\frak{n}_0)$ is a ring homomorphism
$$\kappa: {\bf h}^{\rm n.o.}(\frak{n}_0)\lra \overline{\QQ}_p$$
which takes the form $(a,z) \mapsto a^{v(\kappa)}\chi_\cyc^{m(\kappa)}(z)$ in a small enough neighborhood of $1\in \mathbf{G}$. Here, $v(\kappa)\geq 0$ is a  weight vector and $m(\kappa)\in \ZZ$ is such that $n(\kappa):=m(\kappa)t_F-2v(\kappa)\geq 0$. As a result, the character $(\Phi_\kappa,\Phi_\kappa^\prime):=\kappa(z,a)\chi_\cyc^{-m(\kappa)}a^{-v(\kappa)}$ of $\bfG$ has finite order.
\end{definition}

\begin{definition}
For a finite extension ${\bf L}$ of ${\rm Frac}(\LL^{\rm n.o.})$, we let $\LL_{\bf L}$ denote the integral closure of $\LL^{\rm n.o.}$. A primitive Hida family is a $\LL^{\rm n.o.}$-algebra homomorphism $\g:  {\bf h}^{\rm n.o.}(\frak{n}_0) \to \LL_{\bf L}$ such that for every specialization $\kappa: {\bf h}^{\rm n.o.}(\frak{n}_0) \to \overline{\QQ}_p$, the composite map $\kappa\circ \g$ corresponds to an overconvergent nearly ordinary cuspidal Hilbert modular form. Moreover, when $\kappa$ is arithmetic, the composite map $\kappa\circ \g$ corresponds to a classical nearly ordinary cuspidal Hilbert modular newform of weight $(k(\kappa),w(\kappa))=(n(\kappa)+2t_F, t_F-v(\kappa))$, in the notation of \cite[\S0]{Hida91AIF}. 
\end{definition}
We remark that forms with parallel weight correspond to those for which $v(\kappa)=0$, whereas forms with parallel weight one correspond to those for which $v(\kappa)=0=n(\kappa)$. 
\begin{example}[Branches] 
We pick a maximal ideal $\fm^{\rm n.o.}$. Given a minimal ideal $\frak{a}\subset \fm^{\rm n.o.}$, we put $\bI_\fa$ for the normalization of integral domain ${\bf h}^{\rm n.o.}(\frak{n}_0)_{\fm^{\rm n.o.}}/\fa$. It follows from Hida's work that $\bI_\fa$ is finite 
over the ring $\LL^{\rm n.o.}$. 
The universal nearly ordinary branch of tame level $\frak{n}_0$ associated to the pair $\frak{a}\subset \fm^{\rm n.o.}$ is the tautological $\LL^{\rm n.o.}$-algebra homomorphism 
$$\g=\g_a: {\bf h}^{\rm n.o.}(\frak{n}_0) \lra \bI_\fa$$
which has the property that for every every arithmetic specialization $\kappa: {\bf h}^{\rm n.o.}(\frak{n}_0) \to \overline{\QQ}_p$ whose associated prime ideal $\ker(\kappa)$ contains $\fa$, the composite map $\kappa\circ \g$ corresponds to a classical nearly ordinary cuspidal Hilbert modular newform of weight $(k(\kappa),w(\kappa))$. When $\frak{a}$ and $\g$ are given as above, we will often write $\mathbb{I}_{\g}$ in place of $\mathbb{I}_{\frak a}$. If $\kappa$ is an arithmetic specialization such that $\ker(\kappa)$ contains $\frak{a}$, we will call $\kappa$ an arithmetic specialization of $\bI_\frak{a}=\bI_{\g_a}$. 
\end{example}


\subsection{Hida's $p$-adic $L$-function}
\label{subsec_Hida_RS_L}
Let $\f$ and $\g$ be two primitive Hida families of respective tame levels $\frak{n}_\f$ and $\frak{n}_\g$.  In \cite{Hida91AIF}, Hida proved the existence of a $p$-adic $L$-function $\mathscr{D}(\f,\g)\in \frac{1}{H_\f}\mathbb{I}_\f\, \widehat{\otimes}\, \mathbb{I}_\g$, where $H_\f \subset \mathbb{I}_\f$ is Hida's congruence ideal given as in \cite[\S5]{Hida91AIF}. This $p$-adic $L$-function is characterized by an interpolation property, which we only record below in a very crude form and refer readers to \cite[Theorem I]{Hida91AIF} for details: For every arithmetic specialization $\kappa_1$ of $\bI_\f$ and $\kappa_2$ of $\bI_\g$ with the properties that
\begin{equation}
\label{eqn_Hida_weight_restrictions}
n(\kappa_1)-n(\kappa_2) \geq t_F\,\,\,\,,\,\,\,\,\, v(\kappa_2)\geq v(\kappa_1)\,\,\,\,,\,\,\,\,\,  (m(\kappa_1)-m(\kappa_2))t_F\geq n(\kappa_2)-n(\kappa_1)+2t_F\,,
\ee
we have
\begin{equation}
\label{eqn_Hida_interpolation_general}
\mathscr{D}(\f,\g)(\kappa_1,\kappa_2)=\mathscr{E}^{\rm Hida}(\kappa_1,\kappa_2)\frac{\mathscr{D}\left(1+\frac{m(\kappa_2)-	m(\kappa_1)}{2}, \f(\kappa_1),\g(\kappa_2)^{c}\right)}{\pi^{\sum_{\sigma\in I_F} 2v_\sigma(\kappa_2)+n_\sigma(\kappa_2)+3-2v_\sigma(\kappa_1)}\langle \f(\kappa_1)^\circ, \f(\kappa_1)^{\circ} \rangle}
\end{equation}
where
\begin{itemize}
\item $\mathscr{E}^{\rm Hida}(\kappa_1,\kappa_2)$ are the interpolation factors which appear in the statement of   \cite[Theorem I]{Hida91AIF}, which we do not copy here;
\item $\g(\kappa_1)^c$ is the complex conjugate of the form $\g(\kappa_1)$;
\item $\f(\kappa_1)^\circ$ is the primitive form of level $\frak{n}_\f$ associated to $\f(\kappa_1)$;
\item $\langle \f(\kappa_1)^\circ, \f(\kappa_1)^{\circ} \rangle$ is the self Petersson inner product of the primitive form $\f(\kappa_1)^\circ$\,.
\end{itemize}


\subsection{CM forms and factorization of Hida's $p$-adic $L$-function}
\label{sec_example_CM_family}
Recall the totally imaginary quadratic extension $K$ of $F$. We shall revisit Hida's explicit description of the CM component of his nearly ordinary universal Hecke algebra. Our main reference for that portion is \cite[\S\S4.5--4.6]{HT94}. We will then use this description to factor Hida's $p$-adic $L$-function associated to a pair of CM families into a product of Katz $p$-adic $L$-functions.

Let us fix once and for all a ray class character $\psi$ of $K$ of conductor $\frak{c}$. In Hida's terminology, this is a branch character, which determines a CM universal nearly ordinary branch. We assume that $(\frak{c},p)=1$; this amounts to the requirement that the parallel weight one specialization which arises from the character $\psi$ be crystalline at $p$. We fix a sufficiently large finite flat extension $\cO$ of $\ZZ_p$ that contains the values of $\psi$.

\begin{remark}\label{rem_no_structure_onWK}
In this paragraph, we explain how to endow $\LL_\cO(W_K)$ with the structure of a finite $\LL^{\rm n.o.}$-module, by constructing a map $\bfG=\bfG^\prime \times U_{1,F} \hookrightarrow W_K$. By class field theory, $W_K$ fits into an exact sequence 
\[
0\lra \overline{U}_K \lra \prod_{\frak{P}\in S_p(K)}U_{1,K_{\frak{P}}}\lra W_K, 
\]
where $U_{1,K_{\frak{P}}}$ are the $1$-units in $\cO_{K_{\frak{P}}}^\times$ and $\overline{U}_K$ is the closure of $\displaystyle{\cO_K^\times \cap  \prod_{\frak{P}\in S_p(K)}U_{1,K_{\frak{P}}}}$. For any prime $\frak{p}\in S_p(F)$, the choice of $\iota_p$ determines a prime $\frak{P}_\circ \in \Sigma_p$ and we obtain a natural map $U_{1,F}\to W_K$ given as the compositum
$$U_{1,F}=\prod_{\frak{p}\in S_p(K)}U_{1,F_{\frak p}}\stackrel{\sim}{\lra} \prod_{\frak{P}\in \Sigma_p^c}U_{1,K_{\frak{P}}} \hookrightarrow  \prod_{\frak{P}\in S_p(K)}U_{1,K_{\frak{P}}} \lra W_K\,.$$
This map is injective, since the intersection ${\overline{U}_K\cap \left( \prod_{\frak{P}\in \Sigma_p^c}U_{1,K_{\frak{P}}}\times \prod_{\frak{P}\in \Sigma_p}\{1\}\right)}$ is trivial. We also have a map $\bfG^\prime \to W_K$ induced by base extension from $F$ to $K$ of the maximal abelian $p$-ramified extension of $F$. This map is injective since $p$ is odd and $K/F$ is quadratic. We therefore obtain a map $\bfG^\prime \times U_{1,F}\to W_K$, which is injective since $\Gal(K/F)$ acts trivially on the image of $\bfG^\prime$ and it clearly doesn't on the image of $U_{1,F}$. This completes the construction of a natural injection 
$$\LL_\cO(\bfG)=\LL^{\rm n.o.} \lra \LL_{\cO}(W_K)\,.$$
Since $\bfG$ and $W_K$ have the same $\ZZ_p$-rank, it follows that $\LL_{\cO}(W_K)$ acquires naturally the structure of a $\LL^{\rm n.o.}$-module of finite rank.
 \end{remark}
\begin{convention}\label{convention_4_5} For any continuous ring homomorphism $\kappa: \LL_{\cO}(W_K)\to\overline{\QQ}_p$, we denote the composite map $\LL^{\rm n.o.} \hookrightarrow \LL_{\cO}(W_K)\xrightarrow{\kappa} \overline{\QQ}_p$ also by $\kappa$. 
\end{convention}

\begin{definition}
Let $\sC(W_K,\cO)$ denote the $\cO$-module of continuous functions $W_K \to \cO$.
\end{definition}
We note that there is a natural perfect pairing $\sC(W_K,\cO) \times \LL_{\cO}(W_K)\to \cO$, given by $(f,k)\mapsto f(k)$.

\subsubsection{Hecke characters and theta series}
We let $\lambda: \mathbb{A}_K^\times/K^\times \to \bC^\times$ be a Hecke character of type $A_0$, with conductor $\frak{c}_\lambda$ dividing $\frak{c}p^\infty$ and infinity type $\eta \in \ZZ[I_K]$, so that $\lambda(x_\infty)=x_\infty^{-\eta}$ for each $x_\infty \in (K\otimes_{\QQ}\mathbb{R})^\times$. Let $m\geq 0$ be an integer such that $\eta_\sigma+\eta_{\sigma^c}=m$ for every $\sigma\in I_K$. We say that $\lambda$ is $\Sigma$-admissible if $\eta_\sigma>\eta_{\sigma^c}$ for every $\sigma \in \Sigma$.

\begin{definition}
The $p$-adic avatar $\widehat{\lambda}: \mathbb{A}_K^\times/K^\times (K\otimes \bR)^\times \to \overline{\bQ}_p^\times$ of $\lambda$ is given as $\widehat{\lambda}(x)=\lambda(x)x_\infty^{\eta}x_p^{-\eta}$, where $x_p\in (K\otimes_{\QQ} \QQ_p)^\times$ is the component of the idele $x$.
\end{definition}

Via the Artin reciprocity law, we view $\widehat{\lambda}$ as a character of $\bG_{\frak c}=\Delta_{\frak c} \times W_K$. This yields a factorization $\widehat{\lambda}=\widehat{\lambda}_{\frak c} \otimes \widehat{\lambda}_W$. From now on, we will assume that  $\widehat{\lambda}_{\frak c}=\psi$ and to indicate that, we say $\lambda$ is in the $\psi$-branch.

\begin{definition} Let $\lambda$ be a Hecke character of $A_0$-type.
\item[i)] We let $\infty(\lambda)$ denote the infinity type of $\lambda$. 
\item[ii)] We define the unitarization of $\lambda$ on setting $\lambda_0=\lambda|\cdot|_{\mathbb{A}}^{m/2}$, for which we have 
$$\infty(\lambda_0)={\sum_{\sigma\in \Sigma_\infty}\,(\eta_{\sigma}-\eta_{\sigma^c})(\sigma-\sigma^c)}\in \ZZ[I_K]\,.$$
We also set $\lambda^{\rm ad}:=\lambda^c/\lambda$. 
\item[iii)]We define the weight of $\lambda$ by setting $k(\lambda):={\sum_{\sigma\in \Sigma_\infty}\,(\eta_{\sigma}-\eta_{\sigma^c}+1)\sigma_{\vert_F} \in \ZZ[I_F]}$. We set $v(\lambda):=\sum_{\sigma\in \Sigma_\infty^c}\,\eta_{\sigma}\sigma_{\vert_F} \in \ZZ[I_F]$ and $w(\lambda):=t-v(\lambda)$. 
\item[iv)]We finally set $\frak{n}_\psi:=D_{K/F}N_{K/F}(\frak{c})$, $N_\psi:=N_{F/\QQ}(\frak{n}_\psi)$ and $\frak{n}_\lambda:=D_{K/F}N_{K/F}(\frak{c}_\lambda)$. Note that $\frak{n}_\lambda$ divides $N_\psi p^r$ for suitably large $r$.
\end{definition}

In \cite[\S4.5]{HT94}, Hida and Tilouine explain the construction of 
\begin{itemize}
\item the Hilbert Modular Form $\theta(\lambda_0)$ (the theta-series of the unitary Hecke character $\lambda_0$) of weight $(k(\lambda),k(\lambda)/2)$ and level $U_1(\frak{n}_\lambda)$ and $\theta^{\rm cl}(\lambda):=|\det(\cdot)|_{\mathbb{A}}^{\frac{1-m}{2}}\theta(\lambda_0)$ of weight $(k(\lambda),w(\lambda))$;
\item the nearly ordinary theta series $\theta(\lambda)$ associated to $\theta^{\rm cl}(\lambda)$, of weight $(k(\lambda),w(\lambda))$ and level $S(p^r)=U_1(N)\cap U_1(p^r)$,
\item a continuous $A$-linear map
$$\Theta_{\Sigma,\psi}: \sC(W_K,\cO)\lra \overline{\mathscr S}^{\rm n.o.}(S(p^\infty),\cO)\,,$$
where $ \overline{\mathscr S}^{\rm n.o.}(S(p^\infty),\cO)$ is the space overconvergent nearly ordinary Hilbert modular cuspforms given as in \cite[Definition~4.4.1]{HT94}, which has the interpolative property that 
$$\Theta_{\Sigma,\psi}(\widehat{\lambda}_W)=\theta(\lambda)$$
for any Hecke character $\lambda$ of $A_0$-type in the $\psi$-branch.
\end{itemize}

Moreover, in \S4.6 op.cit., Hida and Tilouine use the universal nearly ordinary theta series $\Theta_{\Sigma,\psi}$ to determine a $\Lambda^{\rm n.o.}$-algebra homomorphism 
$$\Theta_{\Sigma,\psi}^*: {\bf h}^{\rm n.o.}(\frak{n}_\psi) \lra \LL_{\cO}(W_K),$$
which we call the $\psi$-branch of the universal CM nearly ordinary Hida family of tame level $\frak{n}_\psi$. We note that $\Theta_{\Sigma,\psi}^*$ factors as 
$${\bf h}^{\rm n.o.}(\frak{n}_\psi) \lra \LL_{\cO}(\bG_{\frak c})=\ZZ[\Delta_{\frak{c}}]\otimes \LL_{\cO}(W_K)\xrightarrow{\psi\otimes{\rm id}}\LL_{\cO}(W_K).$$

\begin{definition}
Let $\kappa,\kappa^\prime: \LL_{\cO}(W_K)\lra\overline{\QQ}_p$ be a pair of continuous ring homomorphisms and let $s\in 1+p\ZZ_p$ be an arbitrary element.
\item[i)] We write $\kappa^{(\psi)}:  \LL_{\cO}(\bG_{\frak c}) \to \overline{\QQ}_p$ for the compositum 
$$ \LL_{\cO}(\bG_{\frak c})=\ZZ[\Delta_{\frak{c}}]\otimes \LL_{\cO}(W_K)\xrightarrow{\psi\otimes \kappa}\overline{\QQ}_p\,.$$
We also denote the induced map ${\bf h}^{\rm n.o.}(\frak{n}_\psi)\lra \overline{\QQ}_p$ also by  $\kappa^{(\psi)}$.
\item[ii)] We put $\widehat{\lambda}_{\kappa,\psi}: \bG_{\frak c}=\Delta_\frak{c}\times W_K\xrightarrow{\,\psi\,\kappa\,} \overline{\QQ}_p^\times$. We say that $\kappa$ is anticyclotomic if for every $x\in W_K$, we have
$$\widehat{\lambda}_\kappa(cxc^{-1})=:\widehat{\lambda}_\kappa^c(x)=\widehat{\lambda}_\kappa^{-1}(x)\,.$$
\item[ii)] We define the cyclotomic twist $\Theta_{\Sigma,\psi}\{s\}$ of the family $\Theta_{\Sigma,\psi}$ as the one that corresponds to the ring homomorphism 
$${\bf h}^{\rm n.o.} \xrightarrow{\Theta_{\Sigma,\psi}^*} \LL_{\cO}(W_K) \xrightarrow{\gamma \mapsto \chi_\cyc^s(\gamma)\gamma} \LL_{\cO}(W_K)\,.$$
\item[iii)] We write $\kappa+s$ both for the continuous ring homomorphism given as the composite map
$$\LL_{\cO}(W_K)\xrightarrow{\gamma \mapsto \chi_\cyc^s(\gamma)\gamma} \LL_{\cO}(W_K)\stackrel{\kappa}{\lra}\overline{\QQ}_p$$
and $\kappa^{(\psi)}+s$ for the compositum $ \LL_{\cO}(\bG_{\frak c})=\ZZ[\Delta_{\frak{c}}]\otimes \LL_{\cO}(W_K)\xrightarrow{\psi\otimes{\rm id}}\LL_{\cO}(W_K)\xrightarrow{\kappa+s}\overline{\QQ}_p$. We denote the induced map ${\bf h}^{\rm n.o.}(\frak{n}_\psi)\to  \LL_{\cO}(\bG_{\frak c})\to \overline{\QQ}_p$ also by  $\kappa^{(\psi)}+s$. More generally, we define $\kappa\pm\kappa^\prime$ as the continuous ring homomorphism
$$\LL_{\cO}(W_K)\xrightarrow{\gamma \mapsto (\kappa^{\prime})^{\pm 1}(\gamma)\gamma} \LL_{\cO}(W_K)\stackrel{\kappa}{\lra}\overline{\QQ}_p$$
and similarly define $(\kappa\pm \kappa^\prime)^{(\psi)}$.
\end{definition}
Observe that $\widehat{\lambda}_{\kappa+s,\psi}=\chi_\cyc^s\widehat{\lambda}_{\kappa,\psi}$ and  $\widehat{\lambda}_{\kappa,\psi}= \widehat{\lambda}_{\kappa-\kappa^\prime,\psi}\, \widehat{\lambda}_{\kappa^\prime,\psi}$.

\begin{definition}
\item[i)] Suppose $\kappa$ is an arithmetic specialization. Let $\lambda_{\kappa,\psi}$ be the unique Hecke character of $A_0$-type whose $p$-adic avatar gives rise to the Galois character $\widehat{\lambda}_{\kappa,\psi}$.
\item[ii)] Let $\lambda_{\Sigma}$ be any Hecke character of conductor dividing $p$ and infinity type $\Sigma_\infty$. Let $\Phi$ be a sufficiently large finite extension of $\QQ_p$ that contains $\iota_p({\rm im}(\lambda_{\Sigma}))$ and let $\cO_\Phi$ denote its ring of integers. Let us put $\cO_\Phi^\times=\mu(\Phi)\times W_\Phi$ where $\mu(\Phi)$ is finite and $W_\Phi$ is free of finite rank over $\ZZ_p$. We let $\langle \,\cdot\, \rangle: \cO_\Phi^\times \to W$ denote the natural projection and consider the character $\langle \lambda_{\Sigma}\rangle$. The character $\langle \lambda_{\Sigma}\rangle$ clearly doesn't depend on the choice of $\lambda_{\Sigma}$. We set 
$\langle \lambda_{\Sigma}^{(\psi)}\rangle:=\psi\langle \lambda_{\Sigma}\rangle$, which is a $p$-adic Hecke character of conductor dividing $\frak{c}p^\infty$ and $p$-adic type $\Sigma_p$.
\end{definition}
\begin{lemma}
Suppose $\kappa$ is an arithmetic specialization and $s$ is any integer. The infinity type of $\lambda_{\kappa,\psi}$ is given by
$$\infty(\lambda_{\kappa,\psi})=\sum_{\sigma\in \Sigma_\infty} (n_\sigma(\kappa)+v_\sigma(\kappa)+1)\sigma+\sum_{\sigma\in \Sigma_\infty^c} v_\sigma(\kappa)\sigma\,.$$
\end{lemma}

\begin{proof}
This follows from the definition of an $A_0$-type Hecke character.
\end{proof}


\subsubsection{Factorization} Our objective in this subsection is to recall the factorization result due to Hida and Tilouine~\cite[Theorem 8.1]{HT93}. To that end, we consider Hida's Rankin--Selberg $p$-adic $L$-function
$$\mathscr{D}(\Theta_{\Sigma,\psi},\Theta_{\Sigma,\psi}) \in \LL_{\cO}(W_K)\,\widehat{\otimes} \, \LL_{\cO}(W_K)\,.$$
associated to the self-product of the $\psi$-branch $\Theta_{\Sigma,\psi}$. We identify $W_K$ with $\Gamma_\infty$ via the tautological isomorphism $W_K\to \Gamma_\infty$. The following is a restatement of \cite[Theorem 8.1]{HT93}.

\begin{theorem}[Hida--Tilouine]
\label{thm_HT_factorization}
There exists an Iwasawa function $\frak{F}\in \LL_{\cO}(W_K)\,\widehat{\otimes} \, \LL_{\cO}(W_K)$ with the following properties.
For any pair of ring homomorphisms $\kappa,\kappa^\prime: \LL_{\cO}(W_K) \to \overline{\QQ}_p$ we have
$$\mathscr{D}(\Theta_{\Sigma,\psi},\Theta_{\Sigma,\psi})(\kappa,\kappa^\prime)\cdot   L_{p, \psi^{\rm ad}}^{\rm Katz} \big{|}_{\Gamma_{\rm ac}}(\widehat{\lambda}_{\kappa}^{\rm ad})= \frak{F}(\kappa,\kappa^\prime) \,  L_{p,\mathds{1}}^{\rm Katz} (\widehat{\lambda}_{\kappa^\prime-\kappa})\, L_{p, \psi^{\rm ad}}^{\rm Katz}(\widehat{\lambda}_{\kappa}^{\rm ad}\widehat{\lambda}_{\kappa^\prime-\kappa}^c)$$
and $\frak{F}$ is non-zero when $s=0$ and $\kappa=\kappa^\prime$.
\end{theorem}

\begin{proof}
We explain how to translate \cite[Theorem 8.1]{HT93} to the present statement. In op. cit., we take 
\begin{itemize}
\item $\lambda=\Theta_{\Sigma,\psi}=\nu$,
\item $P=\kappa$ and $Q=\kappa^\prime$
\end{itemize}
so that $\lambda_P=\widehat{\lambda}_{\kappa,\psi}$ and $\nu_Q=\widehat{\lambda}_{\kappa^\prime,\psi}$. As a result, $\varepsilon_{P,Q}=\lambda_P^{-1}\nu_Q=\widehat{\lambda}_{\kappa-\kappa^\prime}$. Similarly, 
$$\varepsilon_{P,Q}^\prime=\lambda_P^{-1}\nu_Q^c=\psi^{\rm ad}\widehat{\lambda}_{\kappa^\prime}^{c}/\widehat{\lambda}_{\kappa}=\psi^{\rm ad}\widehat{\lambda}_\kappa^{\rm ad}\widehat{\lambda}_{\kappa^\prime-\kappa}^{c}$$ 
and $\varepsilon_{P}=\lambda_P^{-1}\lambda_P^c=\psi^{\rm ad}\widehat{\lambda}_\kappa^{\rm ad}$. The term $\frak{F}$ corresponds to the expression $U\Psi_1\Psi_2$ in \cite[Theorem 8.1]{HT93} and the proof follows.
\end{proof}

\begin{remark}
We will be interested in the following particular choices of the pairs $(\kappa,\kappa^\prime)$:
\item[i)] $\kappa^\prime=\kappa+s$ and $\kappa$ is anticyclotomic and $s\in \ZZ$: In this scenario, the factorization of Theorem~\ref{thm_HT_factorization} reduces to 
\begin{equation}
\label{eqn_thm_HT_factorization_1}
\mathscr{D}(\Theta_{\Sigma,\psi},\Theta_{\Sigma,\psi})(\kappa,\kappa+s)\cdot L_{p, \psi^{\rm ad}}^{\rm Katz}\big{|}_{\Gamma_{\rm ac}}(\widehat{\lambda}_{\kappa}^{\rm ad})= \frak{F}(\kappa,\kappa+s) \, L_{p,\mathds{1}}^{\rm Katz}(\chi_\cyc^s)\, L_{p, \psi^{\rm ad}}^{\rm Katz}(\widehat{\lambda}_{\kappa}^{\rm ad}\chi_\cyc^s)\,.
\end{equation}
\item[ii)] $\widehat{\lambda}_{\kappa^\prime-\kappa}=\langle \lambda_\Sigma\rangle^s$, where $s\in \ZZ$. We write in this case $\kappa^\prime=\kappa+s\Sigma$. In this set up, Theorem~\ref{thm_HT_factorization} reads
\begin{equation}
\label{eqn_thm_HT_factorization_2}
\mathscr{D}(\Theta_{\Sigma,\psi},\Theta_{\Sigma,\psi})(\kappa,\kappa+s\Sigma)\cdot L_{p, \psi^{\rm ad}}^{\rm Katz}\big{|}_{\Gamma_{\rm ac}}(\widehat{\lambda}_{\kappa}^{\rm ad})= \frak{F}(\kappa,\kappa+s\Sigma) \, L_{p,\mathds{1}}^{\rm Katz}(\langle \lambda_\Sigma\rangle^s)\, L_{p, \psi^{\rm ad}}^{\rm Katz}(\widehat{\lambda}_{\kappa}^{\rm ad}\langle \lambda_\Sigma^c\rangle^s)\,.
\end{equation}
\end{remark}

Let us set $\Gamma_{\rm ac} := W_K/W_{K}^{+}$\,, where $W_{K}^{+}\subset W_K$ is the $\ZZ_p$-submodule on which $\Gal(K/F)$ acts trivially. We denote by $K_{\rm ac}\subset K(p^{\infty})$ the fixed field of $W_{K}^{+}$; note that $K_{\rm ac}$ is the maximal $\ZZ_p$-power anticyclotomic extension of $K$. We recall that 
\[
\Gamma_\infty^{\circ} := \Gal(K_{\rm cyc}K_{\rm ac}/K). 
\]
Note that 
the natural map $\Gamma_\infty^{\circ} \longrightarrow \Gamma_{\rm cyc} \times \Gamma_{\rm ac}$ is an isomorphism and 
we have a natural surjection $\Gamma_\infty \twoheadrightarrow \Gamma_\infty^{\circ}$, which is an isomorphism if the Leopoldt conjecture for $F$ is valid. 

\begin{definition}
\label{defn_regularized_Hida_RS}
We let ${ \widetilde{\mathscr{D}}^{\rm Hida}} \in \LL_{\cO}(\Gamma_\infty^{\circ})$ denote the unique element that verifies 
$${ \widetilde{\mathscr{D}}^{\rm Hida}}(\kappa,s)=\mathscr{D}(\Theta_{\Sigma,\psi},\Theta_{\Sigma,\psi})(\kappa,\kappa+s)\cdot L_{p, \psi^{\rm ad}}^{\rm Katz}\big{|}_{\Gamma_{\rm ac}}(\widehat{\lambda}_{\kappa}^{\rm ad})$$
for all anticyclotomic $\kappa$ and $s\in 1+p\ZZ_p$. We also put 
\[
\widetilde{\mathscr D}_{{\rm ad}^0\Theta_{\Sigma,\psi}}(\kappa,s):={ \widetilde{\mathscr{D}}^{\rm Hida}}(\kappa,s)/\zeta_p(1-s), 
\]
where $\zeta_p$ is the $p$-adic Dedekind zeta function.
\end{definition}

\begin{remark}
\label{remark_gross_dasgupta}
\item[i)] We let $\omega_F:G_F\to \ZZ_p^\times$ denote the character giving the action of $G_F$ on $p$th roots of unity and let $\epsilon_{K/F}$ denote the non-trivial quadratic character of $\Gal(K/F)$. We let $L_p^{\rm DR}(s,\epsilon_{K/F}\omega)$ denote the Deligne--Ribet $p$-adic $L$-function associated to the totally even character $\epsilon_{K/F}\omega$. One may raise the question whether or not we have 
\begin{equation}
\label{eqn_Gross_CM}
   L_{p,\mathds{1}}^{\rm Katz} (\chi_\cyc^s)\stackrel{?}{=}\zeta_p(1-s)L_p^{\rm DR}(s,\epsilon_{K/F}\omega) \,,
\end{equation}
as a generalization of Gross' factorization result (main theorem of \cite{Gross1980Factorization}). Since the set of critical values for $\zeta_p(1-s)$ and that for $L_p^{\rm DR}(s,\epsilon_{K/F}\omega)$ are disjoint, the expected factorization \eqref{eqn_Gross_CM} cannot be deduced from the Artin formalism for complex $L$-functions.
\item[ii)] The $p$-adic Dedekind zeta function $\zeta_p(1-s)$ has a simple pole at $s=0$, whereas we have ${\rm ord}_{s=0}\,L_p(s,\chi\omega)\geq {\#S_{p}(F)}$ (Gross conjectures an equality). In view of \eqref{eqn_Gross_CM}, one is lead to expect that
\begin{equation}
\label{eqn_conj_trivial_branch_exceptional_zeros}
{\rm ord}_{s=0}\,L_{p,\mathds{1}}^{\rm Katz}(\chi_\cyc^s) \stackrel{?}{=} {\#S_{p}(F)}-1\,.
\end{equation}
\item[iii)] Based on the Artin formalism for complex $L$-functions, one expects that $\widetilde{\mathscr D}_{{\rm ad}^0\Theta_{\Sigma,\psi}}(\kappa,s)$ interpolates the critical values of the relevant adjoint $L$-values. However, since $n(\kappa+s)=n(\kappa)$, there are no critical specializations for ${ \widetilde{\mathscr{D}}^{\rm Hida}}$.  As a result, the expected interpolation properties for $\widetilde{\mathscr D}_{{\rm ad}^0\Theta_{\Sigma,\psi}}(\kappa,s)$ cannot be deduced from the interpolation property \eqref{eqn_Hida_interpolation_general}.
The required interpolation formula can be proved if the factorization \eqref{eqn_Gross_CM} is valid, since then we would have
\begin{equation}
\label{eqn_Dasgupta_CM}
\widetilde{\mathscr D}_{{\rm ad}^0\Theta_{\Sigma,\psi}}(\kappa,s)= \frak{F}(\kappa,\kappa+s) \, L_p^{\rm DR}(s,\epsilon_{K/F}\omega)
\, L_{p, \psi^{\rm ad}}^{\rm Katz}(\widehat{\lambda}_{\kappa}^{\rm ad}\chi_\cyc^s)
\end{equation}
if \eqref{eqn_Gross_CM} holds true. The conjectural identity \eqref{eqn_Dasgupta_CM} is an extension of Dasgupta's factorization result for Rankin--Selberg products of elliptic modular forms to the present context.
\item[iv)] In Theorem~\ref{thm_algebraic_gross_factorization} and Theorem~\ref{theorem_adjoint} below, we shall prove that algebraic $p$-adic $L$-functions (given as the Fitting ideals of appropriately defined extended Selmer groups) verify the expected factorization results \eqref{eqn_Gross_CM} and \eqref{eqn_Dasgupta_CM}.
\end{remark}


\subsubsection{Exceptional zeros of Hida's $p$-adic $L$-function} 

Let us define 
$$e := \# \{\frak{P}\in \Sigma^c: \psi(\frak{P})=\psi(\frak{P}^c)\}\,.$$
In view of \cite[(1)]{BS19} and the identities \eqref{eqn_conj_trivial_branch_exceptional_zeros} and \eqref{eqn_Dasgupta_CM} which one expects due to Artin formalism, one is lead to the following question.

\begin{question}
\label{conj_exceptional_zeros_RankinSelberg_Hida}
Let ${\mathscr{A}^\circ} \subset \LL_{\cO}({\Gamma_\infty^\circ})$ denote the augmentation ideal. Then is it true that
\item[a)]${ \widetilde{\mathscr{D}}^{\rm Hida}}\in {(\mathscr{A}^\circ)}^{e+ \#S_p(F)-1} \setminus {(\mathscr{A}^\circ)}^{e+ \#S_p(F)}$\,,
\item[b)] $\widetilde{\mathscr D}_{{\rm ad}^0\Theta_{\Sigma,\psi}} \in  {(\mathscr{A}^\circ)}^{e+ \#S_p(F)} \setminus {(\mathscr{A}^\circ)}^{e+1+ \#S_p(F)}$\,\,?
\end{question} 

\begin{remark} 
Let $\mathscr{A}^{\rm ac}\subset \LL_{\cO}(\Gamma_{\rm ac})$ denote the augmentation ideal. It is worthwhile to note that the anticyclotomic restriction of Katz' $p$-adic $L$-function $L_{p, \psi^{\rm ad}}^{\rm Katz}\big{|}_{\Gamma_{\rm ac}}$ also has an exceptional zero at $\mathscr{A}^{\rm ac}$. In fact, it follows from \cite[Theorem 1.1]{BS19} that 
$${\rm ord}_{\mathscr{A}^{\rm ac}}L_{p, \psi^{\rm ad}}^{\rm Katz}\big{|}_{\Gamma_{\rm ac}}\geq e$$
with equality if and only if the group-ring-valued $\cL$-invariant  $\cL_{\Sigma,\Gamma}^{\rm Gal}$ (given as in Definition 5.2 in op. cit.) is non-zero form some $\ZZ_p$-extension $K_\Gamma/K$ contained in the anticyclotomic $\ZZ_p^{[F:\QQ]}$-extension $K^{\rm ac}$ of $K$. One is then lead to expect that Hida's Rankin--Selberg $p$-adic $L$-function $\mathscr{D}(\Theta_{\Sigma,\psi},\Theta_{\Sigma,\psi})$ has a {zero} of order {$\#S_{p}(F) -1$ at its weight one specialization corresponding to $\psi^{\rm ad}$}.
\end{remark}

One of our objectives in the present work is to answer Question~\ref{conj_exceptional_zeros_RankinSelberg_Hida} for the algebraic $p$-adic $L$-function (see Corollary~\ref{cor_exceptional_zeros_RankinSelberg_Hida} for our result in this direction). Whenever the Iwasawa Main Conjecture for the CM field $K$ is available, one may in turn obtain an affirmative answer to Question~\ref{conj_exceptional_zeros_RankinSelberg_Hida}; c.f. Remark~\ref{rem_Hsieh_2_Intro}. The remainder of our discussion is dedicated to formulate the algebraic counterpart of Question~\ref{conj_exceptional_zeros_RankinSelberg_Hida} and prove its validity. We will revisit Question~\ref{conj_exceptional_zeros_RankinSelberg_Hida} in the context of elliptic modular forms in the sequel~\cite{BS3}, and study this conjecture with the aid of Beilinson--Flach elements.


\subsection{Galois representations attached to a CM family}\label{sec:galois repn}
The description of Galois representations attached to the branches of CM families of Hilbert Modular Forms is rather simple. In the present article, we will focus on the exceptional zeros of algebraic $p$-adic $L$-functions, which are given in terms of the Selmer groups attached to these Galois representations.

\begin{remark}
When we work with the Rankin--Selberg product of two CM Hida families, we will be interested in the anticyclotomic character $\psi^{\rm ad}:=\psi^c/\psi$. As explained in \cite[Remark 3.1.1]{HT94}, we may and will assume without loss of generality that the conductor of $\psi^{\rm ad}$ is ${\rm LCM}(\frak{c},\frak{c}^c)$.
\end{remark}

We now introduce the the two-dimensional $G_F$-representation attached to the nearly ordinary CM family $\Theta_{\Sigma,\psi}$.

\begin{definition}
\label{defn_bigPsi}
We put $\Psi: G_K\to \LL_{\cO}(W_K)^{\times}$ for the tautological character, given as the compositum 
$$G_K \lra \bG_{\frak{c}}=\Delta_{\frak c}\times W_K \xrightarrow{(\delta,\gamma)\mapsto \psi(\delta)[\gamma]} \LL_{\cO}(W_K)^\times\,,$$
where $[\gamma] \in  \LL_{\cO}(W_K)^\times $ is the image of $\gamma\in W_K$ under the map $W_K\hookrightarrow \LL_{\cO}(W_K)^\times$. We set $M_{\Psi}:={\rm Ind}_{K/F}\Psi$, which is a free $ \LL_{\cO}(W_K)$-module of rank two, endowed with a continuous action of $G_F$. We also consider $M_{\Psi^{-1}}:={\rm Ind}_{K/F} \Psi^{-1}={\rm Hom}_{\LL_\cO(W_K)}(M_{\Psi},\LL_\cO(W_K))$. \end{definition}


\subsubsection{Symmetric Rankin--Selberg products} 
We start this paragraph recalling that we endow any $\LL_{\cO}(W_K)$-module with the structure of a $\LL_{\cO}({\bf G}')$-module via the embedding ${\bf G}' \hookrightarrow W_K$, which we have reviewed in Remark~\ref{rem_no_structure_onWK}.
We put 
\[
M({\rm ad}\Theta_{\Sigma,\psi}):= (M_\Psi \otimes_{\LL_{\cO}(W_K)} M_{\Psi^{-1}}) \otimes_{\LL_{\cO}({\bf G}')} \LL_{\cO}({\bf G}')^{\sharp}, 
\] 
where 
$\LL_{\cO}({\bf G}')^\sharp$ is the free $\LL_{\cO}({\bf G}')$-module of rank one on which $G_F$ acts via the tautological character $G_F\twoheadrightarrow {\bf G}' \hookrightarrow \LL_{\cO}({\bf G}')^\times$. 
 We note that we have a natural isomorphism 
$$\Hom_{\LL_{\cO(W_K)}}(M_\Psi \otimes_{\LL_{\cO}(W_K)} M_{\Psi^{-1}},\LL_{\cO(W_K)})\stackrel{\sim}{\lra} M_\Psi \otimes_{\LL_{\cO}(W_K)} M_{\Psi^{-1}}\,.$$
of Galois modules. This induces a self-duality isomorphism
\begin{equation}
\label{eqn_self_duality_end_full}
\Hom_{\LL_{\cO(W_K)}}(M({\rm ad}\Theta_{\Sigma,\psi}),\LL_{\cO(W_K)})\stackrel{\sim}{\lra}M({\rm ad}\Theta_{\Sigma,\psi})^\iota, 
\end{equation}
where 
\[
M({\rm ad}\Theta_{\Sigma,\psi})^\iota=(M_\Psi \otimes_{\LL_{\cO}(W_K)} M_{\Psi^{-1}}) \otimes_{\LL_{\cO}({\bf G}')} \LL_{\cO}({\bf G}')^{\iota}
\]
and $ \LL_{\cO}({\bf G}')^{\iota}=\Hom_{\LL_{\cO}({\bf G}')}(\LL_{\cO}({\bf G}')^{\sharp},\LL_{\cO}({\bf G}'))$  is the free $\LL_{\cO}({\bf G}')$-module of rank one on which $G_F$ acts via the character $G_F\twoheadrightarrow {\bf G}' \xrightarrow{\gamma\mapsto\gamma^{-1}} {\bf G}' \hookrightarrow \LL_{\cO}({\bf G}')^\times$.

The extended Selmer groups which are expected have a bearing on the regularized non-critical specialization of the Rankin--Selberg $p$-adic $L$-function $\widetilde{\mathscr{D}}^{\rm Hida}(\kappa, s)$ will come attached to the Galois representations $M({\rm ad}\Theta_{\Sigma,\psi})^\iota(1)\,{\otimes_{\LL_{\cO}(W_K)}\,{\LL_{\cO}(\Gamma_\infty^{\circ})}}$ and its Tate-dual $M({\rm ad}\Theta_{\Sigma,\psi})\,{\otimes_{\LL_{\cO}(W_K)}\,{\LL_{\cO}(\Gamma_\infty^{\circ})}}$. 


\subsubsection{Traceless adjoint} We observe the natural identifications 
\begin{align*}
M({\rm ad}\Theta_{\Sigma,\psi}) 
&= (M_\Psi \otimes_{\LL_{\cO}(W_K)} M_{\Psi^{-1}}) \otimes_{\LL_{\cO}({\bf G}')}\LL_{\cO}({\bf G}')^{\sharp} 
\\
&= \left(M_\Psi \otimes_{\LL_{\cO}(W_K)} \Hom_{\LL_{\cO}(W_K)} (M_\Psi,\LL_{\cO}(W_K))\right)\otimes_{\LL_{\cO}({\bf G}')} \LL_{\cO}({\bf G}')^{\sharp}
\\
&=\Hom_{\LL_{\cO}(W_K)} (M_\Psi, M_\Psi)\otimes_{\LL_{\cO}({\bf G}')} \LL_{\cO}({\bf G}')^{\sharp}
\\
&={\rm End}_{\LL_{\cO}(W_K)}(M_\Psi)\otimes_{\LL_{\cO}({\bf G}')} \LL_{\cO}({\bf G}')^{\sharp}\,.
\end{align*}
With these identifications at hand, we define 
\begin{align*}
M({\rm ad}^{0}\Theta_{\Sigma,\psi}) 
&:={\rm End}^0_{\LL_{\cO}(W_K)}(M_\Psi)\otimes_{\LL_{\cO}({\bf G}')} \LL_{\cO}({\bf G}')^{\sharp}
\\
&\subset {\rm End}_{\LL_{\cO}(W_K)}(M_\Psi)\otimes_{\LL_{\cO}({\bf G}')} \LL_{\cO}({\bf G}')^{\sharp}
\end{align*}
as traceless endomorphisms.

The self-duality isomorphism \eqref{eqn_self_duality_end_full} restricts to an isomorphism
\[
\Hom_{\LL_{\cO(W_K)}}(M({\rm ad}^{0}\Theta_{\Sigma,\psi}),\LL_{\cO(W_K)})\stackrel{\sim}{\lra}M({\rm ad}^{0}\Theta_{\Sigma,\psi})^\iota, 
\]
where $M({\rm ad}^{0}\Theta_{\Sigma,\psi})^\iota={\rm End}^0_{\LL_{\cO}(W_K)}(M_\Psi)\otimes_{\LL_{\cO}({\bf G}')} \LL_{\cO}({\bf G}')^{\iota}$. 
The extended Selmer groups which are expected have a bearing on the $p$-adic $L$-function $\widetilde{\mathscr{D}}({\rm ad}^{0}\Theta_{\Sigma,\psi})$ will come attached to the Galois representations $M({\rm ad}^{0}\Theta_{\Sigma,\psi})^\iota(1)$ and its Tate-dual $M({\rm ad}^{0}\Theta_{\Sigma,\psi})$. 

\begin{lemma}
We have the following decompositions of $G_F$-representations.
\item[i)] $M({\rm ad}\Theta_{\Sigma,\psi})= \mathds{1} \otimes
\LL_{\cO}({\bf G}')^{\sharp}\oplus M({\rm ad}^{0}\Theta_{\Sigma,\psi})$. 
\item[ii)] $M({\rm ad}^{0}\Theta_{\Sigma,\psi})=(\ZZ_p(\epsilon_{K/F})\otimes \LL_{\cO}({\bf G}')^{\sharp})\oplus ({\rm Ind}_{K/F} \Psi^{\rm ad}\otimes  \LL_{\cO}({\bf G}')^{\sharp})$.
\end{lemma}

\begin{proof}
Clear.
\end{proof}

\begin{remark}\label{remark:character}
We let ${\bbchi}: G_K\to \LL_{\cO}(W_K)^\times$, given as the compositum $ G_K\twoheadrightarrow W_K \hookrightarrow \LL_{\cO}(W_K)^\times$ denote the tautological (universal) character, so that $\Psi=\psi\bbchi$ and $\Psi^{\rm ad}=\psi^{\rm ad}\bbchi^{\rm ad}$. Similarly, we let ${\bbchi}_{\rm ac}: G_K\to \LL_{\cO}(\Gamma_{\rm ac})^\times$ denote the tautological (universal) anticyclotomic character. Let us consider the transfer map 
$${\rm ver}_{\rm ac}: \LL_{\cO}(\Gamma_{\rm ac}) \to \LL_{\cO}(W_K)$$ 
given by $\sigma \mapsto \widetilde{\sigma}^{c-1}$, where $\widetilde{\sigma}\in W_K$ is any element that lifts $\sigma$.

Observe that $\bbchi^{\rm ad}=\bbchi^c/\bbchi$ factors through $\Gamma_{\rm ac}$. It is related to the universal anticyclotomic character ${\bbchi}_{\rm ac}$ via the identity ${\rm ver}_{\rm ac}\circ {\bbchi}_{\rm ac}=\bbchi^{\rm ad}$. Through this relation, we may and will identify the $G_F$-representation ${\rm Ind}_{K/F} \Psi^{\rm ad}{\,\widehat\otimes\,  \LL_{\cO}(\bfG')^{\sharp}}$ with 
\[
{\rm Ind}_{K/F} \psi^{\rm ad}\otimes  \LL_{\cO}(\Gamma_{\rm ac})^{\sharp}\, \widehat{\otimes}\, \LL_{\cO}({\bf G}')^{\sharp}={\rm Ind}_{K/F} \psi^{\rm ad}\otimes \LL_{\cO}(W_K)^{\sharp}\,.
\]
Here (as before), $ \LL_{\cO}(\Gamma_{\rm ac})^{\sharp}$ is the free $ \LL_{\cO}(\Gamma_{\rm ac})$-module of rank one on which $G_K$ acts via ${\bbchi}_{\rm ac}$.
\end{remark}


\section{Factorization of algebraic Rankin--Selberg $p$-adic $L$-functions}\label{sec:Factorisation of Selmer groups for symmetric Rankin--Selberg products}
Our objective in this section is to prove the algebraic counterparts of Hida and Tilouine's factorization result \eqref{eqn_thm_HT_factorization_1} for the Rankin--Selberg $p$-adic $L$-function for nearly ordinary families of CM Hilbert modular forms, as well as the algebraic variants of the (conjectural) factorizations \eqref{eqn_Gross_CM} and \eqref{eqn_Dasgupta_CM}. This is achieved in Theorem~\ref{thm:iwasama main conj rankin-selberg}, Theorem~\ref{thm_algebraic_gross_factorization} and Theorem~\ref{theorem_adjoint} below, respectively.

This will enable us to study the exceptional zero problem for Katz $p$-adic $L$-functions via the exceptional zero problem for Rankin--Selberg products of nearly ordinary CM families; see Corollary~\ref{cor_exceptional_zeros_RankinSelberg_Hida} for a result with this flavour. Note that one may study the exceptional zeros of Rankin--Selberg $p$-adic $L$-functions for ${{\rm GL}_2}_{/\QQ}\times {{\rm GL}_2}_{/\QQ}$ with the aid of Beilinson--Flach elements. This is the subject of the sequel \cite{BS3} (and the main motivation for the results we present in Section~\ref{sec:Factorisation of Selmer groups for symmetric Rankin--Selberg products}); see also \cite{RiveroRotgerJEMS} for a related work.

Throughout Section~\ref{sec:Factorisation of Selmer groups for symmetric Rankin--Selberg products}, the notation we have set in \S\ref{sec:galois repn} are in effect.  In particular, we recall that we have fixed a branch character $\psi$ and its universal deformation $\Psi$ given as in Definition~\ref{defn_bigPsi}. We have set $M_{\Psi} := {\rm Ind}_{K/F}\Psi$, $M_{\Psi^{-1}} := {\rm Ind}_{K/F}\Psi^{-1}$, and 
\[
M({\rm ad}\Theta_{\Sigma,\psi})^{\iota} := (M_{\Psi} \otimes_{\Lambda_{\cO}(W_{K})}M_{\Psi^{-1}} )\otimes_{\Lambda_{\cO}({\bf G}')} \Lambda_{\cO}({\bf G}')^{\iota}. 
\]
Note that both $M_{\Psi}$ and $M_{\Psi^{-1}}$ are free $\LL_{\cO}(W_K)$-modules of rank $2$, whereas the $\LL_{\cO}(W_K)$-module $M({\rm ad}\Theta_{\Sigma,\psi})^{\iota}$ has rank $4$. We also recall that we have identified $W_K$ with $\Gamma_\infty$ via the tautological isomorphism $W_K \to \Gamma_\infty$.

Recall that $c\in \Gal(K/F)$ denotes the (unique) generator of this Galois group. By the definition of $M_{\Psi}$ and $M_{\Psi^{-1}}$, we can pick a basis $\{e_1, e_2\}$ (respectively, a basis $\{e_1^*, e_2^*\}$) of $M_{\Psi}$ (respectively, $M_{\Psi^{-1}}$) which verifies 
\begin{itemize}
    \item $c \cdot e_1 = e_2$ (respectively, $c \cdot e_1^* = e_2^*$), 
    \item $\sigma \cdot e_1 = \Psi(\sigma) e_1$ (respectively, $\sigma \cdot e_1^* = \Psi^{-1}(\sigma) e_1^*$) for any $\sigma \in G_{K}$, 
    \item $\sigma \cdot e_2 = \Psi^c(\sigma) e_2$ (respectively, $\sigma \cdot e_2^* = (\Psi^{-1})^c(\sigma) e_2^*$) for any $\sigma \in G_{K}$. 
\end{itemize}
Note that we have ${\rm span}_{\LL_{\cO}(W_K)}\{ e_1 \otimes e_1^*, e_2 \otimes e_2^* \} = {\rm Ind}_{K/F}\mathds{1}$ and ${\rm span}_{\LL_{\cO}(W_K)}\{  e_1 \otimes e_2^*, e_2 \otimes e_1^* \} = {\rm Ind}_{K/F}\Psi^{\rm ad}$.  
For notational simplicity, we shall put 
\begin{align}
\label{eqn_RS_repn_1} \TT &:= M({\rm ad}\Theta_{\Sigma,\psi})^{\iota}(1), 
\\
\label{eqn_RS_repn_2} \TT_1 &:= ({\rm Ind}_{K/F}\mathds{1} \otimes_{\Lambda_{\cO}({\bf G}')} \Lambda_{\cO}({\bf G}')^{\iota})(1), 
\\
\label{eqn_RS_repn_3} \TT_2 &:= ({\rm Ind}_{K/F}\Psi^{\rm ad} \otimes_{\Lambda_{\cO}({\bf G}')} \Lambda_{\cO}({\bf G}')^{\iota})(1)
\end{align}
so that we have 
\[
\TT = \TT_{1} \oplus \TT_{2}. 
\]

For each $v \in S_{p}(F)$, we denote by $\tilde{v}\in \Sigma$ the prime of $K$ lying above $v$. 
Via the identification $G_{F_{v}} = G_{K_{\tilde{v}}}$ and the choices of the bases $\{e_1,e_2\}$ and $\{e_1^*,e_2^*\}$, one has the decompositions
\begin{align*}
M_{\Psi}|_{G_{F_{v}}} &= \Psi|_{G_{K_{\tilde{v}}}} \oplus \Psi^c|_{G_{K_{\tilde{v}}}},  
\\ 
M_{\Psi^{-1}}|_{G_{F_{v}}} &= \Psi^{-1}|_{G_{K_{\tilde{v}}}} \oplus (\Psi^{-1})^c|_{G_{K_{\tilde{v}}}}.   
\end{align*}
\subsection{Selmer complexes for Rankin--Selberg convolutions}
\label{subsec_5_1_2021_09_10}
Our objective in this section is to define the algebraic Rankin--Selberg $p$-adic $L$-function associated to the self-product $\Theta_{\Sigma,\psi}^{\otimes 2}$ (as the Fitting ideals of appropriately chosen extended Selmer groups, in the sense of Nekov\'a\v{r}) and prove that it factors into a product of algebraic Katz' $p$-adic $L$-functions (see Theorem~\ref{thm:iwasama main conj rankin-selberg}(i) below). We shall use this result later in Section~\ref{subsec_selmer_traceless_adj} to prove the algebraic variant of the conjectural factorization \eqref{eqn_Dasgupta_CM}. We note in addition that the main conjectures for the family of $\Theta_{\Sigma,\psi}^{\otimes 2}$ follows from this factorization result combined with the Hida--Tilouine factorization \eqref{eqn_thm_HT_factorization_1} and the Main Conjectures~\ref{conj:iwasawa main conj} for the CM field $K$. We record this observation as Theorem~\ref{thm:iwasama main conj rankin-selberg}(ii) below.

\begin{definition}
\label{defn_Greenberg_Data}
\item[i)] For each prime $v \in S_p(F)$, 
define the $G_{F_{v}}$-submodule $F^+\TT$ of $\TT$ on setting
\begin{align*}
F^+\TT &:= ({\Psi} \otimes_{\Lambda_{\cO}(W_{K})}M_{\Psi^{-1}} )\otimes_{\Lambda_{\cO}({\bf G}')} \Lambda_{\cO}({\bf G}')^{\iota}(1). 
\end{align*}

\item[ii)] We set $S := S_{\rm ram}(\psi) \cup S_{\rm ram}(K/F) \cup S_p(F)$.

\item[iii)]
We define local conditions $\Delta$ for $C^{\bullet}(G_{F,S}, \TT)$ given by the choices 
\[
U_{v}^{+} := 
\begin{cases}
C^{\bullet}(G_{F_{v}},F^+\TT) & \text{if} \ v \in S_p(F), 
\\
C^{\bullet}_{\rm ur}(G_{F_{v}}, \TT) & \text{if} \ v \not\in S \setminus S_p(F). 
\end{cases}
\]
For any ideal $I$ of $\Lambda_{\cO}(W_{K})$, we  also define $F^+(\TT/I\TT)$ and 
local conditions $\Delta$ for $C^{\bullet}(G_{F,S}, \TT/I\TT)$ in an identical fashion. 
\end{definition}

We put 
\begin{align*}
\TT_{K,1} &:= (\mathds{1} \otimes_{\Lambda_{\cO}({\bf G}')} \Lambda_{\cO}({\bf G}')^{\iota})(1), 
\\
\TT_{K,2} &:=  (\Psi^{\rm ad} \otimes_{\Lambda_{\cO}({\bf G}')} \Lambda_{\cO}({\bf G}')^{\iota})(1).  
\end{align*}
Here, we denote by $\mathds{1}$ the free $\Lambda_{\cO}(W_{K})$-module of rank $1$ with  trivial Galois action. We note by Shapiro's Lemma that for any ideal $I$ of $\Lambda_{\cO}(W_{K})$, the decomposition 
$\TT/I\TT= \TT_{1}/I\TT_{1} \oplus \TT_{2}/I\TT_{2}$ induces a canonical identification 
\begin{align}\label{decomposition galois cohomology global}
{\bf R}\Gamma(G_{F,S}, \TT/I\TT) 
= {\bf R}\Gamma(G_{K, S_K}, \TT_{K,1}/I\TT_{K,1}) \oplus {\bf R}\Gamma(G_{K, S_K}, \TT_{K,2}/I\TT_{K,2}), 
\end{align}
where $S_K$ denotes the set of primes of $K$ above $S$.

\begin{lemma}\label{lem:decomposition local condition at p}
Let $v$ be a prime of $F$ above $p$ and fix an identification $G_{K_{\tilde{v}}} = G_{F_{v}}$. 
For any ideal $I$ of $\Lambda_{\cO}(W_{K})$, the decomposition $\TT/I\TT= \TT_{1}/I\TT_{1} \oplus \TT_{2}/I\TT_{2}$ induces an identification 
\begin{align*}
{\bf R}\Gamma(G_{F_{v}}, F^+(\TT/I\TT)) = {\bf R}\Gamma(G_{K_{\tilde{v}}}, \TT_{K, 1}/I\TT_{K, 1}) \oplus {\bf R}\Gamma(G_{K_{\tilde{v}^{c}}}, \TT_{K, 2}/I\TT_{K, 2}). 
\end{align*}
\end{lemma}
\begin{proof}
The argument in the general case being identical, we give a proof only when $I = 0$. Since 
\begin{align*}
(\Psi \otimes M_{\Psi^{-1}})|_{G_{F_{v}}} 
&= \Psi|_{G_{K_{\tilde{v}}}} \otimes (\Psi^{-1}|_{G_{K_{\tilde{v}}}} \oplus (\Psi^{-1})^c|_{G_{K_{\tilde{v}}}}) 
\\
&= \mathds{1}|_{G_{K_{\tilde{v}}}} \oplus (\Psi^{\rm ad})^{c}|_{G_{K_{\tilde{v}}}} 
\\
&=  \mathds{1}|_{G_{K_{\tilde{v}}}} \oplus \Psi^{\rm ad}|_{G_{K_{\tilde{v}^{c}}}}\,,
\end{align*}
the identification  $G_{K_{\tilde{v}}} = G_{F_{v}}$ induces a decomposition $F^+\TT|_{G_{F_{v}}} = \TT_{K,1}|_{G_{K_{\tilde{v}}}} \oplus \TT_{K, 2}|_{G_{K_{\tilde{v}^{c}}}}$. 
The proof of our lemma follows. 
\end{proof}

\begin{definition}
For any ideal $I$ of $\Lambda_{\cO}(W_{K})$, we define the local conditions $\Delta_{1}$ on the complex $C^{\bullet}(G_{K, S_{K}}, \TT_{K,1}/I\TT_{K, 1})$ by the choices
\[
U_{v}^{+} := 
\begin{cases}
C^{\bullet}(G_{K_{v}}, \TT_{K,1}/I\TT_{K, 1}) & \text{if} \ v \in \Sigma,
\\
0 & \text{if} \ v \in \Sigma^{c},
\\
C^{\bullet}_{\rm ur}(G_{F_{v}}, \TT_{K,1}/I\TT_{K, 1}) & \text{if} \ v \not\in S_{K} \setminus S_p(K). 
\end{cases}
\]
\end{definition}

Recall from Remark~\ref{remark:character} that the splitting of $W_K$
$$W_K=W_K^{c=1} \oplus W_K^{c=-1}$$
into eigenspaces for the $\langle c\rangle=\Gal(K/F)$-action induces an injective homomorphism of profinite groups ${\rm ver}_{\rm ac}: \Gamma_{\rm ac}\to W_K$, identifying $\Gamma_{\rm ac}$ with $W_K^{c=-1}$. Similarly, we have an injection 
\begin{equation}
\label{eqn_cyc_transfer_map}
{\rm ver}_{\rm +} \colon  {\bf G}' \to W_K
\end{equation} 
that identifies ${\bf G}'$ with $W_K^{c=1}$ and is given by $\gamma\mapsto \widetilde{\gamma}^{c+1}$ for any lift $\widetilde{\gamma}$ of $\gamma$.

\begin{lemma}\label{lemma:scalar-extension}
Suppose either that $G = G_{K, S_{K}}$ or else $G_{K_{v}}$ for some prime $v$ of $K$.  
We then have a canonical isomorphism 
\[
{\bf R}\Gamma(G, \Lambda_{\cO}({\bf G}')^{\iota}(1)) \otimes^{\bL}_{{\rm ver}_+} \Lambda_{\cO}(W_{K}) \stackrel{\sim}{\longrightarrow} {\bf R}\Gamma(G, \bT_{K, 1}). 
\]
\end{lemma}
\begin{proof}
For  notational simplicity, let us put (only in this proof) $A :=  \Lambda_{\cO}({\bf G}')$ and $B :=  \Lambda_{\cO}(W_{K})$. We let $\fm_{X}$ denote the maximal ideal of $X \in \{A, B\}$. Recall that $A^\iota$ is the free $A$-module of rank one on which $G_{K,S_K}$ acts via the contragredient of the tautological character $G_{K,S_K}\to \LL_{\cO}({\bf G}')^\times$.

Since $C^{\bullet}(G, A^{\iota}(1))$ is a perfect complex, there is a bounded complex $F^{\bullet}$ of finitely generated free $A$-modules with a quasi-isomorphism $F^{\bullet}  \longrightarrow C^{\bullet}(G, A^{\iota}(1))$. 
Let us put 
\[
D^{\bullet} := {\rm Cone}\left(F^{\bullet}  \longrightarrow C^{\bullet}(G, A^{\iota}(1))\right). 
\]
Note that $D^{\bullet}$ is acyclic. 
Since the canonical morphism 
\[
{\bf R}\Gamma(G, A^{\iota}(1)) \otimes^{\bL}_{A} A/\fm_{A} \longrightarrow {\bf R}\Gamma(G, A^{\iota}(1) \otimes_{A} A/\fm_{A})
\] 
is an isomorphism, the composite map 
\[
F^{\bullet} \otimes_{A} A/\fm_{A} \longrightarrow C^{\bullet}(G, A^{\iota}(1)) \otimes_{A} A/\fm_{A}  
\cong C^{\bullet}(G, A^{\iota}(1) \otimes_{A} A/\fm_{A}) 
\]
is a quasi-isomorphism. We therefore conclude that the complex $D^{\bullet} \otimes_{A} A/\fm_{A}$ is acyclic. 
This fact shows for any integer $n>0$ that the canonical map 
\[
F^{\bullet} \otimes_{A} B/\fm^{n}_{B} \longrightarrow C^{\bullet}(G, A^{\iota}(1)) \otimes_{A} B/\fm^{n}_{B}
\]
is a quasi-isomorphism as well.  
Furthermore, since $B/\fm^{n}_{B}$ is discrete, we have a natural isomorphism 
\[
C^{\bullet}(G, A^{\iota}(1)) \otimes_{A} B/\fm_{B}^{n} \stackrel{\sim}{\longrightarrow} C^{\bullet}(G, A^{\iota}(1) \otimes_{A} B/\fm_{B}^{n}). 
\]
We conclude that the morphism $F^{\bullet} \otimes_{A} B/\fm_{B}^{n} \longrightarrow C^{\bullet}(G, A^{\iota}(1) \otimes_{A} B/\fm_{B}^{n})$ is a quasi-isomorphism, and on passing to limit, we obtain a quasi-isomorphism 
\[
F^{\bullet} \otimes_{A} B \cong \varprojlim_{n>0}F^{\bullet} \otimes_{A} B/\fm_{B}^{n}  
\longrightarrow  \varprojlim_{n>0}C^{\bullet}(G, A^{\iota}(1) \otimes_{A} B/\fm_{B}^{n})   \cong C^{\bullet}(G, A^{\iota}(1) \otimes_{A} B). 
\]
This quasi-isomorphism descends to the required isomorphism in the derived category.
\end{proof}

Using Lemma~\ref{lemma:scalar-extension}, we obtain a canonical identification 
\begin{align}
\begin{split}\label{eq:selmer}
\widetilde{{\bf R}\Gamma}_{\rm f}(G_{K, S_K}, \TT_{K, 1}, \Delta_{1}) 
&=  \widetilde{{\bf R}\Gamma}_{\rm f}(G_{K, S_K},  \Lambda_{\cO}({\bf G}')^{\iota}(1), \Delta_{\Sigma^{c}}) \otimes^{\bL}_{\Lambda_{\cO}({\bf G}')} \Lambda_{\cO}(W_{K}) 
\\
&\xrightarrow{\sim} \widetilde{{\bf R}\Gamma}_{\rm f}(G_{K, S_K},  \Lambda_{\cO}({\bf G}')^{\iota}(1), \Delta_{\Sigma}) \otimes^{\bL}_{\Lambda_{\cO}({\bf G}')} \Lambda_{\cO}(W_{K}), 
\end{split}
\end{align}
where the  isomorphism on the second row is induced by the generator $c \in \Gal({K}/F)$.

The generator $c$ acts on $\Gamma_{\rm ac}$ via $\gamma \mapsto \gamma^{-1}$. Therefore, it induces an isomorphism between the $G_{K}$-representations $\Lambda_{\cO}(\Gamma_{\rm ac})^{\#}$ and  $\Lambda_{\cO}(\Gamma_{\rm ac})^{\iota}$, where $\Lambda_{\cO}(\Gamma_{\rm ac})^{\#}$ is the free $\Lambda_{\cO}(\Gamma_{\rm ac})$-module of rank one on which $G_K$ acts via the tautological character $G_K \to \Lambda_{\cO}(\Gamma_{\rm ac})^\times$, whereas $\Lambda_{\cO}(\Gamma_{\rm ac})^{\iota}$ is its contragredient. In particular, as explained in Remark~\ref{remark:character}, we have an identification  
\[
\TT_{K,2} = (\cO(\psi^{\rm ad}) \otimes_{\cO} \Lambda_{\cO}(W_{K})^{\iota})(1), 
\]
and we can define via this identification the local conditions $\Delta_{\Sigma}$ for $C^{\bullet}(G_{K, S_{K}}, \TT_{K, 2})$ to be the one that corresponds to the choice of local conditions given in Definition~\ref{def:extended selmer} (with $\chi^{-1}=\psi^{\rm ad}$).

\begin{theorem}\label{thm:decomposition selmer complexes}
For any ideal $I$ of $\Lambda_{\cO}(W_{K})$, 
one has a natural isomorphism 
\[
\widetilde{{\bf R}\Gamma}_{\rm f}(G_{F,S}, \TT/I\TT, \Delta) \xrightarrow{\sim} 
\widetilde{{\bf R}\Gamma}_{\rm f}(G_{K, S_K}, \TT_{K,1}/I\TT, \Delta_{1}) \oplus \widetilde{{\bf R}\Gamma}_{\rm f}(G_{K, S_K}, \TT_{K, 2}/I\TT, \Delta_{\Sigma}). 
\]
\end{theorem}
\begin{proof}
The argument in the general case being identical, we give a proof only when $I = 0$. 
We note for each prime $v \nmid p$ of $F$ that we have a natural decomposition 
\begin{align*}
{\bf R}{\Gamma}_{\rm ur}(G_{F_{v}}, \TT) 
&= 
{\bf R}{\Gamma}_{\rm ur}(G_{F_{v}}, \TT_{1})  \oplus 
{\bf R}{\Gamma}_{\rm ur}(G_{F_{v}}, \TT_{2}) 
\\
&= \bigoplus_{w \mid v} {\bf R}{\Gamma}_{\rm ur}(G_{K_{w}}, \TT_{K, 1}) \oplus  \bigoplus_{w \mid v} {\bf R}{\Gamma}_{\rm ur}(G_{K_{w}}, \TT_{K, 2}). 
\end{align*}
This fact together with \eqref{decomposition galois cohomology global} and Lemma~\ref{lem:decomposition local condition at p} concludes the proof of the theorem. 
\end{proof}

\begin{corollary}\label{parf[1,2] and Euler--Poincare = 0}
    \item[i)] For any ideal $I$ of $\Lambda_{\cO}(W_{K})$, there is a natural isomorphism 
\begin{align*}
    \widetilde{{\bf R}\Gamma}_{\rm f}(G_{F,S}, \bT, \Delta) \otimes^{\bL}_{\Lambda_{\cO}(W_{K})} \Lambda_{\cO}(W_{K})/I \stackrel{\sim}{\longrightarrow} \widetilde{{\bf R}\Gamma}_{\rm f}(G_{F,S}, \bT/I\bT, \Delta)\,,
\end{align*}
    as well as a natural isomorphism 
\begin{align*}
\widetilde{{\bf R}\Gamma}_{\rm f}(G_{F,S}, \bT, \Delta) \otimes^{\bL}_{\Lambda_{\cO}(W_{K})} \Lambda_{\bZ_{p}^{\rm ur}}(W_{K}) \stackrel{\sim}{\longrightarrow} \widetilde{{\bf R}\Gamma}_{\rm f}(G_{F,S}, \bT \otimes_{\Lambda_{\cO}(W_{K})} \Lambda_{\bZ_{p}^{\rm ur}}(W_{K}), \Delta).
\end{align*}
\item[ii)] $\widetilde{{\bf R}\Gamma}_{\rm f}(G_{F,S}, \TT, \Delta) \in D_{\rm parf}^{[1,2]}(_{\Lambda_{\cO}(W_{K})}{\rm Mod})$. 
\item[iii)] $\chi(\widetilde{{\bf R}\Gamma}_{\rm f}(G_{F,S}, \TT, \Delta)) = 0$.
 
\end{corollary}
\begin{proof}
By Theorem~\ref{thm:decomposition selmer complexes} and \eqref{eq:selmer}, we have the following identifications: 
\begin{align*}
\widetilde{{\bf R}\Gamma}_{\rm f}(G_{F,S}, \TT/I\TT, \Delta) & \xrightarrow{\sim}   
\widetilde{{\bf R}\Gamma}_{\rm f}(G_{K, S_K}, \TT_{K,1}/I\TT_{K, 1}, \Delta_{1}) \oplus \widetilde{{\bf R}\Gamma}_{\rm f}(G_{K, S_K}, \TT_{K, 2}/I\TT_{K, 2}, \Delta_{\Sigma})
\\
\widetilde{{\bf R}\Gamma}_{\rm f}(G_{K, S_K}, \TT_{K,1}/I\TT_{K, 1}, \Delta_{1}) &= \widetilde{{\bf R}\Gamma}_{\rm f}(G_{K, S_K},  \Lambda_{\cO}({\bf G}')^{\iota}(1), \Delta_{\Sigma^{c}}) \otimes^{\bL}_{\Lambda_{\cO}({\bf G}')} \Lambda_{\cO}(W_{K})/I. 
\end{align*}
The first assertion in (i) now follows from the fundamental base-change property \eqref{eq:base-change} for the Selmer complex $\widetilde{{\bf R}\Gamma}_{\rm f}(G_{K, S_K}, \TT_{K, 2}, \Delta_{\Sigma})$.  The second assertion in (i) follows from the isomorphism~\eqref{eq:scalar-extension}. Proposition~\ref{prop:parf[1,2]} shows that $\widetilde{{\bf R}\Gamma}_{\rm f}(G_{F,S}, \TT, \Delta) \in D_{\rm parf}^{[1,2]}(_{\Lambda_{\cO}(W_{K})}{\rm Mod})$, which is (ii), whereas Lemma~\ref{lemma:Euler characteristic = 0} shows that $\chi(\widetilde{{\bf R}\Gamma}_{\rm f}(G_{F,S}, \TT, \Delta)) = 0$, which is (iii). 
\end{proof}

\begin{corollary}\label{cor:delta-torison}
\item[i)] $\widetilde{H}^{1}_{\rm f}(G_{F,S}, \TT, \Delta)=0$. 
\item[ii)] The $\Lambda_{\cO}(W_K)$-module $\widetilde{H}^{2}_{\rm f}(G_{F,S}, \TT, \Delta)$ is torsion. 
\end{corollary}
\begin{proof}
As in the proof of Corollary~\ref{parf[1,2] and Euler--Poincare = 0}, we start with the observation that we have
\begin{align}
\label{eqn_cor_57_1}
\widetilde{{\bf R}\Gamma}_{\rm f}(G_{F,S}, \TT, \Delta) &\xrightarrow{\sim} 
\widetilde{{\bf R}\Gamma}_{\rm f}(G_{K, S_K}, \TT_{K,1}, \Delta_{1}) \oplus \widetilde{{\bf R}\Gamma}_{\rm f}(G_{K, S_K}, \TT_{K, 2}, \Delta_{\Sigma})
\end{align}
by Theorem~\ref{thm:decomposition selmer complexes}. By Corollary~\ref{cor:H^1_f=0 and H^2_f is torsion}(i), we have 
$$\widetilde{H}^{1}_{\rm f}(G_{K, S_{K}}, \TT_{K, 1}, \Delta_{1}) = \widetilde{H}^{1}_{\rm f}(G_{K, S_{K}}, \TT_{K, 2}, \Delta_{\Sigma}) = 0,$$ 
which shows combined with \eqref{eqn_cor_57_1} that 
\[
\widetilde{H}^{1}_{\rm f}(G_{F,S}, \TT, \Delta) \cong 
\widetilde{H}^{1}_{\rm f}(G_{K, S_{K}}, \TT_{K, 1}, \Delta_{1}) \oplus \widetilde{H}^{1}_{\rm f}(G_{K, S_{K}}, \TT_{K, 2}, \Delta_{\Sigma}) = 0. 
\]
This concludes the proof of (i).

To prove (ii), we note that both $\Lambda_{\cO}(W_{K})$-modules $\widetilde{H}^{2}_{\rm f}(G_{K, S_{K}}, \TT_{K, 1}, \Delta_{1})$ and $\widetilde{H}^{2}_{\rm f}(G_{K, S_{K}}, \TT_{K, 2}, \Delta_{\Sigma})$ are torsion thanks to Corollary~\ref{cor:H^1_f=0 and H^2_f is torsion}(ii). 
We therefore conclude using \eqref{eqn_cor_57_1} that the $\Lambda_{\cO}(W_{K})$-module 
\[
\widetilde{H}^{2}_{\rm f}(G_{F,S}, \TT, \Delta) \xrightarrow{\sim} 
\widetilde{H}^{2}_{\rm f}(G_{K, S_{K}}, \TT_{K, 1}, \Delta_{1}) \oplus \widetilde{H}^{2}_{\rm f}(G_{K, S_{K}}, \TT_{K, 2}, \Delta_{\Sigma})
\] 
is torsion as well. 
\end{proof}

\begin{definition}
\label{defn_algebraic_padic_Rankin_Selberg}
We define the non-critical specialization of the algebraic Rankin--Selberg $p$-adic $L$-function $$\widetilde{\mathscr{D}}^{\rm alg}\in {\rm Frac}(\LL_{\cO}(W_K))^\times \big{/}\LL_{\cO}(W_K)^\times$$
as a generator of the cyclic $\LL_{\cO}(W_K)$-module  ${\rm Fitt}_{\Lambda_{\cO}(W_{K})}\left(\widetilde{H}^{2}_{\rm f}(G_{F, S}, \TT, \Delta)\right)$.
\end{definition}

\begin{remark}
\item[i)] In view of Definition~\ref{defn_algebraic_Katz_padic_L}, \eqref{eqn_RS_repn_3} and the choice of the local conditions $\Delta_\Sigma$ on $C^\bullet(G_{F,S},\bT_2)$, recall that we have the algebraic Katz $p$-adic $L$-function 
$$ L_{p, \psi^{\rm ad}}^{\rm alg}\in {\rm Frac}(\LL_{\cO}(W_K))^\times \big{/}\LL_{\cO}(W_K)^\times$$ 
as a  generator of the cyclic $\LL_{\cO}(W_K)$-module  ${\rm Fitt}_{\Lambda_{\cO}(W_{K})}\left(\widetilde{H}^{2}_{\rm f}(G_{F, S}, \TT_{K, 2}, \Delta_{\Sigma})\right)$.
\item[ii)] Likewise, ${\rm Fitt}_{\Lambda_{\cO}({\bf G}')}(\widetilde{H}^{2}_{\rm f}(G_{K, S_K}, \Lambda_{\cO}({\bf G}')^{\iota}(1), \Delta_{\Sigma}))$ is generated by the restriction 
\[
L_{p, \mathds{1}^{\rm alg}}\vert_{{\bf G}'} \in {\rm Frac}(\LL_{\cO}({\bf G}'))^\times \big{/}\LL_{\cO}({\bf G}')^\times
\]
of the algebraic Katz $p$-adic $L$-function $L_{p, \mathds{1}^{\rm alg}}$ to ${\bf G}'$. 
\end{remark}

Recall from \eqref{eqn_cyc_transfer_map} the injection ${{\rm ver}_{+}} \colon \LL_{\cO}({\bf G}')\to \LL_{\cO}(W_K)$ 
and note that we have ${\bf G}' = \Gamma_{\rm cyc}$ and $\Gamma_\infty^\circ = \Gamma_\infty$ if the Leopoldt conjecture for $F$ is valid.

\begin{theorem}
\label{thm:iwasama main conj rankin-selberg}
\item[i)] The algebraic Rankin--Selberg $p$-adic $L$-function factors as
\[
\widetilde{\mathscr{D}}^{\rm alg}
= {\rm ver}_{+}(L_{p, \mathds{1}}^{\rm alg}\vert_{{\bf G}'})\cdot 
  L_{p, \psi^{\rm ad}}^{\rm alg}, 
\]
where the equality takes place in ${\rm Frac}(\LL_{\cO}(W_K))^\times \big{/}\LL_{\cO}(W_K)^\times$. 

\item[ii)] If the Main Conjecture~\ref{conj:iwasawa main conj} for the CM field $K$ holds true, then  $$\widetilde{\mathscr{D}}^{\rm alg}= \widetilde{\mathscr{D}}^{\rm Hida}\quad \in\, {\rm Frac}(\LL_{\cO}(\Gamma_\infty^\circ))^\times \big{/}\LL_{\cO}(\Gamma_\infty^\circ)^\times,$$ 
 where we regard $\widetilde{\mathscr{D}}^{\rm alg}$, which is a priori defined as an element of ${\rm Frac}(\LL_{\cO}(W_K))^\times \big{/}\LL_{\cO}(W_K)^\times$, also as an element of ${\rm Frac}(\LL_{\cO}(\Gamma_\infty^\circ))^\times \big{/}\LL_{\cO}(\Gamma_\infty^\circ)^\times$ via the natural map induced from the canonical surjection $W_K\twoheadrightarrow \Gamma_\infty^\circ$.
\end{theorem}

\begin{proof}
The first assertion is an immediate consequence of \eqref{eq:selmer} and Theorem~\ref{thm:decomposition selmer complexes}. The second assertion follows on combining (i), \eqref{eq:scalar-extension}, and Corollary \ref{parf[1,2] and Euler--Poincare = 0}(ii). 
\end{proof}

\begin{remark}
\item[i)] The factorization in Theorem~\ref{thm:iwasama main conj rankin-selberg}(i) is the algebraic counterpart of the factorization result \eqref{eqn_thm_HT_factorization_1} of Hida and Tilouine, which factors Hida's Rankin--Selberg $p$-adic $L$-function attached to a pair of nearly ordinary CM families of Hilbert Modular forms to a product of two Katz $p$-adic $L$-function. 
\item[ii)] Note that the identity in Theorem~\ref{thm:iwasama main conj rankin-selberg}(ii) relates the non-critical specialization $\widetilde{\mathscr D}^{\rm Hida}$ of Hida's Rankin--Selberg $p$-adic $L$-function to the non-critical specialization $\widetilde{\mathscr D}^{\rm alg}$ of the algebraic Rankin--Selberg $p$-adic $L$-function (c.f. the first paragraph of \S\ref{subsubsec_intro_3} and Remark~\ref{rem_Hsieh_2_Intro}). We refrain from referring it as a ``main conjecture'' since the objects involved are both non-critical specializations.  Specializations of this identity are particular instances of Perrin-Riou's $p$-adic Beilinson conjectures in the presence of exceptional zeros.
\end{remark}

\subsection{Interlude: Algebraic version of Gross' factorization for CM fields}
\label{sec_agfCM}
Our objective in Section~\ref{sec_agfCM} is to prove the algebraic counterpart of the expected factorization \eqref{eqn_Gross_CM} for the cyclotomic restriction of Katz' $p$-adic $L$-function for CM fields. This result is recorded in Theorem~\ref{thm_algebraic_gross_factorization} below. We recall that \eqref{eqn_Gross_CM} predicts an extension of Gross' factorization~\cite{Gross1980Factorization} of Katz' $p$-adic $L$-function for imaginary quadratic fields to general CM fields. As we have indicated in Remark~\ref{remark_gross_dasgupta}, this enables us to prove the algebraic counterpart of the expected factorization \eqref{eqn_Dasgupta_CM}; see Theorem~\ref{theorem_adjoint} below. 

Throughout Section~\ref{sec_agfCM}, we will work with the $G_F$-representation
\[
\bT_F^\cyc := {\rm Ind}_{K/F}(\bZ_{p}(1)) \otimes_{\bZ_{p}} \Lambda_{\bZ_{p}}(\Gamma_{\rm cyc})^{\iota}.  
\]
We write $\epsilon_{K/F} \colon G_{F} \longrightarrow \{\pm 1 \}$ for the quadratic character attached to the extension $K/F$. 
 We then have a canonical identification 
\[
{\rm Ind}_{K/F}(\bZ_{p}(1)) = \bZ_{p}(1) \oplus \left(\bZ_{p}(1) \otimes \epsilon_{K/F}\right). 
\]
For notational simplicity, we shall set 
\begin{align*}
\Lambda &:= \Lambda_{\bZ_{p}}(\Gamma_{\rm cyc}), 
\\
S &:= S_{\rm ram}(K/F) \cup S_{p}(F), 
\\
\bT_{\mathds 1}^\cyc &:= \bZ_{p}(1) \otimes_{\bZ_{p}} \Lambda^{\iota}, 
\\
\bT_{\epsilon}^\cyc &:= (\bZ_{p}(1) \otimes \epsilon_{K/F}) \otimes_{\bZ_{p}} \Lambda^{\iota}. 
\end{align*}

We fix a $p$-ordinary CM type $\Sigma$ as before. 
For each prime $v \in S_{p}(F)$, we denote by $\tilde{v}\in \Sigma$ the prime that lies above $v$. 
By Shapiro's lemma, the identification $G_{K_{\tilde{v}}} = G_{F_{v}}$ induces a decomposition 
\begin{align*}
C^{\bullet}(G_{F_{v}}, \bT_F^\cyc) &= C^{\bullet}(G_{F_{v}}, {\rm Ind}_{K/F}\bT_{\mathds 1}^\cyc) \\
&= C^{\bullet}(G_{K_{\tilde{v}}}, \bT_{\mathds 1}^\cyc) \oplus C^{\bullet}(G_{K_{{\tilde{v}^{c}}}}, \bT_ {\mathds 1}^\cyc)
\end{align*}
We note that this decomposition is \emph{not} stable under the action of $\Gal(K/F)$ and that 
$C^{\bullet}(G_{F_{v}}, \bT_{\mathds 1}^\cyc) = C^{\bullet}(G_{F_{v}}, \bT)^{c=1}$ and $C^{\bullet}(G_{F_{v}}, \bT_{\epsilon}^\cyc) = C^{\bullet}(G_{F_{v}}, \bT_F^\cyc)^{c=-1}$. 
In particular, this decomposition is \emph{not} consistent with that induced from the decomposition $\bT_F^\cyc = \bT_{\mathds 1}^\cyc \oplus \bT_{\epsilon}^\cyc$. 

\begin{definition}
We define the local condition $\Delta_{\Sigma}$ on $C^{\bullet}(G_{F,S}, \bT_F^\cyc)$ with the choices
\[
U_{v}^{+} := 
\begin{cases}
C^{\bullet}(G_{K_{\tilde{v}}}, \bT_{\mathds 1}^\cyc) & \text{if} \ v \in S_{p}(F), 
\\
C^{\bullet}_{\rm ur}(G_{F_{v}}, \bT_F^\cyc) & \text{if} \ v \in S \setminus S_{p}(F). 
\end{cases}
\]
We note that the natural isomorphism ${\bf R}\Gamma(G_{F,S}, \bT_F^\cyc) \cong {\bf R}\Gamma(G_{K,S_{K}}, \bT_{\mathds 1}^\cyc)$ induced by Shapiro's Lemma gives rise to an isomorphism 
\[
\widetilde{{\bf R}\Gamma}_{\rm f}(G_{F,S}, \bT_F^\cyc, \Delta_{\Sigma}) \cong \widetilde{{\bf R}\Gamma}_{\rm f}(G_{K,S_{K}}, \bT_{\mathds 1}^\cyc, \Delta_{\Sigma}),  
\]
where we recall that $S_{K}$ denotes the set of primes of $K$ above $S$. 
\end{definition}

\begin{definition}\ 
We define
\[
\widetilde{C}^{\bullet}_{\rm c}(G_{F,S}, \bT_{\mathds 1}^\cyc) := {\rm Cone}\left( C^{\bullet}(G_{F,S}, \bT_{\mathds 1}^\cyc) \oplus  \bigoplus_{v \in S \setminus S_{p}(F)} C^{\bullet}_{\rm ur}(G_{F_{v}}, \bT_{\mathds 1}^\cyc) \longrightarrow  \bigoplus_{v \in S  } C^{\bullet}(G_{F_{v}}, \bT_{\mathds 1}^\cyc)  \right)[-1]. 
\]
and set 
\begin{align*}
\widetilde{C}^{\bullet}(G_{F,S}, \bT_{\epsilon}^\cyc) := 
{\rm Cone}\left( C^{\bullet}(G_{F,S}, \bT_{\epsilon}^\cyc) \oplus \bigoplus_{v \in S_{p}(F)} C^{\bullet}(G_{F_{v}}, \bT_{\epsilon}^\cyc) 
 \longrightarrow  \bigoplus_{v \in S} C^{\bullet}(G_{F_{v}}, \bT_{\epsilon}^\cyc)  \right)[-1]. 
\end{align*}
\end{definition}

On identifying $\bT_{\epsilon}^\cyc$ with the quotient $\bT_F^\cyc/\bT_{\mathds 1}^\cyc$, we have an exact sequence 
\begin{equation}
\label{eqn_TFcyc_extact_sequence}
0 \longrightarrow \bT_{\mathds 1}^\cyc \longrightarrow \bT_F^\cyc \longrightarrow \bT_{\epsilon}^\cyc \longrightarrow 0\,. 
\end{equation}

\begin{proposition}\label{prop:decomp-gross}
The exact sequence \eqref{eqn_TFcyc_extact_sequence} induces an exact sequence of complexes 
\[
0 \longrightarrow \widetilde{C}^{\bullet}_{\rm c}(G_{F,S}, \bT_{\mathds 1}^\cyc) \longrightarrow \widetilde{C}^{\bullet}_{\rm f}(G_{F,S}, \bT_F^\cyc, \Delta_{\Sigma}) \longrightarrow \widetilde{C}^{\bullet}(G_{F,S}, \bT_{\epsilon}^\cyc) \longrightarrow 0\,.
\]
\end{proposition}

\begin{proof}
The projection $\bT_F^\cyc \longrightarrow \bT_{\epsilon}^\cyc$ given by the exact sequence \eqref{eqn_TFcyc_extact_sequence} induces an isomorphism 
$$C^{\bullet}(G_{K_{\tilde{v}}}, \bT_{\mathds 1}^\cyc) \stackrel{\sim}{\longrightarrow} C^{\bullet}(G_{F_{v}}, \bT_{\epsilon}^\cyc) $$ 
for each prime $v \in S_{p}(F)$. Moreover, for each prime $v \in S_{\rm ram}(K/F) \setminus S_{p}(F)$ we have $(\bT_{\epsilon}^\cyc)^{I_{v}}=0$,  and therefore infer that $C^{\bullet}_{\rm ur}(G_{F_{v}}, \bT_{\mathds 1}^\cyc) = C^{\bullet}_{\rm ur}(G_{F_{v}}, \bT_F^\cyc)$. Our assertion from these facts.
\end{proof}

\begin{lemma}
\item[i)] The inflation map $C^{\bullet}(G_{F,S_{p}(F)}, \bT_{\mathds 1}^\cyc) \longrightarrow  C^{\bullet}(G_{F,S}, \bT_{\mathds 1}^\cyc)$ induces a quasi-isomorphism
\[
{C}_{\rm c}^{\bullet}(G_{F, S_{p}(F)}, \bT_{\mathds 1}^\cyc) \xrightarrow{{\rm Qis}}  \widetilde{C}^{\bullet}_{\rm c}(G_{F,S}, \bT_{\mathds 1}^\cyc),
\]
where ${C}_{\rm c}^{\bullet}(G_{F, S_{p}(F)}, \bT_{\mathds 1}^\cyc)$ denotes the cochains with compact support, defined as in \cite[Definition 5.3.1.1]{Nek}. 
\item[ii)] The complex $\widetilde{C}^{\bullet}(G_{F,S}, \bT_{\epsilon}^\cyc)$ is a quasi-isomorphic to $C^{\bullet}(G_{F,S}, \bT_{\epsilon}^\cyc)$.  
\end{lemma}
\begin{proof}
Since $S_{\rm ram}(\bT_{\mathds 1}^\cyc) = S_{p}(F)$, the first assertion is clear. The second part follows from the acyclicity of  the complex $C^{\bullet}(G_{F_{v}}, \bT_{\epsilon}^\cyc)$ for any prime $v \in S_{\rm ram}(K/F)$ (see Remark~\ref{remark_local_cohom_bad_primes_acyclic}). 
\end{proof}

\begin{lemma}\label{lemma:compact-supp-coh} The following statements are valid.
\item[i)] $\chi({\bf R}\Gamma_{\rm c}(G_{F,S}, \bT_{\mathds 1}^\cyc)) = 0$. 
\item[ii)] $H^{0}_{\rm c}(G_{F,S}, \bT_{\mathds 1}^\cyc)=0=H^{1}_{\rm c}(G_{F,S}, \bT_{\mathds 1}^\cyc)$. 
\item[iii)] We have an isomorphism $H^{3}_{\rm c}(G_{F,S}, \bT_{\mathds 1}^\cyc) \cong \bZ_{p}$. 
\item[iv)] The  ${\Lambda}$-module $H^{2}_{\rm c}(G_{F,S}, \bT_{\mathds 1}^\cyc)$ is  torsion. 
\end{lemma}
\begin{proof}
\item[i)] The computation of the global Euler-Poincar\'e characteristic in \cite[Theorem~7.8.6]{Nek} shows that
\[
\chi({\bf R}\Gamma_{\rm c}(G_{F,S},\bT_{\mathds 1}^\cyc)) = 
\sum_{v \in S_\infty(F)}1 - \sum_{v \in S_{p}(F)}[F_v \colon \bQ_p] 
= 0,
\]
as required.
\item[ii)] Note that when $G = G_{F,S}$ or $G_{F_{v}}$ for some $v \in S_{p}(F)$,  we have $H^{0}(G, \bT_{\mathds 1}^\cyc)=0$ by \cite[Lemma~B.3.2]{rubin00}. 
In particular, $H^{0}_{\rm c}(G_{F,S}, \bT_{\mathds 1}^\cyc)=0$ as well and the vanishing of $H^{1}_{\rm c}(G_{F,S}, \bT_{\mathds 1}^\cyc)$ follows from this fact and the weak Leopoldt conjecture (which is a theorem of Iwasawa in this set up). 

\item[iii)] This assertion follows on observing that
\[
H^{3}_{\rm c}(G_{F,S}, \bT_{\mathds 1}^\cyc) \cong 
\Hom_{\rm cont}(H^{0}(G_{F,S}, \Hom_{\rm cont}(\bT_{\mathds 1}^\cyc, \bQ_{p}/\bZ_{p})(1)), \bQ_{p}/\bZ_{p}) = \bZ_{p}. 
\]
where the first isomorphism is a consequence of global duality.
\item[iv)] Observe that ${\bf R}\Gamma_{\rm c}(G_{F,S}, \bT_{\mathds 1}^\cyc) \in D_{\rm parf}^{[0,3]}(_{\Lambda}{\rm Mod})$ by its very definition. Since 
$\chi({\bf R}\Gamma_{\rm c}(G_{F,S}, \bT_{\mathds 1}^\cyc)) = 0$ by (i), this statement follows on combining (ii) and (iii).  
\end{proof}

We are next set to prove that the $\Lambda$-modules $H^{1}(G_{F,S}, \bT_{\epsilon}^\cyc)$ and $H^{2}(G_{F,S}, \bT_{\epsilon}^\cyc)$ are both torsion. For each prime $v \in S_{p}(F)$, Shapiro's lemma yields a canonical identification  
\[
H^{1}(G_{F_{v}}, \bT_{\epsilon}^\cyc) = \varprojlim_{n> 0} \bigoplus_{w \mid v} H^{1}(G_{F_{n, w}}, \bZ_{p}(1) \otimes \epsilon_{K/F}), 
\]
where $F_{n}$ is the $n$-th layer of  the cyclotomic $\ZZ_p$-tower $F_{\rm cyc}/F$ and $w$ runs over the set of primes of $F_{n}$ that lie above $v$. 
We put 
\begin{align*}
    H^{1}_{\rm f}(G_{F_{v}}, \bT_{\epsilon}^\cyc) &:= \varprojlim_{n> 0} \bigoplus_{w \mid v} H^{1}_{\rm f}(G_{F_{n, w}}, \bZ_{p}(1) \otimes \epsilon_{K/F}), 
    \\
H^{1}_{/{\rm f}}(G_{F_{v}}, \bT_{\epsilon}^\cyc) &:= H^{1}(G_{F_{v}}, \bT_{\epsilon}^\cyc)/H^{1}_{\rm f}(G_{F_{v}}, \bT_{\epsilon}^\cyc). 
\end{align*}

\begin{lemma}\label{lemma:local-isom-torsion}
For each prime $v \in S_{p}(F)$, 
the $\Lambda$-module $H^{1}_{/{\rm f}}(G_{F_{v}}, \bT_{\epsilon}^\cyc)$ is torsion and 
we have an isomorphism $H^{1}_{/{\rm f}}(G_{F_{v}}, \bT_{\epsilon}^\cyc) \cong H^{2}(G_{F_{v}}, \bT_{\epsilon}^\cyc)$.  
\end{lemma}
\begin{proof}
For an integer $n > 0$, let $K_{n}$ denote the $n$-th layer of the cyclotomic $\ZZ_p$-tower $K_{\rm cyc}/K$ and 
$S_{v}(K_{n})$  the set of primes of $K_{n}$ above $v$. 
We have a pair of isomorphisms 
\[
\bigoplus_{w \in S_{v}(K_{n})} H^{1}_{/{\rm f}}(G_{F_{n, w}}, \bZ_{p}(1) \otimes \epsilon_{K/F}) \xrightarrow[\ord]{\sim}
\left(\bigoplus_{w \in S_{v}(K_{n})}\bZ_{p}\right)^{\epsilon_{F/K}} \xleftarrow[{\rm inv}]{\sim} \bigoplus_{w \in S_{v}(K_{n})} H^{2}(G_{F_{n, w}}, \bZ_{p}(1) \otimes \epsilon_{K/F}). 
\]
where the isomorphism on the right is induced from the local invariant map.

Note that the extension $K_{n+1}/K_{n}$ is totally ramified at any prime above $p$ for all sufficiently large values of $n$ (say for all $n\geq N$). Therefore, the transition maps 
$$\bigoplus H^{1}_{/{\rm f}}(G_{F_{n+1, w}}, \bZ_{p}(1) \otimes \epsilon_{K/F}) \longrightarrow \bigoplus H^{1}_{/{\rm f}}(G_{F_{n, w}}, \bZ_{p}(1) \otimes \epsilon_{K/F})$$ 
are surjective and $|S_{v}(K_{n})|=|S_{v}(K_{N})|$ for $n\geq N$. 
This shows  that 
$$H^{1}_{/{\rm f}}(G_{F_{v}}, \bT_{\epsilon}^\cyc)\xrightarrow[\ord]{\sim} \left(\oplus_{w \in S_{v}(K_{N})}\bZ_{p}\right)^{\epsilon_{F/K}}$$ 
and in particular, we conclude that the $\LL$-module $H^{1}_{/{\rm f}}(G_{F_{v}}, \bT_{\epsilon}^\cyc)$ is torsion. 

Note also that the cohomological dimension of the absolute Galois group of a $p$-adic local field is $2$ and hence, the transition maps $\bigoplus H^{2}(G_{F_{n+1, w}}, \bZ_{p}(1) \otimes \epsilon_{K/F}) \longrightarrow \bigoplus H^{2}(G_{F_{n, w}}, \bZ_{p}(1) \otimes \epsilon_{K/F})$ are surjective for all $n\geq N$. We therefore conclude that $H^{2}(G_{F_{v}}, \bT_{\epsilon}^\cyc)\xrightarrow[{\rm inv}]{\sim} \left(\oplus_{w \in S_{v}(K_{N})}\bZ_{p}\right)^{\epsilon_{F/K}}$. Combined with the discussion in the previous paragraph, this completes our proof. 
\end{proof}

\begin{lemma}\label{lemma:coh-beta}
\item[i)] $\chi({\bf R}\Gamma(G_{F,S}, \bT_{\epsilon}^\cyc)) = 0$. 
\item[ii)] $H^{0}(G_{F,S}, \bT_{\epsilon}^\cyc)=0$.
\item[iii)] Both $\Lambda$-modules $H^{1}(G_{F,S}, \bT_{\epsilon}^\cyc)$ and $H^{2}(G_{F,S}, \bT_{\epsilon}^\cyc)$ are torsion. 
\end{lemma}
\begin{proof}
\item[i)] 
Since $H^{0}(G_{F_{v}}, \bT_{\epsilon}^\cyc) = 0$ for any $v \in S_{\infty}(F)$, this portion follows from the global Euler--Poincar\'e characteristic formula \cite[Theorem~7.8.6]{Nek}
\[
\chi({\bf R}\Gamma_{\rm c}(G_{F,S},\bT_{\epsilon}^\cyc)) = \sum_{v \in S_\infty(F)}0\,.
\]

\item[ii)]   This part is an immediate consequence of \cite[Lemma~B.3.2]{rubin00}.

\item[iii)] Thanks to (i) and (ii), it is enough to check that the $\Lambda$-module $H^{1}(G_{F,S}, \bT_{\epsilon}^\cyc)$ is torsion. Since we have $H^{1}(G_{F_{v}}, \bT_{\epsilon}^\cyc) = 0$ for any prime $v \in S_{\rm ram}(K/F)$, there exists an exact sequence 
\[
0 \longrightarrow \varprojlim_{n>0} \cO_{K_{n}}^{\times, \epsilon_{K/F}} \longrightarrow H^{1}(G_{F,S}, \bT_{\epsilon}^\cyc) \longrightarrow \bigoplus_{v \in S_{p}(F)} H^{1}_{/{\rm f}}(G_{F_{v}}, \bT_{\epsilon}^\cyc)\,. 
\]
In view of  Lemma \ref{lemma:local-isom-torsion}, it suffices to prove that $ \varprojlim_{n>0} \cO_{K_{n}}^{\times, \epsilon_{K/F}} $ is a torsion $\LL$-module. Since $K_{n}$ is a CM field, we have $\cO_{K_{n}}^{\times, \epsilon_{K/F}} = \mu_{p^{\infty}}(K_{n})$. Thence, $\varprojlim_{n>0} \cO_{K_{n}}^{\times, \epsilon_{K/F}} = \varprojlim_{n>0}\mu_{p^{\infty}}(K_{n})$, which is manifestly $\LL$-torsion. 
\end{proof}

\begin{definition}
We let $T_{p}(\mu_{p^{\infty}}(K_{\rm cyc}))$ denote the $p$-adic Tate module of $\mu_{p^{\infty}}(K_{\rm cyc})$. 
\end{definition}

The proof of Lemma~\ref{lemma:coh-beta}(iii) shows that we have a natural identification $T_{p}(\mu_{p^{\infty}}(K_{\rm cyc}))=\varprojlim_{n>0} \cO_{K_{n}}^{\times, \epsilon_{K/F}}$.

\begin{corollary}\label{corollary:isom-torsion}
The canonical map 
\[
T_{p}(\mu_{p^{\infty}}(K_{\rm cyc})) = \varprojlim_{n>0} \cO_{K_{n}}^{\times, \epsilon_{K/F}} \longrightarrow H^{1}(G_{F,S}, \bT_{\epsilon}^\cyc)
\]
induced by Kummer theory is an isomorphism. 
\end{corollary}
\begin{proof}
It is well-known that the homomorphism 
\[
 T_{p}(\mu_{p^{\infty}}(K_{\rm cyc}))   = (\bZ_{p}(1) \otimes \epsilon_{K/F})^{G_{F_{\rm cyc}}}  \longrightarrow H^{1}(G_{F,S}, \bT_{\epsilon}^\cyc)_{\rm tors} 
\]
is an isomorphism (c.f. \cite[Corollary 9.1.7]{Nek}). Our assertion follows from Lemma \ref{lemma:coh-beta}(iii). 
\end{proof}

Let $\cH(K_{\rm cyc})$ denote the maximal unramified pro-$p$ abelian extension of $K_{\rm cyc}$ and put 
\[
X_{K, \infty}^{-} := \Gal(\cH(K_{\rm cyc})/K_{\rm cyc})^{\epsilon_{K/F}}. 
\]

\begin{lemma}\label{lemma:char-sel}
We have 
\[
{\rm char}_{\Lambda}(X_{K, \infty}^{-}) = 
{\rm char}_{\Lambda}(H^{2}(G_{F,S}, \bT_{\epsilon}^\cyc))
\]
as non-zero ideals of $\LL$.
\end{lemma}
\begin{proof}
Since we have $ H^{1}(G_{F_{v}}, \bT_{\epsilon}^\cyc)=0=H^{2}(G_{F_{v}}, \bT_{\epsilon}^\cyc) $ for any prime $v \in S_{\rm ram}(K/F)$, we have the following exact sequence of torsion $\Lambda$-modules by global duality: 
\begin{align*}
0 \longrightarrow T_{p}(\mu_{p^{\infty}}(K_{\rm cyc})) \longrightarrow & \,H^{1}(G_{F,S}, \bT_{\epsilon}^\cyc) \longrightarrow \bigoplus_{v \in S_{p}(F)} H^{1}_{/{\rm f}}(G_{F_{v}}, \bT_{\epsilon}^\cyc) 
\\
&\longrightarrow X_{K, \infty}^{-} \longrightarrow H^{2}(G_{F,S}, \bT_{\epsilon}^\cyc) \longrightarrow \bigoplus_{v \in S_{p}(F)} H^{2}(G_{F_{v}}, \bT_{\epsilon}^\cyc) \longrightarrow 0. 
\end{align*}
We therefore infer that 
$$\frac{{\rm char}_{\Lambda}( X_{K, \infty}^{-})}{{\rm char}_{\Lambda}(H^{2}(G_{F,S}, \bT_{\epsilon}^\cyc))}=\frac{{\rm char}_{\Lambda}( \oplus_{v \in S_{p}(F)} H^{2}(G_{F_{v}}, \bT_{\epsilon}^\cyc) )}{{\rm char}_{\Lambda}(\oplus_{v \in S_{p}(F)} H^{1}_{/{\rm f}}(G_{F_{v}}, \bT_{\epsilon}^\cyc) )}\times \frac{{\rm char}_{\Lambda}(H^{1}(G_{F,S}, \bT_{\epsilon}^\cyc))}{{\rm char}_{\Lambda}( T_{p}(\mu_{p^{\infty}}(K_{\rm cyc})))}\,.$$
The asserted equality follows from Lemma~\ref{lemma:local-isom-torsion} and Corollary~\ref{corollary:isom-torsion}. 
\end{proof}
\begin{definition}
\label{defn_alg_DR_p-adicL}
We define the algebraic Deligne--Ribet $p$-adic $L$-function 
\[
L_{\epsilon_{K/F}\omega_F}^{\rm alg, DR} \in {\rm Frac}(\LL)^\times/\LL^\times
\]
attached to the totally even character $\epsilon_{K/F}\omega_F$, as the image of a generator of the cyclic $\LL$-module 
\[
\Ann_{\Lambda}(T_{p}(\mu_{p^{\infty}}(K_{\rm cyc})))^{-1} {\rm char}_{\Lambda}(X_{K, \infty}^{-}) 
= 
\det({\bf R}\Gamma(G_{F,S}, \bT_{\epsilon}^\cyc)). 
\]
Similarly, we define the $p$-adic Dedekind zeta-function 
$\zeta_{{\rm alg}} \in {\rm Frac}(\LL)^\times/\LL^\times$ as the image of a generator of the fractional ideal \[
\sA_{\rm cyc}^{-1}\,{\rm char}_{\Lambda}(H^{2}_{\rm c}(G_{F,S}, \bT_{\mathds 1}^\cyc))  = \det({\bf R}\Gamma_{\rm c}(G_{F,S},\bT_{\mathds 1}^\cyc), 
\]
where $\sA_{\rm cyc}\subset \Lambda = \Lambda_{\bZ_{p}}(\Gamma_{\rm cyc})$  is the augmentation ideal. 
\end{definition}

\begin{remark}
Let $M_{F, S}$ denote the maximal pro-$p$ abelian extension of $F_{\rm cyc}$ which is unramified outside $S$. It follows from global duality that $H^{2}_{\rm c}(G_{F,S}, \bT_{\mathds 1}^\cyc) \cong \Gal(M_{F,S}/F_{\rm cyc})$. 
The classical Iwasawa main conjecture for $F$ asserts that $\Gal(M_{F,S}/F_{\rm cyc})$ is generated by 
$\sA_{\rm cyc} \cdot \zeta_p$, i.e.
\[
\zeta_{\mathrm{alg}} = \zeta_p \,\,\, \textrm{ in ${\rm Frac}(\LL)^\times/\LL^\times$}. 
\]
\end{remark}

Recall also the restriction $L_{p, \mathds{1}^{\rm alg}}\vert_{\Gamma_{\rm cyc}}$ of the algebraic Katz $p$-adic $L$-function to the cyclotomic $\ZZ_p$-tower of $K$.

\begin{theorem}[Algebraic Gross' Factorization for CM fields]
\label{thm_algebraic_gross_factorization} We have the following equality taking place in ${\rm Frac}(\LL)^\times/\LL^\times$:
\[
L_{p, \mathds{1}^{\rm alg}}\vert_{\Gamma_{\rm cyc}} = L_{\epsilon_{K/F}\omega_F}^{\rm alg, DR}\cdot
\zeta_{\rm alg}\,.
\]
\end{theorem}

\begin{proof}
Recall that $L_{p, \mathds{1}^{\rm alg}}\vert_{\Gamma_{\rm cyc}}$ is a generator of ${\rm Fitt}_{\Lambda}(\widetilde{H}^{2}_{\rm f}(G_{K, S_{K}}, \bT_{\mathds 1}^\cyc, \Delta_{\Sigma}))$. Note in addition that $\widetilde{H}^{1}_{\rm f}(G_{K, S_{K}}, \bT_{\mathds 1}^\cyc, \Delta_{\Sigma})=0$ thanks to Corollary~\ref{cor:H^1_f=0 and H^2_f is torsion}. Proposition \ref{prop:parf[1,2]} therefore implies that 
\[
\det(\widetilde{{\bf R}\Gamma}_{\mathrm{f}}(G_{K, S_{K}}, \bT_{\mathds 1}^\cyc, \Delta_{\Sigma})) = 
{\rm Fitt}_{\Lambda}(\widetilde{H}^{2}_{\rm f}(G_{K, S_{K}}, \bT_{\mathds 1}^\cyc, \Delta_{\Sigma})). 
\]
Moreover, Proposition~\ref{prop:decomp-gross} shows that 
\[
\det(\widetilde{{\bf R}\Gamma}_{\mathrm{f}}(G_{K, S_{K}}, \bT_{\mathds 1}^\cyc, \Delta_{\Sigma}))
= 
\det({\bf R}\Gamma(G_{F,S}, \bT_{\epsilon}^\cyc))
\det({\bf R}\Gamma_{\rm c}(G_{F,S},\bT_{\mathds 1}^\cyc)). 
\]
Our proof is complete once we plug in the definitions of $L_{p, \mathds{1}^{\rm alg}}$, $L_{\epsilon_{K/F}\omega_F}^{\rm alg, DR}$ and $\zeta_{\rm alg}$.
\end{proof}

\begin{proposition}\label{proposition:order-H2}
\item[i)] The Leopoldt conjecture for $F$ holds true if and only if ${\rm char}_\LL(H^{2}_{\rm c}(G_{F,S}, \bT_{\mathds 1}^\cyc)) \not\subset \sA_\cyc$.  
\item[ii)]We have ${\rm ord}_{\sA_\cyc}({\rm char}_{\Lambda}(H^{2}(G_{F,S}, \bT_{\epsilon}^\cyc))) \geq \# S_{p}(F)$, with equality if Gross' $p$-adic regulator given as in \cite[(3.8)]{FedererGross81} is non-vanishing.
\end{proposition}

\begin{remark}
Since ${\rm char}_{\Lambda}(X^-_{K, \infty}) = {\rm char}_{\Lambda}(H^{2}(G_{F,S}, \bT_{\epsilon}^\cyc))$, Proposition \ref{proposition:order-H2}(ii) follows from  \cite[Proposition 6.1(1)]{FedererGross81} and we have provided here a proof of Proposition~\ref{proposition:order-H2}(ii) for the reader's convenience.  We also remark that the corresponding analytic result is proved by Charollois and Dasgupta in \cite{CharolloisDasgupta14} without relying on the Iwasawa main conjecture. 
\end{remark}

\begin{proof}[Proof of Proposition~\ref{proposition:order-H2}]

\item[i)] 
Let $M_{F, S}$ denote the maximal pro-$p$ abelian extension of $F_{\rm cyc}$ which is unramified outside $S$. It follows by global duality that $H^{2}_{\rm c}(G_{F,S}, \bT_{\mathds 1}^\cyc) \cong \Gal(M_{F,S}/F_{\rm cyc})$. We therefore have an isomorphism 
\[
H^{2}_{\rm c}(G_{F,S}, \bT_{\mathds 1}^\cyc)/\sA_{\rm cyc}H^{2}_{\rm c}(G_{F,S}, \bT_{\mathds 1}^\cyc) \cong \Gal((M_{F,S} \cap F^{\rm ab})/F_{\rm cyc}). 
\]
The Galois group $\Gal((M_{F,S} \cap F^{\rm ab})/F_{\rm cyc})$ has finite cardinality if and only if Leopoldt's conjecture for $F$ holds true. This completes the proof of the first portion.

\item[ii)]  
Let $M_{K}^{\rm spl}$ denote  the maximal pro-$p$ abelian unramified extension of $K_{\rm cyc}$ such that any prime of $K_{\rm cyc}$ above $p$ splits completely  in $M_{K}^{\rm spl}$. 
Since we have $H^{2}(G_{F_{v}}, \bT_{\epsilon}^\cyc) = 0$ for any prime $v \in S \setminus S_{p}(F)$, global duality gives rise to the exact sequence
\[
0 \longrightarrow \Gal(M_{K}^{\rm spl}/K_{\rm cyc})^{\epsilon_{K/F}} \longrightarrow H^{2}(G_{F,S}, \bT_{\epsilon}^\cyc) \longrightarrow \bigoplus_{v \in S_{p}(F)}H^{2}(G_{F_{v}}, \bT_{\epsilon}^\cyc) \longrightarrow 0.  
\]
For a prime $v \in S_{p}(F)$, since $\epsilon_{K/F}(G_{F_{v}}) = 1$, local Tate duality yields  an identification
\[
H^{2}(G_{F_{v}}, \bT_{\epsilon}^\cyc) \cong \bZ_{p}[[\Gamma_{\rm cyc}/D_{v}]], 
\]
where $D_{v}\subset \Gamma_{\rm cyc}$ is the decomposition subgroup at $v$. 
Since $F_{\rm cyc}/F$ is the cyclotomic $\bZ_{p}$-extension, $D_{v}$ is non-trivial. 
This fact implies that ${\rm ord}_{\sA_\cyc}({\rm char}_{\Lambda}(H^{2}(G_{F_{v}}, \bT_{\epsilon}^\cyc))) =1$ and we obtain 
\[
{\rm ord}_{\sA_\cyc}({\rm char}_{\Lambda}(H^{2}(G_{F,S}, \bT_{\epsilon}^\cyc))) \geq \# S_{p}(F). 
\]
Moreover, Sinnott in \cite[\S6]{FedererGross81} has proved that Gross' $p$-adic regulator does not vanish 
 if and only if  ${\rm ord}_{\sA_\cyc}({\rm char}_{\Lambda}(\Gal(M_{K}^{\rm spl}/K_{\rm cyc})^{\epsilon_{K/F}})) = 0$ (see the 2nd and 3rd line on page 457 of op. cit.). We note that $h \in \Lambda$ defined in the third paragraph of page 454 of op. cit. is a generator of the characteristic ideal of $\Gal(M_{K}^{\rm spl}/K_{\rm cyc})^{\epsilon_{K/F}}$.
This completes the proof of our proposition. 
\end{proof}

\begin{remark}
As Federer and Gross remark within the paragraph following the proof of Proposition 1.7 in \cite{FedererGross81}, Gross' $p$-adic regulator does not vanish in the following scenarios:
\item[a)] $K/\bQ$ is abelian. 
\item[b)] $\# S_{p}(F) = 1$. 
\end{remark}

The following corollary is an immediate consequence of Theorem~\ref{thm_algebraic_gross_factorization} and Proposition~\ref{proposition:order-H2}.

\begin{corollary}\label{corollary:order-trivial-character}
We have 
\[
{\rm ord}_{\sA_{\rm cyc}}\left(L_{p, \mathds{1}^{\rm alg}}\vert_{\Gamma_{\rm cyc}} \right)  \geq \#S_{p}(F) - 1, 
\]
with equality under the validity of Leopoldt's conjecture for $F$ and the non-vanishing of Gross' $p$-adic regulator. 
\end{corollary}

Recall that if Leopoldt's conjecture for $F$ is valid, then ${\bf G}' = \Gamma_{\rm cyc}$ and $\Gamma_\infty^\circ = \Gamma_\infty$. 

\begin{corollary}\label{corollary_}
Assuming the truth of  Leopoldt's conjecture for $F$ and the non-vanishing of Gross' $p$-adic regulator, we have 
\[
{\rm ord}_{\sA}(\widetilde{\mathscr{D}}^{\rm alg} \vert_{\Gamma_\infty}) = 
{\rm ord}_{\sA}(  L_{p, \psi^{\rm ad}}^{\rm alg} \vert_{\Gamma_\infty} ) + \#S_{p}(F) -1. 
\]
Here, ${\rm ord}_{\sA}(f)$ denotes the smallest integer $t \geq 0$ satisfying $f \in \sA^t \setminus \sA^{t+1}$ for any $f \in \Lambda_{\cO}(\Gamma_\infty)$.
\end{corollary}

\begin{proof}
This is an immediate consequence of the factorization Theorem~\ref{thm:iwasama main conj rankin-selberg}(i) and Corollary~\ref{corollary:order-trivial-character}. 
\end{proof}


\subsection{Selmer complexes for the traceless adjoint}
\label{subsec_selmer_traceless_adj}
We resume our study of Rankin--Selberg $p$-adic $L$-functions. Our main objective in this subsection is to prove Theorem~\ref{theorem_adjoint}, which is the algebraic counterpart of the conjectured factorization \eqref{eqn_Dasgupta_CM} for the adjoint $p$-adic $L$-function. Using this result, we also prove Corollary~\ref{cor_exceptional_zeros_RankinSelberg_Hida},  which constitutes the algebraic variant of the containments in Question~\ref{conj_exceptional_zeros_RankinSelberg_Hida} (that are presently unproved).

We retain our notation in \S\ref{sec:Factorisation of Selmer groups for symmetric Rankin--Selberg products}. As in \S\ref{sec_agfCM}, we set
\[
\bT_{\mathds 1}^\cyc := \bZ_{p}(1) \otimes_{\bZ_{p}} \Lambda_{\bZ_{p}}(\Gamma_{\rm cyc})^{\iota} \,\,\,\,\, \textrm{ and } \, \,\,\, \bT_{\epsilon}^\cyc := (\bZ_{p}(1) \otimes \epsilon_{K/F}) \otimes_{\bZ_{p}} \Lambda_{\bZ_{p}}(\Gamma_{\rm cyc})^{\iota}\,.
\] 
Since $(\Gamma_\infty^\circ)^{c=1} = \Gamma_{\rm cyc}$, we have a natural (transfer) map 
\be\label{defn_cyclo_verchiubung}
{\rm ver}_{\rm cyc} \colon \Lambda_\cO(\Gamma_{\rm cyc}) \longrightarrow \Lambda_\cO(\Gamma_{\infty}^\circ)
\ee
that one defines in a manner similar to \eqref{eqn_cyc_transfer_map}.
We then have a decomposition 
\[
\bT_{1} \otimes_{\Lambda_{\cO}(W_K)} \Lambda_\cO(\Gamma_{\infty}^\circ) =\left( \bT_{\mathds 1}^\cyc \otimes_{{\rm ver}_{\rm cyc}} \Lambda_\cO(\Gamma_{\infty}^\circ) \right)\oplus \left(\bT_{\epsilon}^\cyc \otimes_{{\rm ver}_{\rm cyc}} \Lambda_\cO(\Gamma_{\infty}^\circ) \right). 
\]
For notational simplicity, we put 
\[
\bT_{?, {\Gamma_\infty^\circ}} := \bT_{?} \otimes_{{\rm ver}_{\rm cyc}} \Lambda_{\cO}({\Gamma_\infty^\circ})
\]
for $? \in \{\mathds{1}, \epsilon\}$. 
We finally set 
\[
{\bT_{0} := M({\rm ad}^{0}\Theta_{\Sigma,\psi})^\iota \otimes_{\Lambda_{\cO}(W_K)} \Lambda_{\cO}({\Gamma_\infty^\circ}) }
\]
(the traceless adjoint) and observe that $ \bT_{0}=  \bT_{\epsilon, {\Gamma_\infty^\circ}} \oplus (\bT_{2} {\otimes_{\Lambda_{\cO}(W_K)} \Lambda_{\cO}({\Gamma_\infty^\circ})})$.
\begin{definition}\
\item[i)] For each prime $v \in S_p(F)$, we define the Greenberg submodule $F^+\TT_{0}\subset\TT_{0}$ by setting
\begin{align*}
{F^+\TT_{0} := {\rm im}\left(F^{+}\bT \longrightarrow (\TT \otimes_{\Lambda_{\cO}(W_K)} \Lambda_{\cO}( \Gamma_\infty^\circ))/\bT_{\mathds{1}, \Gamma_\infty^\circ} = \bT_{0} \right)}, 
\end{align*}
where $F^{+}\bT \subset \bT $ is given as in Definition~\ref{defn_Greenberg_Data}.
\item[ii)] We put $S := S_{\rm ram}(\psi) \cup S_{\rm ram}(K/F) \cup S_p(F)$ and define the local conditions $\Delta_{0}$ on the complex $C^{\bullet}(G_{F,S}, \TT_{0})$ with the choices
\[
U_{v}^{+} := 
\begin{cases}
C^{\bullet}(G_{F_{v}},F^+\TT_{0}) & \text{if} \ v \in S_p(F), 
\\
C^{\bullet}_{\rm ur}(G_{F_{v}}, \TT_{0}) & \text{if} \ v \not\in S \setminus S_p(F). 
\end{cases}
\]
\item[iii)] For any ideal $I$ of $\Lambda_{\cO}({\Gamma_\infty^\circ})$, 
we similarly define $F^+(\TT_{0}/I\TT_{0})\subset \TT_{0}/I\TT_{0}$ and 
local conditions $\Delta_{0}$ on the complex $C^{\bullet}(G_{F,S}, \TT_{0}/I\TT_{0})$. 
\end{definition}

\begin{proposition}\label{proposition:decomp-adjoint}
The decomposition $\bT_{0} =  \bT_{\epsilon, {\Gamma_\infty^\circ}} \oplus (\bT_{2} {\otimes_{\Lambda_{\cO}(W_K)} \Lambda_{\cO}({\Gamma_\infty^\circ})})$ induces 
\[
\widetilde{{\bf R}\Gamma}_{\rm f}(G_{F,S}, \bT_{0}, \Delta_{0})  =  {\bf R}{\Gamma}(G_{F,S}, \bT_{\epsilon, {\Gamma_\infty^\circ}}) \oplus  \widetilde{{\bf R}\Gamma}_{\rm f}(G_{K, S_K}, \TT_{K, 2} {\otimes_{\Lambda_{\cO}(W_K)} \Lambda_{\cO}({\Gamma_\infty^\circ})} , \Delta_{\Sigma}). 
\]
\end{proposition}
\begin{proof}
We have by definition of the Greenberg submodule
\[
F^{+}\bT_{0} = \bT_{\epsilon, {\Gamma_\infty^\circ}} \oplus {\rm im}\left(F^{+}\bT \longrightarrow \bT/\bT_{1} {\otimes_{\Lambda_{\cO}(W_K)} \Lambda_{\cO}({\Gamma_\infty^\circ})} = \bT_{2} {\otimes_{\Lambda_{\cO}(W_K)} \Lambda_{\cO}({\Gamma_\infty^\circ})} \right). 
\]
Lemma \ref{lem:decomposition local condition at p} shows for each prime $v \in S_{p}(F)$ that we have a natural decomposition 
\[
{\bf R}\Gamma(G_{F_{v}}, F^{+}\TT_{0}) = 
{\bf R}\Gamma(G_{F_{v}}, \TT_{\epsilon, {\Gamma_\infty^\circ}}) \oplus {\bf R}\Gamma(G_{K_{\tilde{v}^{c}}}, \TT_{K, 2} {\otimes_{\Lambda_{\cO}(W_K)} \Lambda_{\cO}({\Gamma_\infty^\circ})})
\]
where $\widetilde{v}\in\Sigma$ is the unique prime above $v$. The proof follows once we recall Definition~\ref{def:extended selmer} where we have introduced the local condition $\Delta_\Sigma$ on the complex $C^\bullet(G_{K,S_K}, \TT_{K, 2} {\otimes_{\Lambda_{\cO}(W_K)} \Lambda_{\cO}({\Gamma_\infty^\circ})})$.
\end{proof}

\begin{corollary}\label{corollary:str:adjoint}\ 
    \item[i)] For any ideal $I$ of $\Lambda_{\cO}({\Gamma_\infty^\circ})$, there is a natural isomorphism 
\begin{align*}
    \widetilde{{\bf R}\Gamma}_{\rm f}(G_{F,S}, \bT_{0}, \Delta_{0}) \otimes^{\bL}_{\Lambda_{\cO}({\Gamma_\infty^\circ})} \Lambda_{\cO}({\Gamma_\infty^\circ})/I \stackrel{\sim}{\longrightarrow} \widetilde{{\bf R}\Gamma}_{\rm f}(G_{F,S}, \bT_{0}/I\bT_{0}, \Delta_{0})
\end{align*}
as well as a natural isomorphism 
\begin{align*}
\widetilde{{\bf R}\Gamma}_{\rm f}(G_{F,S}, \bT_{0}, \Delta_{0}) \otimes^{\bL}_{\Lambda_{\cO}({\Gamma_\infty^\circ})} \Lambda_{\bZ_{p}^{\rm ur}}(\Gamma_\infty^\circ) \stackrel{\sim}{\longrightarrow} \widetilde{{\bf R}\Gamma}_{\rm f}(G_{F,S}, \bT_{0} \otimes_{\Lambda_{\cO}(\Gamma_\infty^\circ)} \Lambda_{\bZ_{p}^{\rm ur}}(W_{K}), \Delta).
\end{align*}

\item[ii)] $\chi(\widetilde{{\bf R}\Gamma}_{\rm f}(G_{F,S}, \TT_{0}, \Delta_{0})) = 0$. 

\item[iii)] We have the following natural identifications:
\begin{align*}
\widetilde{H}_{\rm f}^{1}(G_{F,S}, \bT_{0}, \Delta_{0}) &= H^{1}(G_{F,S}, \bT_{\epsilon}^\cyc) \otimes_{{\rm ver}_{\rm cyc}} \Lambda_{\cO}(\Gamma_\infty^\circ)\\  
&= T_{p}(\mu_{p^{\infty}}(K_{\rm cyc})) \otimes_{{\rm ver}_{\rm cyc}} \Lambda_{\cO}(\Gamma_\infty^\circ)\,. 
\end{align*}

\item[iv)] The $\Lambda_{\cO}(\Gamma_\infty^\circ)$-module $\widetilde{H}^{2}_{\rm f}(G_{F,S}, \TT_{0}, \Delta_{0})$ is torsion. 
\end{corollary}

\begin{proof}
By Lemma \ref{lemma:scalar-extension} and Proposition \ref{proposition:decomp-adjoint}, we have a decomposition 
\begin{align}\begin{split}
\label{eqn_cor532_1}\widetilde{{\bf R}\Gamma}_{\rm f}(G_{F,S}, \bT_{0}, \Delta_{0})  &=  {\bf R}{\Gamma}(G_{F,S}, \bT_{\epsilon, \Gamma_\infty^\circ}) \oplus  \widetilde{{\bf R}\Gamma}_{\rm f}(G_{K, S_K}, \TT_{K, 2} {\otimes \Lambda_{\cO}(\Gamma_\infty^\circ)}, \Delta_{\Sigma})
\\
&=  \left({\bf R}{\Gamma}(G_{F,S}, \bT_{\epsilon}^\cyc) \otimes^{\bL}_{{\rm ver}_{\rm cyc}} \Lambda_{\cO}(\Gamma_\infty^\circ) \right) \oplus  \widetilde{{\bf R}\Gamma}_{\rm f}(G_{K, S_K}, \TT_{K, 2} {\otimes \Lambda_{\cO}(\Gamma_\infty^\circ)}, \Delta_{\Sigma}). 
\end{split}
\end{align}

\item[i)] The first assertion in this portion follows from \eqref{eqn_cor532_1} combined with \eqref{eq:base-change}, whereas the second from \eqref{eqn_cor532_1} combined with \eqref{eq:scalar-extension}. 
\item[ii)] This global Euler--Poincar\'e characteristic formula follows from  \eqref{eqn_cor532_1} combined with Lemma~\ref{lemma:Euler characteristic = 0} and Lemma~\ref{lemma:coh-beta}. 

\item[iii)] It follows from Corollary~\ref{cor:H^1_f=0 and H^2_f is torsion}(i) that $\widetilde{H}^{1}_{\rm f}(G_{K,S_{K}}, \TT_{2}, \Delta_{\Sigma}) = 0$. This combined with \eqref{eqn_cor532_1} proves the first asserted equality. The second equality
\[
H^{1}(G_{F,S}, \bT_{\epsilon}^\cyc) \otimes_{{\rm ver}_{\rm cyc}} \Lambda_{\cO}(\Gamma_\infty^\circ)  = T_{p}(\mu_{p^{\infty}}(K_{\rm cyc})) \otimes_{{\rm ver}_{\rm cyc}} \Lambda_{\cO}(\Gamma_\infty^\circ) 
\] 
is Corollary \ref{corollary:isom-torsion}. 

\item[iv)] This assertion follows on combining \eqref{eqn_cor532_1} with Corollary~\ref{cor:H^1_f=0 and H^2_f is torsion}(ii) and Lemma~\ref{lemma:coh-beta}(iii). 
\end{proof}

\begin{definition}
\label{defn_algebraic_adjoint_padic_L_function}
The algebraic adjoint $p$-adic $L$-function 
\[
\widetilde{\mathscr D}^{\rm alg}_{{\rm ad}^0\Theta_{\Sigma,\psi}}  \in {\rm Frac}(\LL_\cO(\Gamma_\infty^\circ))^\times/\LL_\cO(\Gamma_\infty^\circ)^\times
\]
is the image of a generator of the cyclic $\LL_\cO(\Gamma_\infty^\circ)$-module
\[
{\rm ver}_{\cyc}\left(\Ann_{\Lambda_\cyc}(T_{p}(\mu_{p^{\infty}}(K_{\rm cyc})))\right)^{-1}{\rm char}_{\Lambda_{\cO}(\Gamma_\infty^\circ)}(\widetilde{H}^{2}_{\rm f}(G_{F,S}, \bT_{0}, \Delta_{0})).
\]
\end{definition}

\begin{theorem}\label{theorem_adjoint}
\item[i)] We have the following factorization taking place in ${\rm Frac}(\LL_\cO(\Gamma_\infty^\circ))^\times/\LL_\cO(\Gamma_\infty^\circ)^\times$:
\[
{ \widetilde{\mathscr D}^{\rm alg}_{{\rm ad}^0\Theta_{\Sigma,\psi}} } = 
{\rm ver}_{\cyc}\left( L_{\epsilon_{K/F}\omega_F}^{\rm alg, DR}\right)
\cdot L_{p, \psi^{\rm ad}}^{\rm alg} \,.
\]
\item[ii)] Assuming the non-vanishing of Gross' $p$-adic regulator, we have 
\[
{\rm ord}_{\sA^\circ}(\widetilde{\mathscr D}^{\rm alg}_{{\rm ad}^0\Theta_{\Sigma,\psi}}) = 
{\rm ord}_{\sA^\circ}(L_{p, \psi^{\rm ad}}^{\rm alg} ) + \#S_{p}(F), 
\]
where $\mathscr{A}^\circ \subset \LL_{\cO}(\Gamma_\infty^\circ)$ is the augmentation ideal. 
\end{theorem}

\begin{proof}
The first assertion follows from Proposition~\ref{proposition:decomp-adjoint}, whereas the second from the first part combined with Proposition~\ref{proposition:order-H2}(ii).  
\end{proof}
\begin{remark}
If the Leopoldt conjecture is valid for $F$, we then have $\Gamma_\infty^\circ=\Gamma_\infty$ and the the equality in Theorem~\ref{theorem_adjoint}(ii) in fact takes place in ${\rm Frac}(\LL_\cO( \Gamma_\infty))^\times/\LL_\cO( \Gamma_\infty)^\times$.

Without knowing the truth of Leopoldt's conjecture, the algebraic variant of the expected identity \eqref{eqn_Dasgupta_CM} should indeed take place in ${\rm Frac}(\LL_\cO( \Gamma_\infty^\circ))^\times/\LL_\cO( \Gamma_\infty^\circ)^\times$, rather than ${\rm Frac}(\LL_\cO( \Gamma_\infty))^\times/\LL_\cO( \Gamma_\infty)^\times$. This is because in \eqref{eqn_Dasgupta_CM}, there is a single formal variable (that we denote by $s$) complementing the formal variable $\kappa$ that accounts for anticyclotomic variation, and the variable $s$ parametrizes the cyclotomic variation alone. In other words, the pair $(\kappa,s)$ in \eqref{eqn_Dasgupta_CM} parametrizes variation in the character group $\widehat{\Gamma_\infty^\circ}$ rather than $\widehat{\Gamma_\infty}$.
\end{remark}

Theorem~\ref{thm_HT_factorization}, Theorem~\ref{thm:iwasama main conj rankin-selberg}, Theorem~\ref{theorem_adjoint}, and Theorem~\ref{thm_exceptional_zero_conj_OK} combined together yields the following result.

\begin{corollary}
\label{cor_exceptional_zeros_RankinSelberg_Hida}
If Gross' $p$-adic regulator does not vanish, the $\Sigma$-Leopoldt conjecture for $L/K$  and the Iwasawa Main Conjecture~\ref{conj:iwasawa main conj} for the CM field $K$ hold true, then 
\[
\widetilde{\mathscr D}_{{\rm ad}^0\Theta_{\Sigma,\psi}} \in  (\sA^\circ)^{e+ \#S_p(F)} \setminus (\sA^\circ)^{e + \#S_p(F) + 1}. 
\]
If in addition Leopoldt's conjecture for $F$ holds true, then we have 
\[
{ \widetilde{\mathscr{D}}^{\rm Hida}}\in (\sA^\circ)^{e+ \#S_p(F)-1} \setminus  (\sA^\circ)^{e+ \#S_p(F)}. 
\]
\end{corollary}

Note that under the running assumptions of Corollary~\ref{cor_exceptional_zeros_RankinSelberg_Hida}, we have ${\bf G}'=\Gamma_\cyc$ and $\Gamma_\infty^\circ=\Gamma_\infty$. In particular, we have 
 $$\widetilde{\mathscr{D}}^{\rm Hida}\,,\,\widetilde{\mathscr D}_{{\rm ad}^0\Theta_{\Sigma,\psi}}\in \LL_{\cO}(\Gamma_\infty)$$
 and $\mathscr{A}^\circ=\mathscr{A}$ is the augmentation ideal of $\LL_{\cO}(\Gamma_\infty)$. We may restate the containments in Corollary~\ref{cor_exceptional_zeros_RankinSelberg_Hida} as
$$\widetilde{\mathscr D}_{{\rm ad}^0\Theta_{\Sigma,\psi}} \in   \sA^{e+ \#S_p(F)} \setminus \sA^{e + \#S_p(F) + 1}$$
$${ \widetilde{\mathscr{D}}^{\rm Hida}}\in \sA^{e+ \#S_p(F)-1} \setminus \sA^{e+ \#S_p(F)}$$
under the hypotheses of this corollary.


\appendix
\section{Genuine algebraic $p$-adic $L$-functions and their non-critical specializations}
\label{sec_appendix}
The purpose of this appendix is to expand the discussion in the opening paragraph of \S\ref{subsubsec_intro_3} to verify that $\widetilde{\mathscr{D}}^{\rm alg}$ can be recovered as a specialization of an algebraic $p$-adic $L$-function over a higher-dimensional weight space (c.f. Corollary~\ref{cor_appendix_main} below). 

Throughout this appendix, we retain the notation of \S\ref{sec:galois repn} and \S\ref{sec:Factorisation of Selmer groups for symmetric Rankin--Selberg products}. 
Recall from Remark~\ref{rem_no_structure_onWK} that we have a canonical embedding 
${\bf G}' \hookrightarrow W_K$, which induces the ring homomorphism 
\begin{equation}
\label{eqn_appendix_1}
    \LL_{\cO}({\bf G}') \longrightarrow \LL_{\cO}(W_K)\,. 
\end{equation}
We remind the reader that ${\bf G}'$ is the Galois group of the maximal $\bZ_p$-power extension of $F$. If the Leopoldt conjecture is valid for the totally real field $F$, then ${\bf G}' = \Gamma_{\rm cyc}$.  

For  notational simplicity, let us put 
\begin{align*}
\cR &:= \Lambda_{\cO}(W_K) \widehat{\otimes}_{\cO} \Lambda_{\cO}(W_K)
\\
    M(\mathrm{Ad}\Theta_{\Sigma,\psi})^\iota &:= (M_{\Psi} \widehat{\otimes}_{\cO}M_{\Psi^{-1}} )\otimes_{\Lambda_{\cO}({\bf G}')} \Lambda_{\cO}({\bf G}')^{\iota}, 
    \\
    \bT_{\rm Ad} &:= M(\mathrm{Ad}\Theta_{\Sigma,\psi})^\iota(1),
\end{align*}
and note that the $\cR$-module $\bT_{\rm Ad}$ is free of rank $4$. Let us denote by
\be\label{eqn_pi_projection_appendix}
\pi \colon\,\, \cR \longrightarrow \Lambda_{\cO}(W_K) \otimes_{\Lambda_{\cO}(W_K)} \Lambda_{\cO}(W_K) = \Lambda_{\cO}(W_K)
\ee
the natural surjection. Recalling that 
\[
\bT = (M_{\Psi} \widehat{\otimes}_{\LL_{\cO}(W_K)}M_{\Psi^{-1}} )\otimes_{\Lambda_{\cO}({\bf G}')} \Lambda_{\cO}({\bf G}')^{\iota}(1)\,, 
\]
we have a natural isomorphism
\begin{equation}
\label{eqn_appendix_2}
    \TT_{\rm Ad}/\ker(\pi)\TT_{\rm Ad} \xrightarrow{\sim} \TT\,.
\end{equation}

We recall that $S := S_{\rm ram}(\psi) \cup S_{\rm ram}(K/F) \cup S_p(F)$. The following definition should be compared to Definition~\ref{defn_Greenberg_Data}. 

\begin{definition}
For each prime $v \in S_p(F)$, we define the $G_{F_{v}}$-submodule $F^+\TT_{\rm Ad}$ of $\TT_{\rm Ad}$ on setting
\begin{align*}
F^+\TT_{\rm Ad} &:= ({\Psi} \widehat{\otimes}_{\cO}M_{\Psi^{-1}} )\otimes_{\Lambda_{\cO}({\bf G}')} \Lambda_{\cO}({\bf G}')^{\iota}(1). 
\end{align*}
We  define local conditions $\Delta_{\rm Ad}$ for $C^{\bullet}(G_{F,S}, \TT_{\rm Ad})$ given by the choices 
\[
U_{v}^{+} := 
\begin{cases}
C^{\bullet}(G_{F_{v}},F^+\TT_{\rm Ad}) & \text{if} \ v \in S_p(F), 
\\
C^{\bullet}_{\rm ur}(G_{F_{v}}, \TT_{\rm Ad}) & \text{if} \ v \not\in S \setminus S_p(F). 
\end{cases}
\]
For any ideal $I$ of $\cR$, 
we  also define $F^+(\TT_{\rm Ad}/I\TT_{\rm Ad})$ and 
local conditions $\Delta_{\rm Ad}$ for $C^{\bullet}(G_{F,S}, \TT_{\rm Ad}/I\TT_{\rm Ad})$ in an identical fashion. 
\end{definition}

\begin{remark}
Unlike $\TT$, the representation $\TT_{\rm Ad}$ admits a dense set of critical specializations. Furthermore, the image of $F^+\TT_{\rm Ad}$, under a critical specialization of $\TT_{\rm Ad}$ that is given by an arithmetic specialization $(\kappa,\kappa')$ of $\LL_{\cO}(W_K)\,\widehat\otimes\,\LL_{\cO}(W_K)$ verifying the conditions
\be\label{eqn_Hida_weight_restrictions_appendix}
n(\kappa)-n(\kappa') \geq t_F\,\,\,\,,\,\,\,\,\, v(\kappa')\geq v(\kappa)\,\,\,\,,\,\,\,\,\,  (m(\kappa)-m(\kappa'))t_F\geq n(\kappa')-n(\kappa)+2t_F,
\ee
coincides with the $p$-ordinary filtration. Here, the quantities $n(\kappa), v(\kappa)$ and $m(\kappa)$ are defined by treating $\kappa$ as a specialization of $\LL^{\rm n.o.}$ in accordance with Convention~\ref{convention_4_5} (similarly, for $\kappa'$), and these quantities together with $t_F$ are given as in \S\ref{subsec_4_1_generalities}.

The submodule $F^+\TT_{\rm Ad} \subset \TT_{\rm Ad}$ is characterized by this property. We note that \eqref{eqn_Hida_weight_restrictions_appendix} is the condition \eqref{eqn_Hida_weight_restrictions} that determines the interpolation range for Hida's $p$-adic Rankin--Selberg $L$-function. 
\end{remark}

\begin{lemma}\label{lem:ad-local}
The image of $F^+\TT_{\rm Ad}$ under the surjection $\bT_{\rm Ad} \xrightarrow{\pi} \bT$ coincides with $F^+\TT$. 
\end{lemma}
\begin{proof}
This lemma follows immediately from the definitions of $F^+\TT_{\rm Ad}$ and $F^+\TT$, since we have
\[
F^+\TT = ({\Psi} \widehat{\otimes}_{\Lambda_{\cO}(W_K)}M_{\Psi^{-1}} )\otimes_{\Lambda_{\cO}({\bf G}')} \Lambda_{\cO}({\bf G}')^{\iota}(1)\,. 
\]
\end{proof}

By the work of Pottharst in \cite[Theorem~1.12]{Jonathan13} (see also \cite{Nek}, Proposition 9.7.3(i)), we have the following natural base-change isomorphism in the derived category:
\begin{align}\label{eq:base-change-ad}
    \widetilde{{\bf R}\Gamma}_{\rm f}(G_{F,S}, \TT_{\rm Ad}, \Delta_{\rm Ad}) \otimes^{\bL}_{\cR} \cR/I \stackrel{\sim}{\longrightarrow} \widetilde{{\bf R}\Gamma}_{\rm f}(G_{F,S}, \TT_{\rm Ad}/I\TT_{\rm Ad}, \Delta_{\rm Ad}). 
\end{align}
By the definition of $\Delta$ (c.f. Definition \ref{defn_Greenberg_Data}(iii)) and Lemma \ref{lem:ad-local}, we have 
\be\label{eqn_appendix_A3plus5_1}
\widetilde{C}_{\rm f}^{\bullet}(G_{F,S},\TT_{\rm Ad}/\ker(\pi)\TT_{\rm Ad},\Delta_{\rm Ad}) = \widetilde{C}_{\rm f}^{\bullet}(G_{F,S},\TT,\Delta). 
\ee
The base-change isomorphism \eqref{eq:base-change-ad} together with \eqref{eqn_appendix_A3plus5_1} yields the following control theorem for Selmer complexes:
\begin{align}\label{eq:base-change-ad2}
    \widetilde{{\bf R}\Gamma}_{\rm f}(G_{F,S}, \TT_{\rm Ad}, \Delta_{\rm Ad}) \otimes^{\bL}_{\cR} \cR/\ker(\pi) \stackrel{\sim}{\longrightarrow} \widetilde{{\bf R}\Gamma}_{\rm f}(G_{F,S}, \TT, \Delta)\,. 
\end{align}

\begin{proposition}\label{prop:ad-str} The Selmer complex $\widetilde{{\bf R}\Gamma}_{\rm f}(G_{F,S}, \TT_{\rm Ad}, \Delta_{\rm Ad})$ has the following properties.
\item[i)] Its global Euler--Poincar\'e characteristic is zero: $\chi(\widetilde{{\bf R}\Gamma}_{\rm f}(G_{F,S}, \TT_{\rm Ad}, \Delta_{\rm Ad})) = 0$. 
\item[ii)] It can be represented by a perfect complex concentrated in degrees $1$ and $2$: $$\widetilde{{\bf R}\Gamma}_{\rm f}(G_{F,S}, \TT_{\rm Ad}, \Delta_{\rm Ad}) \in D_{\rm parf}^{[1,2]}(_{\cR}{\rm Mod}).$$ 
In particular, we have a canonical isomorphism 
\begin{align}\label{isom_base_change_Ad_H2}
\widetilde{H}^2_{\rm f}(G_{F,S}, \TT_{\rm Ad}, \Delta_{\rm Ad}) \otimes_{\cR} \cR/\ker(\pi) \stackrel{\sim}{\longrightarrow} \widetilde{H}^2_{\rm f}(G_{F,S}, \TT, \Delta)\,.  
\end{align}
\item[iii)] The $\cR$-module $\widetilde{H}^2_{\rm f}(G_{F,S}, \TT_{\rm Ad}, \Delta_{\rm Ad})$ is torsion and $\widetilde{H}^1_{\rm f}(G_{F,S}, \TT_{\rm Ad}, \Delta_{\rm Ad}) = 0$.
\end{proposition}
\begin{proof}
\item[i)] This portion follows from Corollary \ref{parf[1,2] and Euler--Poincare = 0}(iii) and \eqref{eq:base-change-ad2}, which tell us that
\begin{align*}
\chi(\widetilde{{\bf R}\Gamma}_{\rm f}(G_{F,S}, \TT_{\rm Ad}, \Delta_{\rm Ad})) & \stackrel{}{=}  
\chi(\widetilde{{\bf R}\Gamma}_{\rm f}(G_{F,S}, \TT_{\rm Ad}, \Delta_{\rm Ad}) \otimes^{\bL}_{\cR} \cR/\ker(\pi)) \\
&\stackrel{\eqref{eq:base-change-ad2}}{=}   
\chi(\widetilde{{\bf R}\Gamma}_{\rm f}(G_{F,S}, \TT, \Delta)) \stackrel{{\rm Cor.\, \ref{parf[1,2] and Euler--Poincare = 0}}}{=} \, 0
\end{align*}
as required.
\item[ii)] Combining Corollary \ref{parf[1,2] and Euler--Poincare = 0}(ii) and \eqref{eq:base-change-ad2}, we deduce that 
\[
\widetilde{H}_{\rm f}^0(G_{F,S}, \bT_{\rm Ad}/\fm_{\cR}\bT_{\rm Ad}, \Delta_{\rm Ad}) = 0=
\widetilde{H}_{\rm f}^3(G_{F,S}, \bT_{\rm Ad}/\fm_{\cR}\bT_{\rm Ad}, \Delta_{\rm Ad}), 
\]
where $\fm_{\cR}$ denotes the maximal ideal of $\cR$. The proof of Proposition \ref{prop:parf[1,2]} applies verbatim to verify our assertion (ii).  
\item[iii)] 
Notice that the two assertions in (iii) are equivalent to each other thanks to (ii): The $\cR$-module $\widetilde{H}^2_{\rm f}(G_{F,S}, \TT_{\rm Ad}, \Delta_{\rm Ad})$ is torsion if and only if $\widetilde{H}^1_{\rm f}(G_{F,S}, \TT_{\rm Ad}, \Delta_{\rm Ad}) = 0$. It remains to verify that the $\cR$-module $\widetilde{H}^2_{\rm f}(G_{F,S}, \TT_{\rm Ad}, \Delta_{\rm Ad})$ is torsion. We will do so by showing that its annihilator (as an $\cR$-module) is non-zero.

To see that, observe first that 
$$\mathrm{Fitt}_\cR(\widetilde{H}_{\rm f}^2(G_{F,S}, \bT_{\rm Ad}, \Delta_{\rm Ad})) \subset \mathrm{Ann}_\cR(\widetilde{H}_{\rm f}^2(G_{F,S}, \bT_{\rm Ad}, \Delta_{\rm Ad})).$$ 
It therefore suffices to show that $\mathrm{Fitt}_\cR(\widetilde{H}_{\rm f}^2(G_{F,S}, \bT_{\rm Ad}, \Delta_{\rm Ad})) \neq 0$. 
By Corollary~\ref{cor:delta-torison}, we have
$$\mathrm{Fitt}_{\Lambda_{\cO}(W_K)}(\widetilde{H}_{\rm f}^2(G_{F,S}, \bT, \Delta)) \neq 0.$$ 
It follows from (ii) that we have 
\begin{align}\label{eq:fitt-rel-ad}
    \pi(\mathrm{Fitt}_{\cR}(\widetilde{H}_{\rm f}^2(G_{F,S}, \bT_{\rm Ad}, \Delta_{\rm Ad}))) = \mathrm{Fitt}_{\Lambda_{\cO}(W_K)}(\widetilde{H}_{\rm f}^2(G_{F,S}, \bT, \Delta))\,. 
\end{align}
In particular, $\mathrm{Fitt}_{\Lambda_{\cO}(W_K)}(\widetilde{H}_{\rm f}^2(G_{F,S}, \bT, \Delta)) \neq 0$, as required. 
\end{proof}

\begin{corollary}
\label{corollary_to_Prop_A_4_appendix}
The $\cR$-module ${\rm Fitt}_{\cR}\left(\widetilde{H}^{2}_{\rm f}(G_{F, S}, \TT_{\rm Ad}, \Delta_{\rm Ad})\right)$ is non-trivial and cyclic. 
\end{corollary}
\begin{proof}
We verified the non-triviality of ${\rm Fitt}_{\cR}\left(\widetilde{H}^{2}_{\rm f}(G_{F, S}, \TT_{\rm Ad}, \Delta_{\rm Ad})\right)$ in the proof of Proposition \ref{prop:ad-str}(iii).  It follows from Proposition \ref{prop:ad-str}(ii) that the projective dimension of $\widetilde{H}^{2}_{\rm f}(G_{F, S}, \TT_{\rm Ad}, \Delta_{\rm Ad})$ as an $\cR$-module equals $1$. This completes the proof that ${\rm Fitt}_{\cR}\left(\widetilde{H}^{2}_{\rm f}(G_{F, S}, \TT_{\rm Ad}, \Delta_{\rm Ad})\right)$ is cyclic.
\end{proof}

\begin{definition}
\label{defn_algebraic_padic_Rankin_Selberg_ad}
We define  
\[
\widetilde{\mathscr{D}}^{\rm alg}_{\rm Ad} \in \mathrm{Frac}(\cR)^\times/\cR^\times
\]
as a generator of the cyclic $\cR$-module  ${\rm Fitt}_{\cR}\left(\widetilde{H}^{2}_{\rm f}(G_{F, S}, \TT_{\rm Ad}, \Delta_{\rm Ad})\right)$.
\end{definition}

\begin{corollary}
\label{cor_appendix_main}
$\pi(\widetilde{\mathscr{D}}^{\rm alg}_{\rm Ad}) = \widetilde{\mathscr{D}}^{\rm alg}$.
\end{corollary}

\begin{proof}
This is a restatement of \eqref{eq:fitt-rel-ad}, in view of the definitions of $\widetilde{\mathscr{D}}^{\rm alg}_{\rm Ad}$ and $\widetilde{\mathscr{D}}^{\rm alg}$.
\end{proof}


$\,$
\bibliographystyle{amsalpha}
\bibliography{references}

\end{document}